\documentclass[11pt,a4paper,reqno]{amsart}
\usepackage[utf8]{inputenc}

\usepackage{graphicx}
\usepackage[dvipsnames]{xcolor}   
\usepackage{amssymb}
\usepackage{amsmath}
\usepackage{mathtools}
\usepackage{tikz}
\usetikzlibrary{patterns}
\usepackage{enumitem}
\usepackage{mathrsfs}

\usepackage[colorlinks=true,linkcolor=blue,urlcolor=cyan,citecolor=cyan]{hyperref}

\hypersetup{linkcolor=Blue,urlcolor=cyan,citecolor=cyan}

\setlist[enumerate]{label=\emph{(\roman*)}}

\usepackage{environ}

\newtheorem{theorem}{Theorem}[section]
\newtheorem{corollary}[theorem]{Corollary}
\newtheorem{lemma}[theorem]{Lemma}
\newtheorem{proposition}[theorem]{Proposition}

\newtheorem{definition}[theorem]{Definition}
\theoremstyle{definition}
\newtheorem{remark}[theorem]{Remark}
\numberwithin{equation}{section}
\newcommand{\R}{\mathbb{R}}

\def \C {{\mathbb{C}}}
\def \R {{\mathbb{R} }}
\def \d {{\rm{d}}}
\def \pt{\partial_{t}}

\def\Op{\mathrm{Op}^{\mathrm{w}}}
\def\poscals#1#2{\langle#1,#2\rangle}
\def \p {{\partial}}
\newcommand{\T}{{\mathbb T}}
\newcommand{\Z}{{\mathbb Z}}
\newcommand{\Q}{{\mathbb Q}}
\newcommand{\N}{{\mathbb N}}
\newcommand{\ii }{{\rm i} }
\newcommand{\h}{\mathcal{H}}

\def \A {{\mathbf{A}}}
\newcommand{\avg}[2]{\langle #1\rangle_{#2}}
\newcommand{\bigO}[2]{\mathcal{O}_{#2}\mathopen{}\left(#1\right)}
\newcommand{\BigO}[1]{\mathcal{O}\mathopen{}\left(#1\right)}

\newcommand{\supp}{\operatorname{supp}}
\newcommand{\ad}{\operatorname{ad}}

\renewcommand{\Im}{\mathop{\rm Im}\nolimits}
\newcommand{\vh}[2]{{#1}^{(#2)}_h}
\newcommand{\vnh}[2]{{#1}^{(#2)}}
\newcommand{\WFm}[1]{\mathrm{WF}^{#1}_h}
\newcommand{\WF}{\mathrm{WF}_h}

\def\poscals#1#2{\langle#1,#2\rangle}
\def \p {{\partial}}

\def \A {{\mathbf{A}}}

\renewcommand{\Im}{\mathop{\rm Im}\nolimits}

\newcommand{\rr}{\mathbb{R}}
\newcommand{\nn}{\mathbb{N}}

\newcommand{\norme}[1]{\left\lVert#1\right\rVert}

\begin{document}

 	\title[Geometric condition for observability of electromagnetic Schr\"odinger operators]
	{Geometric condition for the observability of electromagnetic Schr\"odinger operators on $\mathbb{T}^2$}
	
 \author[K. Le Balc'h]{K\'evin Le Balc'h}
	\address{Sorbonne Universit\'e, CNRS, Universit\'e Paris Cit\'e, Inria, Laboratoire Jacques-Louis Lions (LJLL), F-75005 Paris, France}
	\email{kevin.le-balc-h@inria.fr}

    \author[J. Niu]{Jingrui Niu}
	\address{Institute for Advanced Study in Mathematics, Harbin Institute of Technology, 150001, Harbin, China}
	\email{jingrui.niu@hit.edu.cn}

    \author[C. Sun]{Chenmin Sun}
	\address{CNRS, Universit\'e Paris-Est Cr\'eteil, Laboratoire d'Analyse et de Math\'ematiques appliqu\'ees, UMR  8050 du CNRS, 
	94010 Cr\'eteil cedex, France.}
	\email{chenmin.sun@cnrs.fr}

\begin{abstract}
In this article we revisit the observability of the Schr\"odinger equation on the two-dimensional torus. In contrast to the Schr\"odinger operator with a purely electric potential, for which any non-empty open set guarantees observability, the presence of a magnetic potential introduces an additional obstruction. We establish a sufficient and almost necessary geometric condition for the observability of electromagnetic Schr\"odinger operators. This condition incorporates the magnetic potential, which can also be characterized by a geometric control condition for the corresponding magnetic field. 
\end{abstract}
\maketitle

\tableofcontents

\section{Introduction}
In this paper, we focus on the Schr\"odinger equation, set on $\T^2 =\R^2/ (2\pi\Z)^2$ the two-dimensional torus, that governs the time evolution of the wave function $u=u(t,z)$ of a charged particle in an electromagnetic field. More precisely, we consider
\begin{equation}
	\label{eq:SchrodingerObs}
		\left\{
			\begin{array}{ll}
			\ii  \partial_t u  =H_{\A,V}(z) u & \text{ in }  (0,+\infty) \times \T^2, \\
				u(0, \cdot) = u_0 & \text{ in } \T^2,
			\end{array}
		\right.
\end{equation}
where $H_{\A,V}$ is the electromagnetic Schr\"odinger operator given by 
\begin{equation}\label{eq: defi-m-sch-op}
H_{\A,V}(z):=\left(\frac{1}{\ii}\nabla-\A(z)\right)^2+V(z),\qquad z\in\T^2,
\end{equation}
with 
\begin{equation}
    \label{eq:regVA}
    V\in C^{\infty}(\T^2,\R),\ \A=(A_1,A_2)\in C^{\infty}(\T^2,\R^2).
\end{equation}

It is standard that $-i H_{\A,V}$ generates an unitary semi-group $(e^{-\ii tH_{\A,V}})_{t \geq 0}$ in $L^2(\T^2)$ so for every $u_0 \in L^2(\T^2)$, there exists a unique mild solution $u \in C([0,+\infty);L^2(\T^2))$ of \eqref{eq:SchrodingerObs}. Moreover the total $L^2$-mass of the solution is preserved along the time, i.e.
\begin{equation}
    \label{eq:conservationmass}
    \|u(t,\cdot)\|_{L^2(\T^2)} =  \|u_0\|_{L^2(\T^2)}\qquad \forall t \geq 0.
\end{equation}

We are interested in the following notion of \textit{observability} of the Schr\"odinger equation \eqref{eq:SchrodingerObs}.
\begin{definition}
Let $T>0$ and $\omega\subset\T^2$ be a nonempty open subset. We say that the observability of \eqref{eq:SchrodingerObs} at time $T$ from $\omega$ holds true if there exists a constant $C=C(\mathbf{A},V,T,\omega)>0$ such that
\begin{equation}\label{eq: ob-magnetic-sch-intro}
\|u_0\|^2_{L^2(\T^2)}\leq C\int_0^T\int_{\omega}\left|e^{-\ii tH_{\A,V}}u_0(z)\right|^2\d z\d t,   \qquad \forall u_0\in L^2(\T^2). 
\end{equation}
\end{definition}
At time $T>0$ and for $\omega \subset \T^2$, the observability inequality \eqref{eq: ob-magnetic-sch-intro} models the possibility of estimating the (whole) $L^2(\T^2)$-norm of the initial wave function $u(0,\cdot)=u_0$ by the (restricted) $L^2((0,T)\times\omega)$-norm of the wave function $u=u(t,z)$. For instance if the initial datum is normalized, $\|u_0\|_{L^2(\T^2)}=1$, the expression $\int_{\omega}\left|e^{-\ii tH_{\A,V}}u_0(z)\right|^2\d z$ is the probability of finding the particle in the set $\omega$ at time $t$. Having $\int_0^T\int_{\omega}\left|e^{-\ii tH_{\A,V}}u_0(z)\right|^2\d z\d t \geq C^{-1} > 0$ for all solutions of \eqref{eq:SchrodingerObs} means that every quantum particle spends a positive fraction of time of the interval $(0,T)$ in the set $\omega$ with positive probability.\\

\textbf{Brief state of the art.} When $\A\equiv 0, V \equiv 0$, it is well-known that \eqref{eq: ob-magnetic-sch-intro} holds for every $T>0$ and nonempty open subset $\omega \subset \T^2$, see \cite{Jaf90} and \cite{KL05}. The pure electric case i.e. when $\A \equiv 0$ has been investigated a lot in the last decade in the papers \cite{BZ12}, \cite{BBZ13}, \cite{BZ19}, \cite{AM14}, \cite{AFKM15}, \cite{Mac21}, \cite{Tao21}, the best result so far is that \eqref{eq: ob-magnetic-sch-intro} holds for all $V \in L^{2}(\T^2)$, $T>0$ and measurable set $\omega \subset \T^2$ with positive Lebesgue measure $|\omega|>0$. When $\A = (\theta_1,\theta_2) \in \R^2$ is a constant and $V \in L^{\infty}(\T^2)$, \cite{LBM23} shows that \eqref{eq: ob-magnetic-sch-intro} holds for every $T>0$ and measurable set $\omega \subset \T^2$ with positive Lebesgue measure $|\omega|>0$.  \cite{macia2014}, \cite{Wunsch-wave} treated the case $\mathbb{T}^d$ and $\omega\subset\mathbb{T}^d$ open, see Remark \ref{remark:main}. The non-constant magnetic case has not been treated yet up to the knowledge of the authors. Nevertheless, it is important to note that when $\omega$ satisfies the so-called Geometric Control Condition (GCC) from \cite{BLR92}, i.e. every closed geodesic $\gamma \subset \T^2$ with a periodic direction $\Vec{\gamma}$ meets the set $\omega$, then \eqref{eq: ob-magnetic-sch-intro} holds for every $V\in C^{\infty}(\T^2,\R)$, $\A\in C^{\infty}(\T^2,\R^2)$ and $T>0$ by using Lebeau's approach, see \cite{Leb92}. \\

In this paper, we mainly investigate the case where $\omega$ does not satisfy (GCC). In this case, there exists at least one closed geodesic $\gamma\subset\T^2$ with a periodic direction $\Vec{\gamma}$ such that $\gamma\cap\omega=\emptyset$. As we will see, we need to impose geometrical assumptions on the open set $\omega$, involving the critical points of $\A$ to prove \eqref{eq: ob-magnetic-sch-intro}. Hence, the electromagnetic case is completely different from the pure electric case.

\subsection{The Magnetic Geometric Control Condition}
For $N \geq 1$ and any $f \in L^1(\T^2;\R^N)$, we first define the following quantity as the average of $f$ along a given direction $\Vec{e}$
\begin{equation}
\label{eq: direction-average}
\avg{f}{\Vec{e}}(z):=\lim_{T\rightarrow\infty}\frac{1}{T}\int_0^Tf(z+t\Vec{e})\d t\qquad \forall z \in \T^2.
\end{equation}
Two important things have to be noted.
\begin{enumerate}
    \item First, the limit defined in \eqref{eq: direction-average} is well-defined by checking its meaning for trigonometric polynomials and by using a density argument. The resulted function $\avg{f}{\Vec{e}}$ belongs to $L^1(\T^2)$.
    \item Secondly, one can check that the directional derivative of $\avg{f}{\Vec{e}}$ along $\Vec{e}$, in the distributional sense, satisfies
    $$ D\avg{f}{\Vec{e}}(z)\cdot \Vec{e} = 0\qquad \forall z  \in \T^2.$$
    Therefore we can identify $\avg{f}{\Vec{e}}$ as a function on the direction $\Vec{e}^\perp$.
\end{enumerate}
For a closed geodesic $\gamma$ of periodic direction $\Vec{\gamma}$, the average of the vector field $\A$ along the direction $\Vec{\gamma}$ is given by
\begin{equation}
\label{eq: direction-averageVector}
\avg{\A}{\Vec{\gamma}}(z)=\lim_{T\rightarrow\infty}\frac{1}{T}\int_0^T\A(z+t\Vec{\gamma})\d t\qquad \forall z \in \T^2.
\end{equation}
We can then define the scalar function 
\begin{equation}\label{eq: defi-A-gamma}
A_{\gamma}:=\avg{\A}{\Vec{\gamma}}\cdot\Vec{\gamma}.
\end{equation}
 So one can identify $A_{\gamma}$ as a scalar function on the direction of $\Vec{\gamma}^{\perp}$. 

\medskip

For a closed geodesic $\gamma$ with a periodic direction $\Vec{\gamma}$, we denote by $\omega_{\Vec{\gamma}^{\perp}}$ the projection of $\omega$ on the direction of $\Vec{\gamma}^{\perp}$, which is also open.

\begin{figure}[h]
    \centering
    \includegraphics[width=0.4\textwidth]{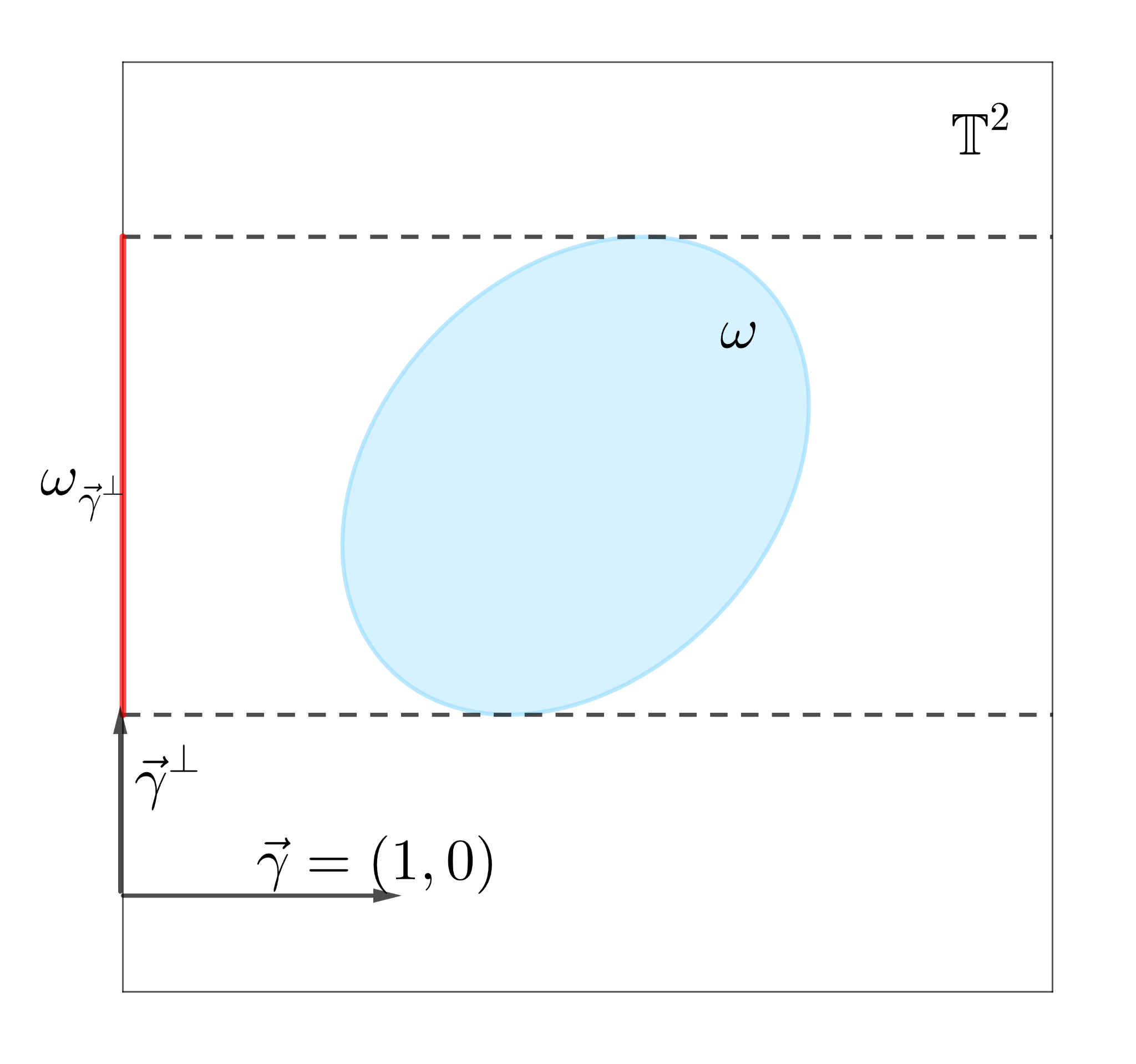}
    \caption{Control region projection}
    \label{fig: control region projection}
\end{figure}

With the concepts above, we introduce the following Magnetic Geometric Control Condition, called {\bf{(MGCC)}} in the next.
\begin{definition}[{\bf{MGCC}}]
\label{def:MGGC}
Let $\A \in C^{\infty}(\T^2;\R^2)$ and $\omega$ be a nonempty open subset of $\T^2$. We say that $\A$ and $\omega$ satisfy the Magnetic Geometric Control Condition {\bf{(MGCC)}} if for every periodic direction $\Vec{\gamma}$, $\omega_{\Vec{\gamma}^{\perp}}$ contains all the critical points of $A_{\gamma}$.
\end{definition}

The condition from Definition \ref{def:MGGC} can be interpreted in an intrinsic way. We identify the magnetic potential $\mathbf{A}$ as a globally defined one-form $\mathcal{A}=A_1\d x+A_2\d y$ on the coordinate system $(x,y)$. We parametrize  $\vec{\gamma}=(\cos\theta,\sin\theta),\vec{\gamma}^{\perp}=(-\sin\theta,\cos\theta)$. Then in the new coordinate system $$(X,Y):=F(x,y)=(x\cos\theta +y\sin\theta, -x\sin\theta+y\cos\theta ),$$ 
 the magnetic potential (identified as the pull-back form $F^{-1*}\mathcal{A}$) is
$$ (\mathbf{A}\cdot\vec{\gamma})\d X + (\mathbf{A}\cdot\vec{\gamma}^{\perp}) \d Y.
$$
Note that the magnetic field  $B=\d \mathcal{A}$, identified as a scalar-function on $\mathbb{T}^2$, in the coordinate system $(X,Y)$, is given by
$$ B=\partial_{X}(\mathbf{A}\cdot\vec{\gamma}^{\perp})-\partial_{Y}(\mathbf{A}\cdot\vec{\gamma}).
$$
Denote 
\begin{equation}
    \label{eq:defBgamma}
    \langle B\rangle_{\gamma}(z,\zeta):=\lim_{T\rightarrow\infty}\frac{1}{T}\int_0^T B(z+t\zeta) \d t, \qquad (z,\zeta) \in S^*\mathbb{T}^2.
\end{equation}
If the direction $\vec{\gamma}$ is periodic, then
the closed geodesic $\gamma$ is parametrized by the equation $Y=0$. So
$$ \langle B\rangle_{\gamma}(Y)=-\partial_{Y}A_{\gamma}(Y).
$$
Hence the following result holds.
\begin{proposition}\label{MGCC:equiv}
Let $\A \in C^{\infty}(\T^2;\R^2)$ and $\omega$ be an open subset of $\T^2$. The following two assertions are equivalent.
\begin{enumerate}
    \item {\bf{(MGCC)}} holds.
    \item For every periodic direction $\Vec{\gamma}$, $\omega_{\Vec{\gamma}^{\perp}}$ contains all the zeros of $\langle B\rangle_{\gamma}$, defined in \eqref{eq:defBgamma}.
\end{enumerate}
\end{proposition}

\subsection{Observability estimates of the electromagnetic Schr\"odinger equation}
In this part, we present the main results of our paper.\\ 

First, we have the following positive result concerning the the observability \eqref{eq: ob-magnetic-sch-intro} of the electromagnetic Schr\"odinger equation \eqref{eq:SchrodingerObs} under {\bf{(MGCC)}}.
\begin{theorem}\label{thm: main}
Let $\A \in C^{\infty}(\T^2;\R^2)$, $V \in C^{\infty}(\T^2;\R)$ and $\omega$ be a nonempty open subset of $\T^2$. Assume that $\A$ and $\omega$ satisfy {\bf{(MGCC)}}. Then, for every $T>0$, the observability \eqref{eq: ob-magnetic-sch-intro} of the electromagnetic Schr\"odinger equation \eqref{eq:SchrodingerObs} holds.
\end{theorem}
\begin{remark}\label{remark:main}
In the special case where the magnetic field $B=0$, since the de Rham cohomology group $H_{\mathrm{dR}}^1(\mathbb{T}^2)\cong \mathbb{R}\oplus \mathbb{R}$, there exists a gauge transform $g\in C^{\infty}(\mathbb{T}^2,\mathbb{R})$, such that $\mathbf{A}+\nabla g$ is a constant magnetic potential. Then, {\bf{(MGCC)}} implies that $\omega = \T^2$. In such case, Theorem \ref{thm: main} still holds but is not optimal. By the work \cite{LBM23}, the observability of \eqref{eq:SchrodingerObs} holds for any $T>0$ and any measurable set $\omega\subset \mathbb{T}^2$ of positive measure. For such a constant magnetic potential case and for $\omega\subset\mathbb{T}^d$ a non-empty open set, the same observability result holds on all $\mathbb{T}^d$, $d\geq 1$, following the arguments of \cite{AM14} and the Bloch-Floquet theory, see \cite{macia2014},\cite{Wunsch-wave}.
\end{remark}

Our second main result concerns resolvent estimates for the stationary Schr\"odinger operator under {\bf{(MGCC)}}.
\begin{theorem}\label{cor:stationary}
	Let $\A \in C^{\infty}(\T^2;\R^2)$, $V \in C^{\infty}(\T^2;\R)$ and $\omega$ be a nonempty open subset of $\T^2$. Assume that $\A$ and $\omega$ satisfy {\bf{(MGCC)}}. Then there exists $C=C(\A,V,\omega)>0$ such that 
	\begin{equation}
	    \label{eq:resolventestimate}
	     \|u_{\lambda}\|_{L^2(\mathbb{T}^2)}\leq C\left(\frac{1}{1+|\lambda|^{\frac{1}{4}}}\|(H_{\A,V} + \lambda) u_{\lambda}\|_{L^2(\mathbb{T}^2)}+\|u_{\lambda}\|_{L^2(\omega)}\right),\quad \forall \lambda \in \R,\ \forall u_{\lambda} \in H^2(\T^2).
	\end{equation}
\end{theorem}
Actually, the resolvent estimate \eqref{eq:resolventestimate} in Theorem \ref{cor:stationary} rather comes from a ``sharp'' semiclassical observability estimate stated in Proposition \ref{prop: semiclassical ob}, which is stronger than  Theorem \ref{thm: main}, that would lead to the following resolvent estimates
	\begin{equation}
	    \label{eq:resolventestimateNonsharp}
	     \|u_{\lambda}\|_{L^2(\mathbb{T}^2)}\leq C\left(\|(H_{\A,V} + \lambda) u_{\lambda}\|_{L^2(\mathbb{T}^2)}+\|u_{\lambda}\|_{L^2(\omega)}\right),\quad \forall \lambda \in \R,\ \forall u_{\lambda} \in H^2(\T^2).
	\end{equation}
Indeed, for any $u_{\lambda}\in H^2(\T^2)$, we set $v_{\lambda}(t,z):=\mathrm{e}^{\ii t\lambda}u_{\lambda}(z)$ and $f_{\lambda}(z):=(H_{\A,V} + \lambda) u_{\lambda}\in L^2(\T^2)$. Then, $v_{\lambda}$ is the unique solution to the equation 
\begin{equation*}
(\ii\p_t-H_{\A,V})v_{\lambda}=-\mathrm{e}^{\ii t\lambda}f,\;\;v_{\lambda}\big|_{t=0}=u_{\lambda}.
\end{equation*} 
Thanks to the observability \eqref{eq: ob-magnetic-sch-intro} and Duhamel's formula, we obtain
\begin{align*}
\|u_{\lambda}\|^2_{L^2(\T^2)}&\leq C\int_0^T\int_{\omega}\left|\mathrm{e}^{-\ii t H_{\A,V}}u_{\lambda}(z)\right|^2\d z\d t\\
&=C\int_0^T\int_{\omega}\left|v_{\lambda}(t,z)+\frac{1}{\ii}\int_0^t\mathrm{e}^{-\ii(t-s) H_{\A,V}}\mathrm{e}^{\ii s\lambda}f(z)\d s\right|^2\d z\d t.
\end{align*}
By Cauchy--Schwarz's inequality and the definitions of $v_{\lambda}$ and $f_{\lambda}$, we derive that
\begin{align*}
\|u_{\lambda}\|^2_{L^2(\T^2)}&\leq 2CT\|u_{\lambda}\|^2_{L^2(\omega)}+2C\int_0^T\int_0^t\int_{\T^2}\left|\mathrm{e}^{-\ii(t-s) H_{\A,V}}f(z)\right|^2\d z\d s\d t\\
&\leq 2CT\|u_{\lambda}\|^2_{L^2(\omega)}+2CT^2\|(H_{\A,V} + \lambda) u_{\lambda}\|_{L^2(\T^2)}^2.
\end{align*}
This means that \eqref{eq:resolventestimateNonsharp} is a direct consequence of \eqref{eq: ob-magnetic-sch-intro}.   
\begin{remark}
Very recently, in \cite{burq-zhu}, Burq and Zhu proved that on a compact Riemannian manifold, resolvent estimates for the Laplace--Beltrami operator imply observability for the Schr\"odinger propagator from time sets of positive Lebesgue measure. In fact, their result can also apply in our electromagnetic Schr\"odinger setup in $\T^2$. Therefore, the resolvent estimate \eqref{eq:resolventestimate} implies the observability 
\begin{equation*}
\|u_0\|^2_{L^2(\T^2)}\leq C\int_E\int_{\omega}\left|e^{-\ii tH_{\A,V}}u_0(z)\right|^2\d z\d t,   \qquad \forall u_0\in L^2(\T^2). 
\end{equation*}
where $E$ is a time set of positive Lebesgue measure and $\omega$ satisfies {\bf{(MGCC)}} in Theorem \ref{thm: main}.
\end{remark}

\vspace{0.3cm}

Our third main result establishes that the geometric condition {\bf{(MGCC)}} of Theorem \ref{thm: main} is optimal in the following sense.
\begin{theorem}\label{thm: optimal}
Let $\A \in C^{\infty}(\T^2;\R^2)$, $V \in C^{\infty}(\T^2;\R)$ and $\omega$ be a nonempty open subset of $\T^2$. Assume that $\omega$ is such that there exists a closed geodesic $\gamma\cap\overline{\omega}=\emptyset$ with the direction $\Vec{\gamma}$ and for some periodic direction $\vec{\gamma}$, $A_{\gamma}$ has at least one non-degenerate critical point outside $\overline{\omega}_{\vec{\gamma}^{\perp}}$.
Then for any $T>0$, the observability \eqref{eq: ob-magnetic-sch-intro} does not hold.
\end{theorem}

\begin{remark}
	For a specific class of magnetic potential $\mathbf{A}$ such that $A_{\gamma}$ is a Morse function for any periodic direction $\vec{\gamma}$, Theorem \ref{thm: optimal}
shows that the condition {\bf{(MGCC)}} is sharp, up to the broad-line case where some critical point of $A_{\gamma}$ lies on the boundary of $\omega_{\vec{\gamma}^{\perp}}$, which is not considered in the current paper. 

More generally, if one deals with finite order degenerate critical points outside $\overline{\omega}_{\vec{\gamma}^{\perp}}$, our proof ideas seem to work in a similar way, at least in the case where $A_{\gamma}'(z_0)=\cdots=A_{\gamma}^{(2k-1)}(z_0)=0$ and $A_{\gamma}^{(2k)}(z_0)\neq0$ for some $k\in\N^*$. After reducing to the one-dimensional model equation, one will face the anharmonic oscillator instead of harmonic oscillator in the non-degenerate critical point case, which requires a careful and delicate spectral analysis of the semiclassical anharmonic oscillator, and we choose to not include it in the current paper.
\end{remark}

\textbf{A toy model.} We present an example of a magnetic potential $A$ and an associated $\omega$ such that (MGGC) holds to illustrate Theorem \ref{thm: main}. Let us assume that $\A = \A(y) = (A_1(y), A_2(y))$. Up to a gauge transformation one can assume that $\int_{\T} \A(y) dy = 0$. Then for every unit periodic direction $\Vec{\gamma}= (\gamma_1,\gamma_2)$ it is easy to see that we have the dichotomy $A_{\gamma} = 0$ if $\gamma_2 \neq 0$ and $A_{\gamma}=\pm A_1$ if $\gamma_2 = 0$. For $\gamma_2 \neq 0$, $\omega_{\gamma^\perp}$ shoud contain all the critical points of $A_{\gamma} = 0$ then $\omega_{\gamma^\perp}$ should be equal to the whole one-dimensional torus $\T$. So $\omega$ has to be an horizontal strip of the form $\omega = \T \times \omega_y$ where $\omega_y \subset \T$. Moreover, for $\gamma_2=0$, the corresponding orthogonal direction is $\gamma^\perp=(0,1)$ so we find that $\omega_{\gamma^\perp} = \omega_y$ hence $\omega_y$ has to contain all the critical points of $A_1$.

\begin{figure}[h]
    \centering
    \includegraphics[width=0.8\textwidth]{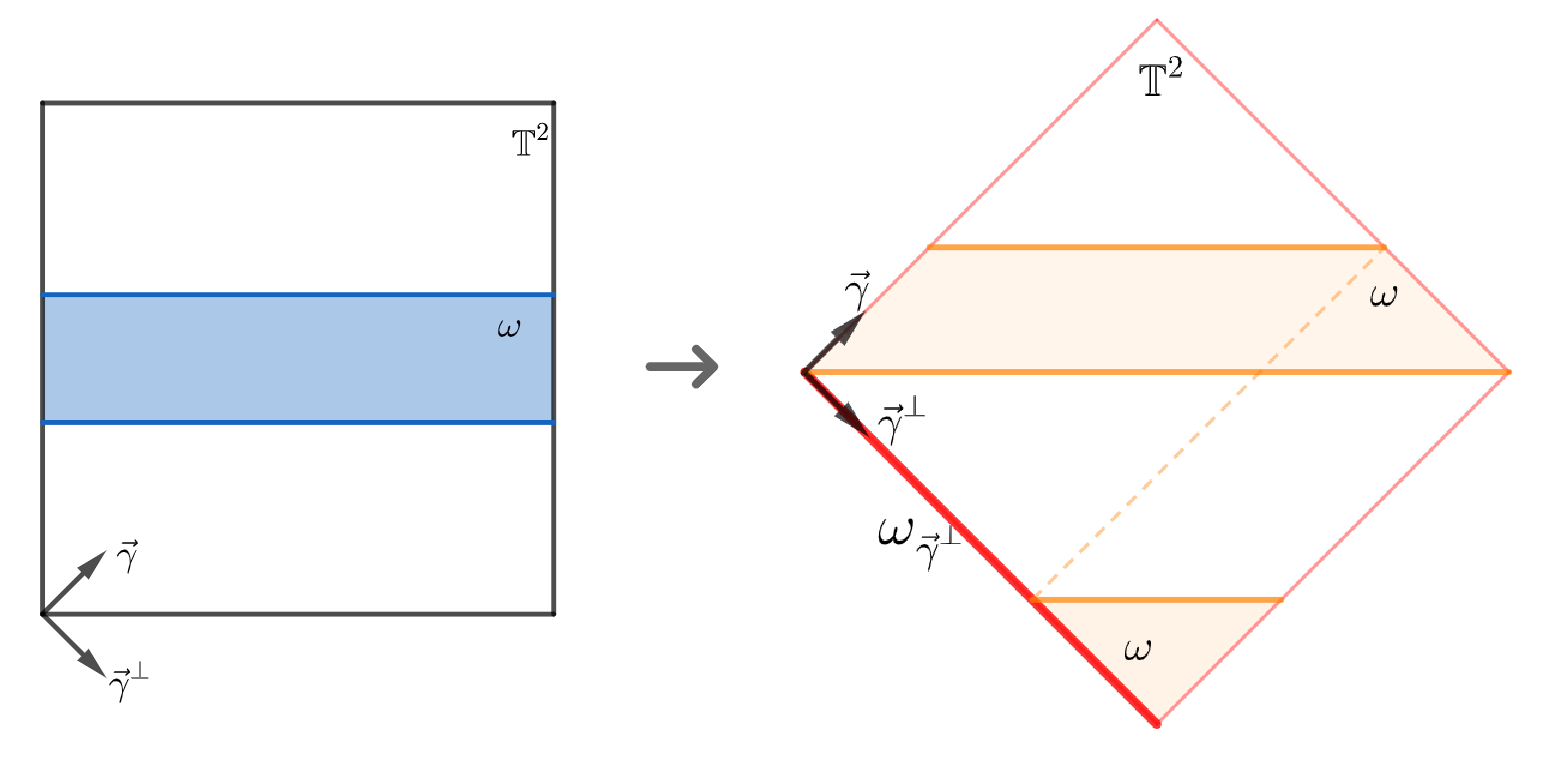}
    \caption{Control region projection for $\gamma_2\neq0$}
    \label{fig: omega-projection}
\end{figure}

\begin{figure}[h]
    \centering
    \includegraphics[width=0.8\textwidth]{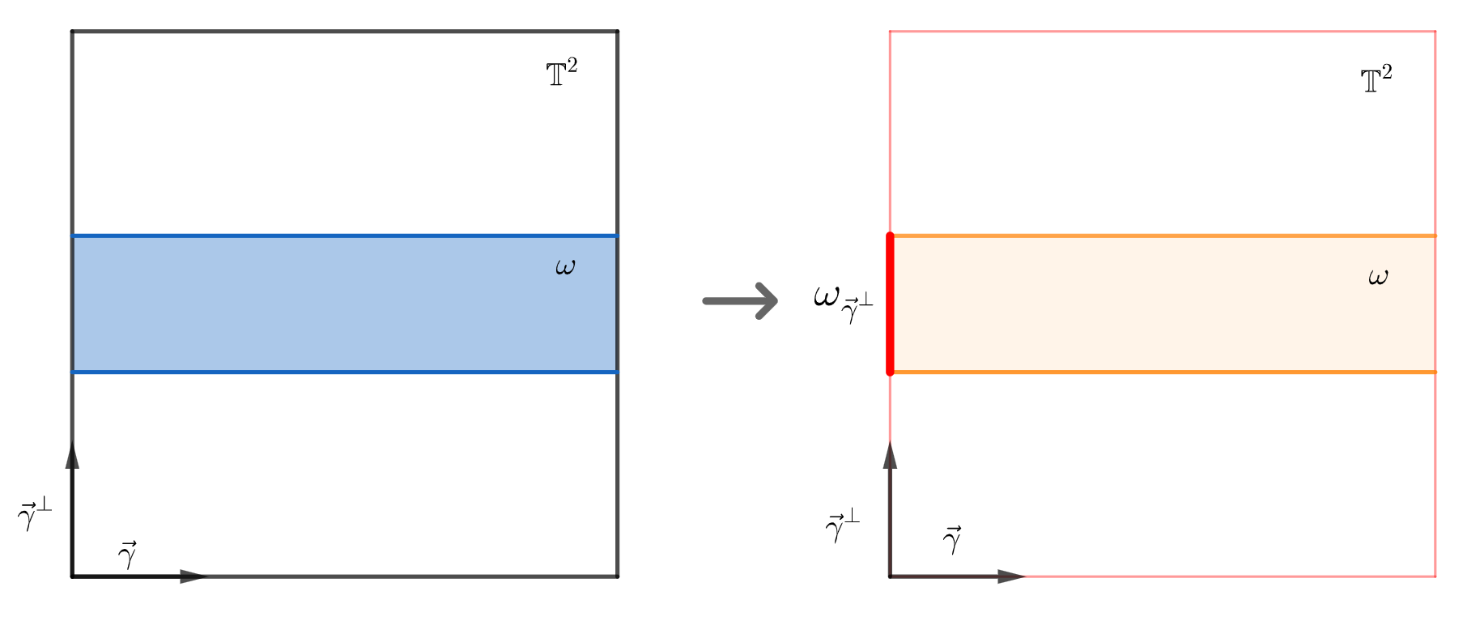}
    \caption{Control region projection for $\gamma_2=0$}
    \label{fig: region-1-0}
\end{figure}

\subsection{Control of the electromagnetic Schr\"odinger equation}

In this part, we state some control results that can be deduced from the observability estimate \eqref{eq:SchrodingerObs} of Theorem \ref{thm: main}.\\

We first consider the controlled Schr\"odinger equation
\begin{equation}
	\label{eq:SchrodingerControl}
		\left\{
			\begin{array}{ll}
			\ii  \partial_t \psi  =H_{\A,V}(z) \psi + h 1_{\omega} & \text{ in }  (0,+\infty) \times \T^2, \\
				\psi(0, \cdot) = \psi_0 & \text{ in } \T^2.
			\end{array}
		\right.
\end{equation}
In \eqref{eq:SchrodingerControl}, at time $t \in [0,+\infty)$, $\psi(t,\cdot) : \T^2 \to \C$ is the state and $h(t,\cdot) : \omega \to \C$ is the control.

From Theorem \ref{thm: main} and a standard duality argument, see for instance \cite[Theorem 2.42]{Cor07}, the following small-time exact controllability of \eqref{eq:SchrodingerControl} holds.

\begin{theorem}
Let $\A \in C^{\infty}(\T^2;\R^2)$, $V \in C^{\infty}(\T^2;\R)$ and $\omega$ be a nonempty open subset of $\T^2$. Assume that $\A$ and $\omega$ satisfy {\bf(MGCC)}. Then for every $T>0$, $\psi_0$ and $\psi_1 \in L^2(\T^2)$, there exists $h \in L^2((0,T)\times\omega)$ such that the solution $\psi \in C([0,T];L^2(\T^2))$ of \eqref{eq:SchrodingerControl} satisfies $\psi(T,\cdot) = \psi_1$.
\end{theorem}

As a consequence of the negative result of Theorem \ref{thm: optimal}, one can also establish that for every time $T>0$, the exact controllability of \eqref{eq:SchrodingerControl} does not hold if the assumptions of Theorem \ref{thm: optimal} are fulfilled.\\

Finally, let us consider the damped Schr\"odinger equation
\begin{equation}
	\label{eq:SchrodingerControlDamped}
		\left\{
			\begin{array}{ll}
			\ii  \partial_t \psi  =H_{\A,V}(z) \psi - i a(x) \psi & \text{ in }  (0,+\infty) \times \T^2, \\
				\psi(0, \cdot) = \psi_0 & \text{ in } \T^2,
			\end{array}
		\right.
\end{equation}
where $a=a(x) \in C^{\infty}(\mathbb{T}^2;[0,+\infty))$, not identically zero. Denote $\omega:=\{x:\; a(x)>0\}$.

From Theorem \ref{thm: main} and a standard result due to Haraux, see \cite{Har89-2} and also \cite[Theorem 5.31]{Tre14},  we have the following exponential stability of 
\eqref{eq:SchrodingerControlDamped}.
\begin{theorem}
Let $\A \in C^{\infty}(\T^2;\R^2)$, $V \in C^{\infty}(\T^2;\R)$ and $\omega$ be a nonempty open subset of $\T^2$. Assume that $\A$ and $\omega$ satisfy {\bf{(MGCC)}}. Then there exists $C=C(\A,V,\omega)>0$ and $\alpha=\alpha(\A,V,\omega)>0$ such that for every $\psi_0 \in L^2(\T^2)$
\begin{equation}
    \label{eq:exponentialstability}
    \|\psi(t,\cdot)\|_{L^2(\T^2)} \leq C e^{- \alpha t}    \|\psi_0\|_{L^2(\T^2)}\qquad \forall t \geq 0.
\end{equation}
\end{theorem}

\subsection{The role of the magnetic potential and the strategy of the proof}\label{sec: role of magnetic}

To understand the role of the magnetic potential $\mathbf{A}$, we assume first that $\mathbf{A}$ takes the simple form $\mathbf{A}=(A_1(y),A_2(y))$ depending only on the transversal direction $(0,1)$.  
In the trapped regime where $|D_x|\sim h^{-1}, |D_y|\ll h^{-1}$, the singularities of the solutions of the Schr\"odinger equation have the potential of concentration along the co-isotropic submanifold $\{\zeta_0=(1,0)\}$. 
Then the simplified magnetic Schr\"odinger operator $H_{\mathbf{A},V}$ has the principle term
$  D_x^2+D_y^2-2A_1(y)D_x,
$
in the regime where $D_y$ oscillates at a second scale $\hbar^{-1}:=h^{-\frac{1}{2}}$, as $D_y^2$ and $2A_1(y)D_x$ have the same order (see Section \ref{sec: optimality} for more details).   Taking the Fourier transform in $x$ and shifting a phase factor in time, the Schr\"odinger equation behaves like
$$  i\hbar^2\partial_tu-(\hbar^2D_y^2\pm2A_1(y))u=0,
$$
 which is dominated by the $\hbar$-semiclassical operator $\hbar^2D_y^2\pm 2A_1(y)$. If $A_1(y)$ has some non-degenerate critical point  
$y=y_0$, then singularities could be confined  in a small neighborhood of $y_0$. This crucial observation leads to the necessity of the geometric condition {\bf{(MGCC)}} for the observability of \eqref{eq:SchrodingerObs} to hold.
\vspace{0.3cm}

To prove the positive result under {\bf{(MGCC)}}, we first follow the classical compactness-uniqueness strategy in the context of Schr\"odinger equations (see \cite{BZ12}) to reduce the proof to the high-frequency observability. When the classic geometric control condition is not satisfied, the crucial part of our analysis is to reduce the problem to the model case described above. Under {\bf{(MGCC)}}, we are indeed able to prove the high-frequency (oscillating at scale $O(h^{-1})$) observability in an optimal time scale $t=O(h^{\frac{1}{2}})$.
Zooming into the new time scale $s=t/h$, the magnetic Schr\"odinger equation takes the semiclassical form
$  ih\partial_su_h=P_hu_h,
$
where
\begin{equation}\label{eq: Ph-p1-p2-intro}
P_h=h^2D_z^2+hp_1^{\mathrm{w}}(z,hD_z)+h^2p_2^{\mathrm{w}}(z,hD_z)
\end{equation}
which contains both the first and second order perturbations of the principal symbol $p_0=|\zeta|^2$. In fact, we shall see the explicit forms for $p_1$ and $p_2$ in \eqref{eq: p1-p2-defi} in Section \ref{sec: first microlocalization}.

The observability property for $t\sim O(h^{\frac{1}{2}})$ requires understanding of the semiclassical propagator $e^{- \frac{\ii sP_h}{h}}$ up to the time scale $s\sim O(h^{-\frac{1}{2}})$. The second order perturbation $h^2p_2^{\mathrm{w}}(z,hD_z)$ will not influence the dynamics within this long-time scale of the free propagator $e^{-\ii shD_z^2}$, and can be simply viewed as a perturbation. It turns out that the presence of the first order perturbation $hp_1^{\mathrm{w}}(z,hD_z)$ differs the long-time dynamics from the free-propagator, which yields different phenomena. This was also observed in \cite{MR18} (see also \cite{Wunsch}) in the study of the two-microlocal regularity of quasimodes of the semiclassical operator $-h^2\Delta+\epsilon_h^2 V$ on $\mathbb{T}^2$, where $0\leq \epsilon_h\leq h$.

We adopt a normal form approach in the spirit of \cite{BZ12} (see also \cite{Sun-dampedwave} in another setting of the first order perturbation). The basic idea is to conjugate the Schr\"odinger operator $H_{\mathbf{A},V}$ by suitable unitary operators that will average the lower order symbols along the trapped direction. In order to average the first order perturbations from the magnetic potential, we need to conjugate by a bounded operator that is not close to the identity (which is different from \cite{BZ12}). This procedure produces remainders that can not a priori be able to absorbed as acceptable errors of size $o(h^{3/2})$. So we have to treat separately the transversal low and high-frequencies in Section 4 and Section 5 in a similar way as in \cite{Sun-dampedwave}, in order to finally reduce the problem to the previously discussed model equation.

\subsection{Related results on the quantum unique ergodicity}

 Very recently,  Morin and Rivi\`ere \cite{RM24} prove the quantum unique ergodicity for the magnetic Laplacian on $\mathbb{T}^2$, under a slightly different but related magnetic geometric control condition. Note that the magnetic Laplacian is more generally defined as operators acting on sections of a line bundle over $\mathbb{T}^2$. Then the quantum unique ergodicity holds under their magnetic geometric control conditions asserting that the averaged magnetic field along any direction $\langle B\rangle_{\gamma}$ is everywhere non-vanishing. Note that in our setting where the magnetic field is derived from the magnetic potential $B=\nabla\wedge \mathbf{A}$, such a condition cannot hold since $\int_{\mathbb{T}^2}B \d z=0$. Nevertheless, their approach, different from ours, could eventually obtain the same resolvent estimate, i.e. Theorem \ref{cor:stationary}, under {\bf{(MGCC)}}. Moreover, all of our results can be likely extended to the general magnetic Laplacian as in the setting of \cite{RM24}, under the magnetic geometric control condition {\bf{(MGCC)}} stated under the form (ii) of Proposition \ref{MGCC:equiv}. 
  Even more recently, Charles and Lefeuvre \cite{CL25} investigate semiclassical defect measures associated with the magnetic Laplacian in the presence of a constant magnetic field on a closed hyperbolic surface.

\subsection*{Acknowledgments}
The authors would like to thank Gabriel Rivi\`ere for useful discussion. C. Sun is partially supported by ANR project SmoothANR-22CE40-0017. K. Le Balc'h is partially supported by ANR project Trecos ANR-20-CE40-0009. J. Niu is supported by Defi Inria EQIP and Inria Team CAGE. \\

\section{Semiclassical reduction of the observability}
\label{sec: semi-classical reduction of the observability}

\subsection{Sharp semiclassical observability estimates}

The goal of this part is to state semiclassical observability estimates that would be the key step of the proof of Theorem \ref{thm: main}.\\

First, let us define $\chi\in C^{\infty}_c(-1,1)$ such that $\chi$ is equal to $1$ near $0$, and the associated spectral projector for some (small) parameters $\rho>0$, $h>0$,
\begin{equation}
\label{eq: spectral projector}
\Pi_{h,\rho}u:=\chi\left(\frac{h^2H_{\A,V}-1}{\rho}\right)u,\qquad u \in L^2(\T^2).
\end{equation}

Our main result of this part is the following semiclassical observability estimate.
\begin{proposition}\label{prop: semiclassical ob}
There exists a numerical constant $T_0=T_0(\A,\omega)>0$, depending only on $\mathbf{A}$ and $\omega$, such that for any $T\geq T_0$, there exist constants $\rho_0>0$, $h_0>0$ and $C>0$ such that for any $\rho \in (0, \rho_0)$, $h\in(0,h_0)$, we have 
\begin{equation}\label{eq: sharp ob}
\|\Pi_{h,\rho}u_0\|^2_{L^2(\T^2)}\leq C\int_0^T\int_{\omega}\left|e^{-\ii th^{\frac{1}{2}}H_{\A,V}}\Pi_{h,\rho}u_0(z)\right|^2\d z\d t\qquad \forall u_0 \in L^2(\T^2).
\end{equation}
\end{proposition}
  We point out that Proposition \ref{prop: semiclassical ob} is the sharp version of semiclassical observability under our electromagnetic setting. For the sharpness, we refer to Section \ref{sec: optimality}.

\subsection{Consequences of the semiclassical observability estimates}
Before the proof of Proposition \ref{prop: semiclassical ob}, we present two direct consequences. \\

The first one is the following small-time semiclassical observability estimate.
\begin{corollary}\label{cor: normal ob inequality}
For any $T>0$, there exists $\rho_0>0$, $h_0>0$ and $C>0$ such that for any $\rho \in (0, \rho_0)$, $h\in(0,h_0)$, we have
\begin{equation}\label{eq: semiclassical ob}
\|\Pi_{h,\rho}u_0\|^2_{L^2(\T^2)}\leq C\int_0^T\int_{\omega}\left|e^{-\ii tH_{\A,V}}\Pi_{h,\rho}u_0(z)\right|^2\d z\d t\qquad \forall u_0 \in L^2(\T^2). 
\end{equation}
\end{corollary}
\begin{proof}[Proof of Corollary \ref{cor: normal ob inequality} from Proposition \ref{prop: semiclassical ob}]
Let $T_0>0$ be the numerical constant of Proposition \ref{prop: semiclassical ob}. Let us apply Proposition \ref{prop: semiclassical ob} with $T=T_0$, it fixes the constant $\rho_0>0$, $h_0>0$ and $C>0$ of Proposition \ref{prop: semiclassical ob}. 

Let $T>0$. We decompose
$$ [0,T] = \bigcup_{k=0}^{N_T-1} I_k,\mbox{ with }I_k:=[kT_0h_{0,T}^{\frac{1}{2}},(k+1)T_0h_{0,T}^{\frac{1}{2}}],$$
for some $h_{0,T} \in (0,h_0)$ depending on $T$. Here $N_T$ is defined so that $N_T=\lfloor h^{-\frac{1}{2}}_{0,T}\rfloor$.

By applying \eqref{eq: sharp ob} for $\rho \in (0, \rho_0)$, $h \in (0,h_{0,T})$, and by using the change of the variable $s=h_{0,T}^{\frac{1}{2}}t$, we have for some $C=C(T_0)>0$,
\[
\|\Pi_{h,\rho}u_0\|^2_{L^2(\T^2)}\leq \frac{C}{h_{0,T}^{\frac{1}{2}}}\int_0^{T_0h_{0,T}^{\frac{1}{2}}}\int_{\omega}\left|e^{-\ii sH_{\A,V}}\Pi_{h,\rho}u_0(z)\right|^2\d z\d s.
\]
We set $u_h(t,z):=e^{-\ii tH_{\A,V}}\Pi_{h,\rho}u_0(z)$. Then, we obtain from the previous estimate that for every $k \in \{0, \dots, N_T-1\}$,
\[
h_{0,T}^{\frac{1}{2}}\|u_h(t=kT_0h_{0,T}^{\frac{1}{2}},\cdot)\|^2_{L^2(\T^2)}\leq C\int_0^{T_0h_{0,T}^{\frac{1}{2}}}\int_{\omega}\left|u_h(t+kT_0h_{0,T}^{\frac{1}{2}},z)\right|^2\d z\d t=C\int_{I_{k}}\int_{\omega}\left|u_h(t,z)\right|^2\d z\d t.
\]
Since $(\ii\pt-H_{\A,V})u_h=0$, we know $\|u_h(t,\cdot)\|_{L^2(\T^2)}=\|u_h(0,\cdot)\|_{L^2(\T^2)}$ by the mass conservation law \eqref{eq:conservationmass}. Consequently, we sum up from $k=0$ to $N_T-1$,
\[
N_Th_{0,T}^{\frac{1}{2}}\|\Pi_{h,\rho}u_0\|^2_{L^2(\T^2)}\leq C\int_0^T\int_{\omega}\left|u_h(t,z)\right|^2\d z\d t.
\]
We conclude that
\[
\|\Pi_{h,\rho}u_0\|^2_{L^2(\T^2)}\leq \frac{C}{N_Th_{0,T}^{\frac{1}{2}}}\int_0^T\int_{\omega}\left|u_h(t,z)\right|^2\d z\d t.
\]
This ends the proof of \eqref{eq: semiclassical ob}.
\end{proof}
By using Corollary \ref{cor: normal ob inequality} and a standard procedure, see for instance \cite{BZ12}, reported in Section \ref{sec:refobsmultid}, we deduce Theorem \ref{thm: main}.\\  

The second consequence of Proposition \ref{prop: semiclassical ob} is the proof of the resolvent estimates from Theorem \ref{cor:stationary}.
\begin{proof}[Proof of Theorem \ref{cor:stationary} from Proposition \ref{prop: semiclassical ob}] Let $C_0>1$ be a positive constant such that $H_{\A,V}+C_0$ is a positive self-adjoint operator. First, from \cite[Proposition 6.6.4]{TW09}, we know that it suffices to prove the resolvent estimate \eqref{eq:resolventestimate} for $|\lambda|$ large enough and the unique continuation property for the eigenfunctions of the operator $H_{\A,V}$. In fact, we have the following unique continuation result \cite[Theorem 5.2]{LLRR} for the operator $H_{\A,V}$, i.e.
$$ \left(H_{\A,V} u = \lambda u\ \text{in}\ \T^2,\ u = 0\ \text{in}\ \omega \right) \Rightarrow u = 0\ \text{in}\ \T^2.$$
Then, one can restrict to prove the resolvent estimate for $|\lambda|$ large enough. Indeed, for $|\lambda| > 1 + C_0$, we then split the proof in two different cases.\\

$\bullet$ Case $\lambda<-1-C_0$. By the choice of $C_0$, $H_{\A,V}+C_0$ is positive self-adjoint. Then, 
\begin{gather*}
\|(H_{\A,V} - \lambda)^{-1}\|_{\mathcal{L}(L^2)}\leq \frac{1}{|\lambda+C_0|}\leq \frac{C_0+2}{|\lambda|+1},\\
\|u_{\lambda}\|_{L^2}=\|(H_{\A,V} - \lambda)^{-1}(H_{\A,V} - \lambda)u_{\lambda}\|_{L^2}\leq \frac{C_0+2}{|\lambda|+1}\|(H_{\A,V} - \lambda)u_{\lambda}\|_{L^2}.
\end{gather*}

$\bullet$ Case $\lambda>1+C_0$. Let us fix $\rho \in (0,\rho_0)$ as in Proposition \ref{prop: semiclassical ob}. Let $h = \sqrt{\lambda}^{-1}$ and we consider the projector $\Pi_{h,\rho}=\chi\left(\frac{\lambda^{-1}H_{\A,V}-1}{\rho}\right)$ be the same as in \eqref{eq: spectral projector}. Then, $u_{\lambda}=\Pi_{h,\rho}u_{\lambda}+(1-\Pi_{h,\rho})u_{\lambda}$. For the first part, we define 
\begin{equation*}
v(t,z):=e^{-\ii t\lambda^{\frac{3}{4}}}\Pi_{h,\rho}u_{\lambda}(z),\;\mbox{ and }f_{\lambda}:=(H_{\A,V}-\lambda)\Pi_{h,\rho}u_{\lambda}. 
\end{equation*}
It is easy to verify that 
\begin{equation*}
v(t)=e^{-\ii \lambda^{-\frac{1}{4}}tH_{\mathbf{A},V}}\Pi_{h, \rho}u_{\lambda}+\ii \int_0^t e^{-\ii (t-s)\lambda^{-\frac{1}{4}}H_{\mathbf{A},V}}\left(\frac{e^{-\ii \lambda^{\frac{3}{4}} s}f_{\lambda }}{\lambda^{\frac{1}{4}}}\right) \d s
\end{equation*}
Thanks to Proposition \ref{prop: semiclassical ob} and the $L^2$-conservation of the group $e^{-\ii t \lambda^{-\frac{1}{4}}H_{\mathbf{A},V}}$, we derive that
\begin{align}
\|\Pi_{h, \rho}u_{\lambda}\|_{L^2}&  \leq C\|e^{-\ii t\lambda^{-\frac{1}{4}}H_{\A,V}}\Pi_{h, \rho}u_{\lambda}\|_{L^2([0,T_0]\times\omega)}\notag\\
&\leq C\|v\|_{L^2([0,T_0]\times\omega)}+\frac{C}{\lambda^{1/4}}\|f_{\lambda }\|_{L^2(\T^2)}\notag\\
&\leq C\left(\frac{1}{\lambda^{\frac{1}{4}}}\|(H_{\A,V} -\lambda) \Pi_{h, \rho}u_{\lambda}\|_{L^2(\mathbb{T}^2)}+\|\Pi_{h, \rho}u_{\lambda}\|_{L^2(\omega)}\right)\notag\\
& \leq C\left(\frac{1}{\lambda^{\frac{1}{4}}}\|\Pi_{h, \rho}(H_{\A,V} -\lambda) u_{\lambda}\|_{L^2(\mathbb{T}^2)}+\|\Pi_{h, \rho}u_{\lambda}\|_{L^2(\omega)}\right)\label{eq: est-1-part}
\end{align}
For the second part, $(1-\Pi_{h, \rho})u_{\lambda}$ satisfies $(H_{\A,V} - \lambda)(1-\widetilde{\Pi}_{\lambda,\rho})(1-\Pi_{h, \rho})u_{\lambda}=(1-\Pi_{h, \rho})(H_{\A,V} - \lambda)u_{\lambda}$, where $\widetilde{\Pi}_{h,\rho}=\widetilde{\chi}\left(\frac{\lambda^{-1}H_{\A,V}-1}{\rho}\right)$ with some $\Tilde{\chi}\in C^{\infty}_c(\R)$ such that $\Tilde{\chi}\chi=\Tilde{\chi}$. Now the operator $(H_{\A,V} - \lambda)(1-\widetilde{\Pi}_{h,\rho})$ is invertible with $\|\left((H_{\A,V} - \lambda)(1-\widetilde{\Pi}_{h,\rho})\right)^{-1}\|\leq \frac{C_{\rho}}{\lambda}$. Then we have 
\begin{multline}
  \|(1-\Pi_{h, \rho})u_{\lambda}\|_{L^2}\leq \|\left((H_{\A,V} - \lambda)(1-\widetilde{\Pi}_{\lambda,\rho})\right)^{-1}(1-\Pi_{h, \rho})(H_{\A,V} - \lambda)u_{\lambda}\|_{L^2}\\
\leq\frac{C}{\lambda}\|(H_{\A,V} - \lambda)(1-\Pi_{h, \rho})u_{\lambda}\|_{L^2} \leq \frac{C}{\lambda}\|(1-\Pi_{h, \rho})(H_{\A,V} - \lambda)u_{\lambda}\|_{L^2} .\label{eq: est-2-part}  
\end{multline}
Therefore, combining the estimates \eqref{eq: est-1-part} and \eqref{eq: est-2-part}, we obtain
\[
\|u_{\lambda}\|_{L^2(\T^2)}\leq C\left( \frac{1}{1+\lambda^{\frac{1}{4}}}\|(H_{\A,V} - \lambda) u_{\lambda}\|_{L^2(\mathbb{T}^2)}+\|u_{\lambda}\|_{L^2(\omega)}\right),
\]
leading to the desired resolvent estimate \eqref{eq:resolventestimate}.
\end{proof}

\section{Beginning of the proof of the semiclassical observability estimate}

The goal of this section is to begin the proof of Proposition \ref{prop: semiclassical ob}.

\subsection{Semiclassical measures}
We turn to the proof of Proposition \ref{prop: semiclassical ob}. 

The proof is based on a contradiction argument. More precisely, suppose that \eqref{eq: sharp ob} is untrue. Since $T_0$ is independent of the parameters $h,\rho$ and the solutions, then for any $T\geq T_0$, there exist sequences 
\begin{align*}
 \rho_n\rightarrow0,\ h_n\rightarrow0,\ v_{0,n}=\Pi_{h_n,\rho_n}u_{0,n}\in L^2(\T^2), \\
(\ii\p_t-h_n^{\frac{1}{2}}H_{\A,V})v_n(t,z)=0,v_n(0,z)=v_{0,n}(z):=\Pi_{h_n,\rho_n}u_{0}(z),
\end{align*}
such that
\begin{equation}\label{eq: sequence hypothesis1}
\|v_{0,n}\|^2_{L^2(\T^2)}=1,
\end{equation}
and 
\begin{equation}
\label{eq: sequence hypothesis2}
    \int_0^{T}\int_{\omega}\left|e^{-\ii th_n^{\frac{1}{2}}H_{\A,V}}v_{0,n}(z)\right|^2\d z\d t\rightarrow0.
\end{equation}
The sequence $(v_n)_{n\in\N^*}$ is bounded in $L^2_{loc}(\R_t\times\T_z^2)$ by the $L^2$-mass conservation law \eqref{eq:conservationmass} and consequently after possibly extracting a subsequence, there exists a semiclassical defect measure $\mu$ on $\R_t\times T^*\T_z^2$ such that for any function $\psi\in L^1(\R_t)$ and any $a\in C^{\infty}_c(T^*\T_z^2)$, we have
\begin{equation}\label{eq: measure-h}
\poscals{\mu}{\psi(t)a(z,\zeta)}=\lim_{n\rightarrow+\infty}\int_{\R_t\times\T_z^2}\psi(t)(\Op_{h_n}(a)v_n)(t,z)\overline{v_n(t,z)}\d z\d t.
\end{equation}
For the semiclassical Weyl quantization on the torus, one can see Appendix \ref{sec:semiclassicaltorus}. For the proof of this existence of semiclassical measure $\mu$, one can refer to \cite{AM14,Macia09,Zwo12,GMMP}.\\

As usual with semiclassical methods, we will drop the index $n$ and rather write $(v_h)_{h\to0^+}$ satisfying that
\begin{equation}\label{eq: original eq-sch}
\left\{
\begin{array}{l}
     (\ii\pt-h^{\frac{1}{2}}H_{\A,V})v_h=0,  \\
     v_h|_{t=0}=v^0_h:=\Pi_{h,\rho}u^0_h, 
\end{array}
\right.
\end{equation}
with the following properties
\begin{equation}\label{eq: sequence hypothesis-h1}
\|v^0_{h}\|^2_{L^2(\T^2)}=1,
\end{equation}
and for $T\geq T_0$,
\begin{equation}
\label{eq: sequence hypothesis-h2}
    \int_0^T\int_{\omega}\left|v_h(t,z)\right|^2\d z\d t=o(1) \mbox{ as }h\to0^+.
\end{equation}
\subsection{First microlocalization}\label{sec: first microlocalization}
In this subsection, we reduce the proof to the analysis of the measure $\mu$ supported in the periodic geodesics. This is accomplished by a standard propagation property of the semi-classical measure $\mu$. More precisely, we have the following result.
\begin{proposition}\label{prop: measure property}
The measure $\mu$ defined by \eqref{eq: measure-h} has the following properties.
\begin{enumerate}
    \item The measure $\mu$ is supported in the set
    \begin{equation}\label{eq: supp of measure-h}
    \supp\mu\subset \{(t,z,\zeta)\in \R_t\times T^*\T^2: |\zeta|=1\}, 
    \end{equation}
    and the wave front set satisfies for a.e. $t\in\mathbb{R}$
    $$ \WF(v_h(t,\cdot))\subset\{(z,\zeta) \in T^*\T^2\ ;\ |\zeta|=1\}.$$
    \item For any $t_0<t_1$, we have
    \begin{equation}
    \label{eq:measureintime}
     \mu((t_0,t_1)\times T^*\T^2)=t_1-t_0.   
    \end{equation}
    \item For a.e. $t\in\mathbb{R}$, 
    \begin{equation}
    \label{eq:supportinvariant}
    \zeta\cdot\nabla_z\mu(t,\cdot)=0.
    \end{equation} In particular, the support of $\mu$ is invariant under the geodesic flow i.e.
    $$(t_0,z_0,\zeta_0)\in\mathrm{supp}(\mu) \Rightarrow (t_0,z_0+s\zeta_0,\zeta_0)\in\mathrm{supp}(\mu)\qquad \forall s\in\mathbb{R}.$$ 
    \item The measure $\mu$ vanishes on $(0,T)\times T^*\omega$.
\end{enumerate}
\end{proposition}
\begin{proof}
The proof of Proposition \ref{prop: measure property} is rather standard, see for instance \cite{AM14,Macia09,Zwo12}, but for the sake of completeness, we include the proof here. \\

For the first statement, we rewrite our operator $h^2H_{\A,V}$ as $P_h=-h^2\Delta_z+hp^{\mathrm{w}}_1(z,hD_z)+h^2p^{\mathrm{w}}_2(z,hD_z)$ (this coincides with \eqref{eq: Ph-p1-p2-intro} in Section \ref{sec: role of magnetic}), where
\begin{equation}\label{eq: p1-p2-defi}
p^{\mathrm{w}}_1(z,hD_z):=\Op_h(A_1\xi+A_2\eta),\;\;p^{\mathrm{w}}_2(z,hD_z):=\Op_h(|\A|^2+V)=|\A|^2+V.
\end{equation}

 By definition, we know the fact that
\[
v_h=\Tilde{\Pi}_{h,\rho}v_h,
\]
where $\Tilde{\Pi}_{h,\rho}=\Tilde{\chi}(\frac{h^2H_{\A,V}-1}{\rho})$ with some $\Tilde{\chi}\in C^{\infty}_c(\R)$ such that $\Tilde{\chi}\chi=\chi$. Thus, on the one hand, we have 
\begin{align*}
(h^2H_{\A,V}-1)v_h=\Tilde{\Pi}_{h,\rho}(h^2H_{\A,V}-1)v_h=\bigO{\rho}{L^2}.
\end{align*}
On the other hand,
\begin{align*}
(h^2H_{\A,V}-1)v_h=&(-h^2\Delta_z-1)v_h +h\left(\Op_h(A_1\xi+A_2\eta)+h^2(|\A|^2+V)\right)v_h\\=&(-h^2\Delta_z-1)v_h +\bigO{h}{L^2}.
\end{align*}
Therefore, we know
\begin{equation}\label{eq: Delta-h-rho-eq}
(-h^2\Delta_z-1)v_h=\bigO{h+\rho}{L^2}.   
\end{equation}
As $\rho,h\rightarrow 0$, we deduce that $(-h^2\Delta-1)v_h=o_{L^2}(1)$, which yields $\supp(\mu)\subset \{ |\zeta|=1\}$. 

The support property of the wave front set is similar to the proof of the first statement, see also \cite[Theorem 8.15]{Zwo12}. We omit the details here.

For the second statement, we take any $\varphi\in C_c^0(\mathbb{R})$ and $a_1(\zeta)\in C_c^{\infty}(\mathbb{R}^2)$ such that $a_1\equiv 1$ in a neighborhood of $|\zeta|=1$. By \eqref{eq: measure-h} and the support property \eqref{eq: supp of measure-h}, we have
$$ \langle\mu,\varphi(t)\rangle=\langle\mu,\varphi(t)a_1(\zeta)\rangle
=\lim_{h\rightarrow 0}\int_{\R}\varphi(t)\Big( \poscals{ (\Op_h(a_1)-1)v_h(t)}{v_h(t)}_{L_z^2}
+\| v_h(t)\|^2_{L_z^2}\Big) dt. 
$$
From the support property \eqref{eq: supp of measure-h}, the first limit on the right hand side is zero. Since $\|v_h(t)\|_{L_z^2}=\|v_h(0)\|_{L_z^2}=1$ by the $L^2$-mass conservation law, we deduce that $\langle\mu,\varphi(t)\rangle=\int_{\R}\varphi(t) dt$. Approximating $\mathbf{1}_{(t_0,t_1)}$ by a sequence of $\varphi_{\epsilon}(t)\in C_c^0(\mathbb{R})$, we obtain \eqref{eq:measureintime}.\\ 

Now we define the Wigner distribution $W_h(t,\cdot)$ on the cotangent bundle $T^*\T^2$ by 
\begin{equation*}
\poscals{W_h(t,\cdot)}{a}:=\int_{\T^2}\left(\Op_h(a)v_h\right)(t,z)\overline{v_h(t,z)}\d z,\qquad \forall a\in C^{\infty}_c(T^*\T^2).
\end{equation*}
Thus,
\begin{align*}
\pt\poscals{W_h(t,\cdot)}{a}&=\int_{\T^2}\left(\Op_h(a)\pt v_h\right)(t,z)\overline{v_h(t,z)}\d z
   +\int_{\T^2}\left(\Op_h(a)v_h\right)(t,z)\overline{\pt v_h(t,z)}\d z\\
   &=-\int_{\T^2}\ii\left(\Op_h(a)H_{\A,V} v_h\right)(t,z)\overline{v_h(t,z)}\d z
   -\int_{\T^2}\left(\Op_h(a)v_h\right)(t,z)\overline{\ii H_{\A,V}v_h(t,z)}\d z\\
   &=-\int_{\T^2}\ii[\Op_h(a),H_{\A,V} ]v_h(t,z)\cdot\overline{v_h(t,z)}\d z.
\end{align*}
Using the symbolic calculus, more precisely the exact commutator formula \eqref{eq: commutators_laplacian}, we obtain
\begin{equation}\label{eq: commutator-formula}
\begin{aligned}
\ii h[\Op_h(a),H_{\A,V}]&=\frac{\ii}{h}[\Op_h(a),-h^2\Delta]+\ii[\Op_h(a),\Op_h(\A\cdot\zeta)]+\ii h[\Op_h(a),|\A|^2+V]\\
&=-\Op_h(2\zeta\cdot\nabla_z a)+\bigO{h}{\mathcal{L}(L^2)}.
\end{aligned}
\end{equation}
For any $\psi\in C^{\infty}_c(\R)$, we have
\begin{align*}
\left|h\int_{\R}\psi(t)\pt\poscals{W_h(t,\cdot)}{a}\d t\right|&=\left|h\int_{\R\times\T^2}\psi'(t)\left(\Op_h(a)v_h\right)(t,z)\overline{v_h(t,z)}\d z\d t\right|\\
&\leq Ch\|\psi'\|_{L^{\infty}}\|a\|_{L^{\infty}}+\mathcal O(h^2).
\end{align*}
Consequently, we obtain
\begin{equation*}
\lim_{h\rightarrow0}\int_{\R_t\times\T^2}\ii h\psi(t)[\Op_h(a),H_{\A,V} ]v_h(t,z)\overline{v_h(t,z)}\d z\d t=0.
\end{equation*}
Thanks to \eqref{eq: commutator-formula}, we have 
\begin{equation*}
0=\lim_{h\rightarrow0}\int_{\R_t\times\T^2}\psi(t)\Op_h(2\zeta\cdot\nabla_z a)v_h(t,z)\overline{v_h(t,z)}\d z\d t=\poscals{\mu}{2\psi(t)\zeta\cdot\nabla_z a}.
\end{equation*}
This implies that $\mu$ solves \eqref{eq:supportinvariant}.
As a consequence, we know that the measure $\mu$ is invariant under the geodesic flow.\\

The last statement comes directly from \eqref{eq: measure-h} and \eqref{eq: sequence hypothesis-h2}. 
\end{proof}

The core of the following proof consists in demonstrating that $$\mu\equiv0.$$
This will be a contradiction with property $\mathrm{(ii)}$.

\subsection{Irrational directions.} In the spirit of \cite{BBZ13,BZ12}, we begin by proving $\supp(\mu)$ does not contain irrational directions. This is due to the fact that the $z-$projection of a trajectory in $T^*\T^2$ associated with an irrational direction is dense (see also \cite{LBM23}). 
We define the set $\Sigma_{\mathbb R \setminus \mathbb Q}$ of irrational directions on the torus $\mathbb T^2$,
\begin{equation*}
    \Sigma_{\mathbb R \setminus \mathbb Q}:=\left\{(z, \zeta)\in T^*\mathbb T^2; \ |\zeta|=1, \, \mathbb Z^2 \cap \{\zeta\}^{\perp}=\{0\}\right\},
\end{equation*}
and $\Sigma_{\mathbb Q}=\{(z, \zeta) \in T^*\T^2; \ |\zeta|=1\}\setminus \Sigma_{\mathbb R \setminus \mathbb Q}$ the set of rational directions. The set $\Sigma_{\mathbb R \setminus \mathbb Q}$ is clearly invariant by the flow
\begin{equation}\label{irr_dir_flow}
\quad (z+s\zeta,\zeta) \in \Sigma_{\mathbb R \setminus \mathbb Q}\qquad \forall (z, \zeta) \in \Sigma_{\mathbb R \setminus \mathbb Q}, \forall s \in \R.
\end{equation}
Let us define $\mu_{\R \setminus \Q}$ to be the restriction of 
the measure $\mu$ to $ \Sigma_{ \R \setminus \Q}$. The following result holds.
\begin{proposition}
We have that $\mu_{\R \setminus \Q}=0$
\end{proposition}
\begin{proof}
Since $\mu$ is 
invariant by the geodesic flow according to $(iii)$, then for any $s\in \R$, 
$$ \mu_{\R\setminus \Q}((t_1, t_2) \times \omega \times \R^2) = \mu_{\R\setminus \Q}  ((t_1, t_2) \times  \Phi_s(\omega \times \R^2)),
$$
where the flow $\Phi_s$ is defined by  $ \Phi_s (z , \zeta)= (z+ s\zeta,
\zeta)$. 
As a consequence, we obtain by Fubini's theorem 
\begin{align*} 
 \mu_{\R\setminus \Q} ((t_1, t_2) \times \omega \times \R^2) & = \frac 1 T \int_0 ^T \mu_{\R\setminus \Q}  ((t_1, t_2) \times  \Phi_s(\omega \times \R^2)) ds \\
& =  \frac 1 T \int_0 ^T \int_{(t_1, t_2) \times\Phi_s(\omega\times \R^2) } d\mu_{\R\setminus \Q}(t,z,\zeta)ds\\
& = \frac 1 T \int_0 ^T \int_{[0,T] \times \T^2 \times \R^2 } 1_{ (t_1, t_2)}(t) 1_{ \Phi_s(\omega\times \R^2)}(z, \zeta)  d\mu_{\R\setminus \Q}(t,z, \zeta) ds\\
&  = \int 1_{ (t_1, t_2)}(t) \times \left(\frac 1 T \int_0 ^T 1_{ \Phi_s(\omega\times \R^2)}(z, \zeta)
  ds \right) d\mu_{\R\setminus \Q}(t,z,\zeta). 
 \end{align*}
The equidistribution theorem shows that for any $(z, \zeta) $ in the support of  $\mu_{\R\setminus \Q}$, 
 $$ \lim_{ T \to \infty } \frac 1 T \int_0 ^T 1_{ \Phi_s(\omega\times \R^2)}(z, \zeta) ds = \frac {| \omega|} {|\T^2|}. $$
Hence the dominated convergence theorem and $(iv)$ show that for $t_1,t_2\in [0,T]$
 \begin{equation}\label{eq.contrainte1}
  \mu_{\R\setminus \Q} ((t_1, t_2) \times \T^2 \times \R^2) = 0,
   \end{equation}
   concluding the proof.
\end{proof}

\subsection{Reduction to a finite number of periodic directions.}

In the same spirit of \cite[Remark 1.4]{BG-stabilization}, we claim that periodic geodesics corresponding to directions $\frac{(n,m)}{\sqrt{n^2+m^2}}$ with $\gcd(n,m)=1$ will also enter the control region $\omega$ as soon as $\sqrt{n^2+m^2}$ is large enough. Therefore, the following result holds. We shall use the following proposition in Section \ref{sec:frequencylocalization} to isolate each of the periodic geodesics not entering the control region.
\begin{proposition}
There exists $p_0\in\N^*$ large enough, for all $|p|\geq p_0$ and $q\in \Z$ with $\gcd(p,q)=1$ and $1\leq |q|<|p|$, such that for any point $(x,y)\in\T^2$, the periodic geodesic $(x,y)+t(p,q)$ enters into the control region $\omega$. In particular, $(x,y,\frac{p}{\sqrt{p^2+q^2}},\frac{q}{\sqrt{p^2+q^2}})\notin\supp(\mu)$.
\end{proposition}
\begin{proof}
As we demonstrated in the previous part, 
only periodic geodesics can be included in $\supp(\mu)$. In this proof, we identify $\T^2=\R^2/\Z^2$ just for the simplicity of the proof. Since $\omega$ is open, by shifting the coordinate axes, we assume that $(a,b)\in (0,1)^2\in\omega$ and there exists $\delta=\delta(a,b,\omega)\in(0,1)$ sufficiently small such that the square $(a-\delta,a+\delta)\times(b-\delta,b+\delta)$ is still contained in $\omega$. 

Now it suffices to prove that 
\begin{gather}
\exists p_0\in\N^* \mbox{ large enough, for all }p\geq p_0\mbox{ and }q\in \Z\mbox{ with } \gcd(p,q)=1, \notag\\
\mbox{ such that the straight line } t(p,q), \mbox{ starting from }(0,0),  \mbox{ enters into }\label{eq: statement}\\
(m_j+a-\delta,m_j+a+\delta)\times(n_j+b-\delta,n_j+b+\delta) \mbox{ for some }(m_j,n_j)_{j\geq1}\in\Z^2.\notag
\end{gather}
This is a direct consequence of the following lemma.
\begin{lemma}\label{lem: finite direction}
There exists $p_0\in\N^*$ large enough, for all $p\geq p_0$ and $q\in \N^*$ with $\gcd(p,q)=1$ and $1\leq q<p$, such that we can find a sequence of integers $(m_j,n_j)_{j\in\N^*}$ such that $\frac{n_j}{m_j}\to\frac{q}{p}$ and for all $j$ large enough, 
\begin{equation}\label{eq: ex of velocity}
\left|\Big(\frac{m_j+a-\delta}{p},\frac{m_j+a+\delta}{p}\Big)\cap \Big(\frac{n_j+b-\delta}{q},\frac{n_j+b+\delta}{q}\Big)\right|\geq \frac{\delta}{p}.
\end{equation}
\end{lemma}
Indeed, \eqref{eq: ex of velocity} shows that there exists $t_0$ such that $(t_0p,t_0q)\in (m_j+a-\delta,m_j+a+\delta)\times(n_j+b-\delta,n_j+b+\delta)$, which implies the statement \eqref{eq: statement} immediately. 
\end{proof}
Therefore, we only need to prove Lemma \ref{lem: finite direction}.
\begin{proof}[Proof of Lemma \ref{lem: finite direction}]
We will realize the choices of $(m_j,n_j)_{j\geq1}$ in the following form
\begin{equation}\label{eq: form-m-n}
m_j=jp+r,\;\;n_j=jq+s,\;\mbox{ with }r,s \mbox{ to be chosen later.}
\end{equation}
Let us denote
\begin{gather*}
I_j:=(\frac{m_j}{p}+\frac{a-\delta}{p},\frac{m_j}{p}+\frac{a+\delta}{p}),\;J_j=(\frac{n_j}{q}+\frac{b-\delta}{q},\frac{n_j}{q}+\frac{b+\delta}{q}),
\end{gather*}
and let $C(I)$ be the center of an interval $I$. Hence,
\begin{equation*}
C(I_j)=\frac{m_j+a}{p},\;C(J_j)=\frac{n_j+b}{q}.
\end{equation*}
Thanks to the form \eqref{eq: form-m-n}, we know that $C(I_j)-C(J_j)=\frac{aq-bp}{pq}+\frac{rq-sp}{pq}$. Since $\gcd(p,q)=1$, there exists an integer pair $(r_q,s_p)\in\Z^2$ such that $r_qq-s_pp=1$. Set $c=\lfloor qa-pb\rfloor$ to be the integer such that $c\leq qa-pb<c+1$ and 
\[
r:=cr_q,\;s:=cs_p.
\]
Then, due to the preceding choice and $r_qq-s_pp=1$, we derive that
\[
|C(I_j)-C(J_j)|=|\frac{aq-bp}{pq}+\frac{cr_qq-cs_pp}{pq}|=\frac{1}{pq}|aq-bp-\lfloor qa-pb\rfloor|<\frac{1}{pq}.
\]
Now we claim that $|I_j\cap J_j|\geq \frac{\delta}{p}$, provided that $p\geq p_0>\frac{1}{\delta}$. \\

Indeed, the length of $I_j$ is $\frac{2\delta}{p}$ and the length of $J_j$ is $\frac{2\delta}{q}$. Choosing that $p_0>\frac{1}{\delta}>1$, for any $p\geq p_0$, we have
\[
\frac{1}{pq}\leq \frac{1}{p_0q}<\frac{\delta}{q}.
\]
Using that $1\leq q<p$, we know that $\frac{\delta}{q}=\max\{\frac{\delta}{p},\frac{\delta}{q}\}$. Since $|C(I_j)-C(J_j)|<\frac{1}{pq}<\frac{\delta}{q}$, we obtain
\[
|I_j\cap J_j|=\frac{|I_j|+|J_j|}{2}-|C(I_j)-C(J_j)|>\frac{\delta}{p}+\frac{\delta}{q}-\frac{\delta}{q}=\frac{\delta}{p},
\]
leading to \eqref{eq: ex of velocity}.
\end{proof}
As a direct consequence, we have the following corollary.
\begin{corollary}\label{cor: finite-direction}
The support of $\mu$ contains only a finite number of periodic directions, i.e. 
$$ \sharp \ \Sigma_{\Q} \cap \left\{(z,\zeta) \in T^* \T^2\ ;\ |\zeta|= 1 , \ (z, \zeta) \in \supp(\mu) \right\} < + \infty.$$    
\end{corollary}
\begin{proof}
Equipped with the preceding lemma, we know that for all $|p|\geq p_0$ and $q\in \Z$ with $\gcd(p,q)=1$ and $1\leq |q|<|p|$, we have $(x,y,\frac{p}{\sqrt{p^2+q^2}},\frac{q}{\sqrt{p^2+q^2}})\notin\supp(\mu)$ for any $(x,y)\in\T^2$. Since $p$ and $q$ are symmetric in the setup, we deduce that
$$
\sharp \ \Sigma_{\Q} \cap \left\{(z,\zeta) \in T^* \T^2\ ;\ |\zeta|= 1 , \ (z, \zeta) \in \supp(\mu) \right\}\leq 2p_0<+\infty.
$$
This ends the proof.
\end{proof}

Let us now define the measure $\mu_T \in \mathcal M_+(T^*\mathbb T^2)$ by
$$\mu_T(dz, d\zeta)=\int_0^T \mu(t, dz, d\zeta) dt.$$
Consequently, $\mu_T(\Sigma_{\mathbb Q})=T>0$ and since
$\left\{\zeta \in \mathbb S^1\ ; \ \exists z\in \mathbb T^2, (z, \zeta) \in \Sigma_{\mathbb Q} \right\}$
is a countable set, there exists $\zeta_0 \in \R^2$ such that
$$\mu_T(\mathbb T^2 \times \{\zeta_0\})>0 \quad \text{and} \quad \zeta_0=\frac{(n,m)}{\sqrt{n^2+m^2}} \quad \text{for some}\quad (n,m) \in \mathbb Z^2.$$
Moreover, inspired by \cite{BBZ13,BZ12}, in order to perform a finer analysis near the periodic geodesics, we apply a change of variables. Similar techniques can also be found in \cite{Sun-dampedwave,LBM23}. 

\subsection{A change of variable}
This step is devoted to showing that, up to a change of variables, we can assume that $\zeta_0=(1,0)$. 
Let $F : \R^2 \longrightarrow \R^2$ be the isometry defined by $$\forall (x, y) \in \R^2, \quad F(x, y)= x \zeta_0^{\perp} +y \zeta_0,$$
where $\zeta_0^{\perp}=\frac{(-m,n)}{\sqrt{n^2+m^2}}$.
One can readily verify that for any function $u$ periodic with respect to $(2\pi \mathbb Z)^2$, the function $F^*u$ is periodic with respect to $(M \mathbb Z)^2$, with $M=2\pi\sqrt{n^2+m^2}$. Moreover, if $u(t,\cdot)$ is solution to the Schr\"odinger equation \eqref{eq:SchrodingerObs} then, $v(t,\cdot)=F^*u(t, \cdot) = u(t,F^{-1}(\cdot))$ is solution to the Schr\"odinger equation posed on $\mathbb R^2/(M \mathbb Z)^2$
\begin{equation}\label{eq:SchrodingerEq_DilatedTorus}
		\left\{
			\begin{array}{ll}
				i  \partial_t v  = H_{\mathbf{F^* A}, F^* V}  v & \text{ in }  (0,T) \times \mathbb R^2/(M \mathbb Z)^2, \\
				v(0, \cdot) = F^*u_0 & \text{ in } \mathbb R^2/(M \mathbb Z)^2.
			\end{array}
		\right.
\end{equation}
As a consequence, if we define $w_n= F^* v_n$ then, 
$$\forall t \in \R, \quad F^*\left(e^{-it H_{\A,V}} v_n\right)= e^{-it H_{\mathbf{F^* \A}, F^* V} } w_n.$$
In the following, the new torus $\mathbb R^2/(M \mathbb Z)^2$ is still denoted $\mathbb T^2$. 
Up to a subsequence, associated to this family of solutions of the Schr\"odinger equation \eqref{eq:SchrodingerEq_DilatedTorus} with initial data $(w_n)_{n \in \N}$ is a semiclassical defect measure $\nu \in L^{\infty}(\mathbb R, \mathcal M_+(T^*\mathbb T^2))$ satisfying:
\begin{itemize}
    \item  for all $t_0,t_1 \in \R,\ \nu([t_0,t_1]\times \mathbb T^2 \times \R^2)=M^2|t_1-t_0|$ and $\mathrm{Supp}\  \nu \subset \R \times \mathbb T^2 \times \mathbb S^1$,
    \item $\nu_T(\mathbb T^2\times \{(1,0)\})>0$, with $\nu_T(dz, d\zeta)=\int_0^T \nu(t, dz, d\zeta)dt$.
\end{itemize}
In other words, we are in the same situation as in the second step, with $\zeta_0=(1,0)$. For the sake of conciseness, we keep in the remainder of the proof the notations adopted in the second step and assume that $\zeta_0=(1,0)$.

\section{Normal form reduction}\label{sec: normal form reduction}
Recall that $\zeta_0=(1,0)$ and our goal now is to prove that $\mu\mathbf{1}_{\zeta=\zeta_0}=0$. This task is divided into three steps: the current section can be seen as a preparation for later proof and aims to provide a suitable normal form for the equation \eqref{eq: original eq-sch}. Section \ref{sec:frequencylocalization} establishes general facts about frequencyy localization. Section \ref{sec: transversal HF} deals with the transversal high-frequency part, while Section \ref{sec: transversal LF} is devoted to dealing with the transversal low-frequency part. Combining these four sections, we conclude that $\mu\mathbf{1}_{\zeta=\zeta_0}=0$, which completes the proof of Proposition \ref{prop: semiclassical ob}. 

\vspace{1em}
Recall that $v_h$ satisfies the equation \eqref{eq: original eq-sch}, the normalization \eqref{eq: sequence hypothesis-h1} and the estimate \eqref{eq: sequence hypothesis-h2}, equivalently,
\begin{equation}
(\ii h^{\frac{3}{2}}\pt-h^2H_{\A,V})v_h=0, \quad v_h|_{t=0}=v^0_h:=\Pi_{h,\rho}u^0_h,\label{eq: original eq-normal-form}
\end{equation}
\begin{equation}
    \|v^0_{h}\|^2_{L^2(\T^2)}=1, \label{eq: sequence hypothesis-h1Bis}
\end{equation}
\begin{equation}
  \int_0^T\int_{\omega}\left|v_h(z)\right|^2\d z\d t=o(1) \mbox{ as }h\to0^+. 
  \label{eq: sequence hypothesis-h2Bis}
\end{equation}
This section is devoted to simplify the equation \eqref{eq: original eq-normal-form} through two successive normal form reductions. 
\begin{itemize}
    \item The first reduction, i.e., the well-known gauge transformation for the magnetic potential $\A=(A_1,A_2)$, removes the $x-$dependence of $A_1$, replacing $A_1(x,y)$ by its average $\langle A_1\rangle(y)$ along the direction $\zeta_0=(1,0)$ (see Proposition \ref{prop: 1-average} below). 
    \item The second reduction aims to remove the $x-$dependence of $\vnh{A}{1}_2$, also replacing $\vnh{A}{1}_2(x,y)$ by its average $\langle \vnh{A}{1}_2\rangle(y)$ along the direction $\zeta_0=(1,0)$, where $\vnh{A}{1}_2$ is a compensated potential after the first averaging process (see the details in Proposition \ref{prop: 2-average} below). 
\end{itemize}
\subsection{First averaging: gauge transformation}
The following result is our first averaging process.
\begin{proposition}\label{prop: 1-average}
Let us define the gauge transformation
\begin{equation}
\label{eq:gaugetransformation1}
  v^{(1)}_h:=e^{\ii g_1}v_h,\ \text{where}\ g_1(x,y)=\int_{-\pi}^x\left(\langle A_1\rangle(y)-A_1(s,y)\right)\d s\qquad (x,y) \in \T^2.  
\end{equation} 
Then, $v^{(1)}_h$ satisfies the following equation
\begin{equation}\label{eq: eq-v-1}
(\ii h^{\frac{3}{2}}\p_t-P^{(1)}_h-Q^{(1)}_h)\vh{v}{1}=0,
\end{equation}
where $\vh{P}{1}$ and $\vh{Q}{1}$ are defined as follows:
\begin{align}
\vh{P}{1}&:=h^2D_x^2+h^2D_y^2-2h\langle A_1\rangle(y)hD_x,\label{eq:defP1}\\
\vh{Q}{1}&:=-hA^{(1)}_2(x,y)hD_y-h^2D_y(A^{(1)}_2(x,y))+h^2V^{(1)}(x,y),\label{eq:defQ1}\\
A^{(1)}_2(x,y)&:=A_2(x,y)+\p_yg_1(x,y),\;V^{(1)}(x,y)=|\A(x,y)+\nabla g_1(x,y)|^2+V(x,y). \label{eq:defMagnetic1Electric1}
\end{align}
Moreover, the following properties hold for $v^{(1)}_h$.
\begin{enumerate}
    \item For every $t \geq 0$, $\|\vh{v}{1}(t,\cdot)\|_{L^2(\T^2)}=\|v_h(t,\cdot)\|_{L^2(\T^2)}$. In particular, $\|\vh{v}{1}|_{t=0}\|_{L^2(\T^2)}=1$.
    \item We have $\|\vh{v}{1}\|^2_{L^2((0,T)\times\omega)}=o(1)$.
    \item The wave front set is conserved, i.e. $\WFm{m}(\vh{v}{1})=\WFm{m}(v_h)$ for every $m \in \N$.
\end{enumerate}    
\end{proposition}
\begin{proof}
We consider the conjugate operator
\begin{align*}
e^{\ii g_1}(\ii h^{\frac{3}{2}}\pt-h^2H_{\A,V})e^{-\ii g_1}&=\ii h^{\frac{3}{2}}\pt+h^2e^{\ii g_1}(\Delta-\ii (\A\cdot\nabla+\nabla\cdot\A))e^{-\ii g_1}-h^2(|\A|^2+V)\\
&=\ii h^{\frac{3}{2}}\pt-h^2(\frac{1}{\ii}\nabla-(\A+\nabla g_1))^2 - h^2 V.
\end{align*}
Thanks to the definition of $g_1$ in \eqref{eq:gaugetransformation1}, we obtain
\begin{align*}
e^{\ii g_1}(\ii h^{\frac{3}{2}}\pt-h^2H_{\A,V})e^{-\ii g_1}&=\ii h^{\frac{3}{2}}\pt-\vh{P}{1}-\vh{Q}{1},
\end{align*}
where $\vh{P}{1}$, $\vh{Q}{1}$ are defined in \eqref{eq:defP1}, \eqref{eq:defQ1} and \eqref{eq:defMagnetic1Electric1}.
Using the definition of $\vh{v}{1}$ and \eqref{eq: original eq-normal-form}, we derive
\[
e^{\ii g_1}(\ii h^{\frac{3}{2}}\pt-h^2H_{\A,V})e^{-\ii g_1}\vh{v}{1}=e^{\ii g_1}(\ii h^{\frac{3}{2}}\pt-h^2H_{\A,V})v_h=0,
\]
so \eqref{eq: eq-v-1} holds. 

Since $e^{\ii g_1}\in C^{\infty}(\T^2;\C)$ and $|e^{\ii g_1}|=1$, (i), (ii) and (iii) are immediate from \eqref{eq: sequence hypothesis-h1Bis} and \eqref{eq: sequence hypothesis-h2Bis}.
\end{proof}

The gauge transform $v_h\mapsto \vh{v}{1}=e^{\ii g_1}v_h$ allows us to assume that $\A$ has the form
\[
\A^{(1)}(x,y)=(\langle A_1\rangle(y),A^{(1)}_2(x,y)).
\]

By using the last two points of Proposition \ref{prop: measure property}, in the following, without loss of generality, we are going to consider the model case $\widetilde{\omega}=\mathbb{T}_x\times\omega_y$, where $\omega_y\subset \mathbb{T}_y$ is an open set containing all critical points of $\langle A_1\rangle(y)$. First, we observe that there exists an open subset $\widetilde{\omega}_y$ of $\omega_y$, consisting of finitely many strips, such that all the critical points of $\langle A_1\rangle(y)$ are contained in this open subset. By shifting the $y$-coordinate, we may assume that $$\widetilde{\omega}_y=\bigcup_{j=1}^{n_0}I_j,\quad I_j=(a_j,b_j), \mod 2\pi,
$$  
where $-\pi<a_1<b_1<a_2<b_2<\cdots<a_{n_0}<\pi<b_{n_0}<a_1+2\pi$. Set $\delta_0<\frac{1}{2}\min(|I_1|,\cdots,|I_{n_0}|)$ such that any critical point of $\langle A_1\rangle(y)$ is contained in one of $\widetilde{I}_j=(a_j+\delta_0,b_j-\delta_0)$ mod $2\pi$.   
\begin{figure}[h]
\begin{minipage}[t]{0.7\linewidth}
    \centering
    \includegraphics[width=\linewidth]{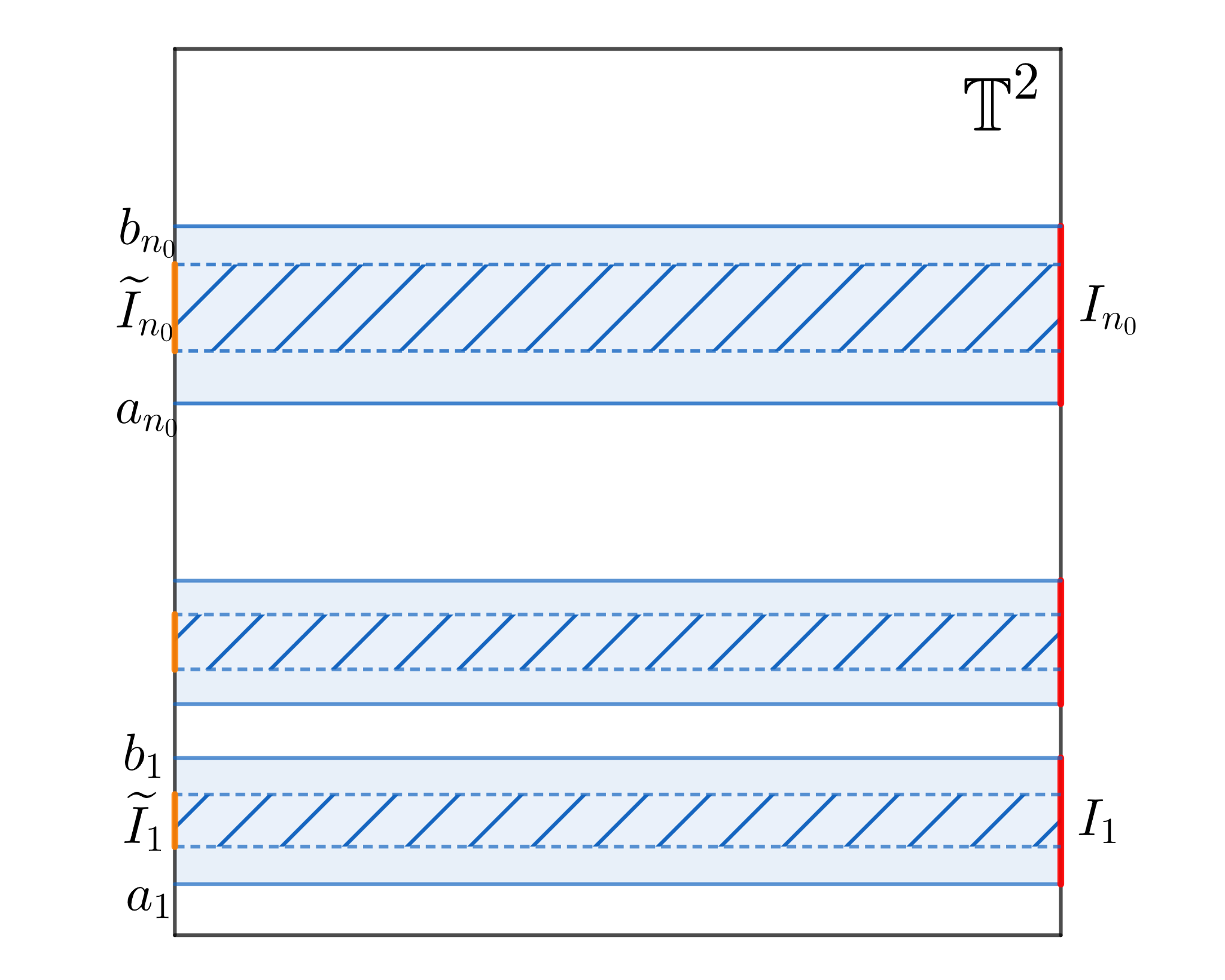}
    \caption{Multiple strips}
  \end{minipage}
  \hfill
  \begin{minipage}[t]{0.2\linewidth}
    \includegraphics[width=\linewidth]{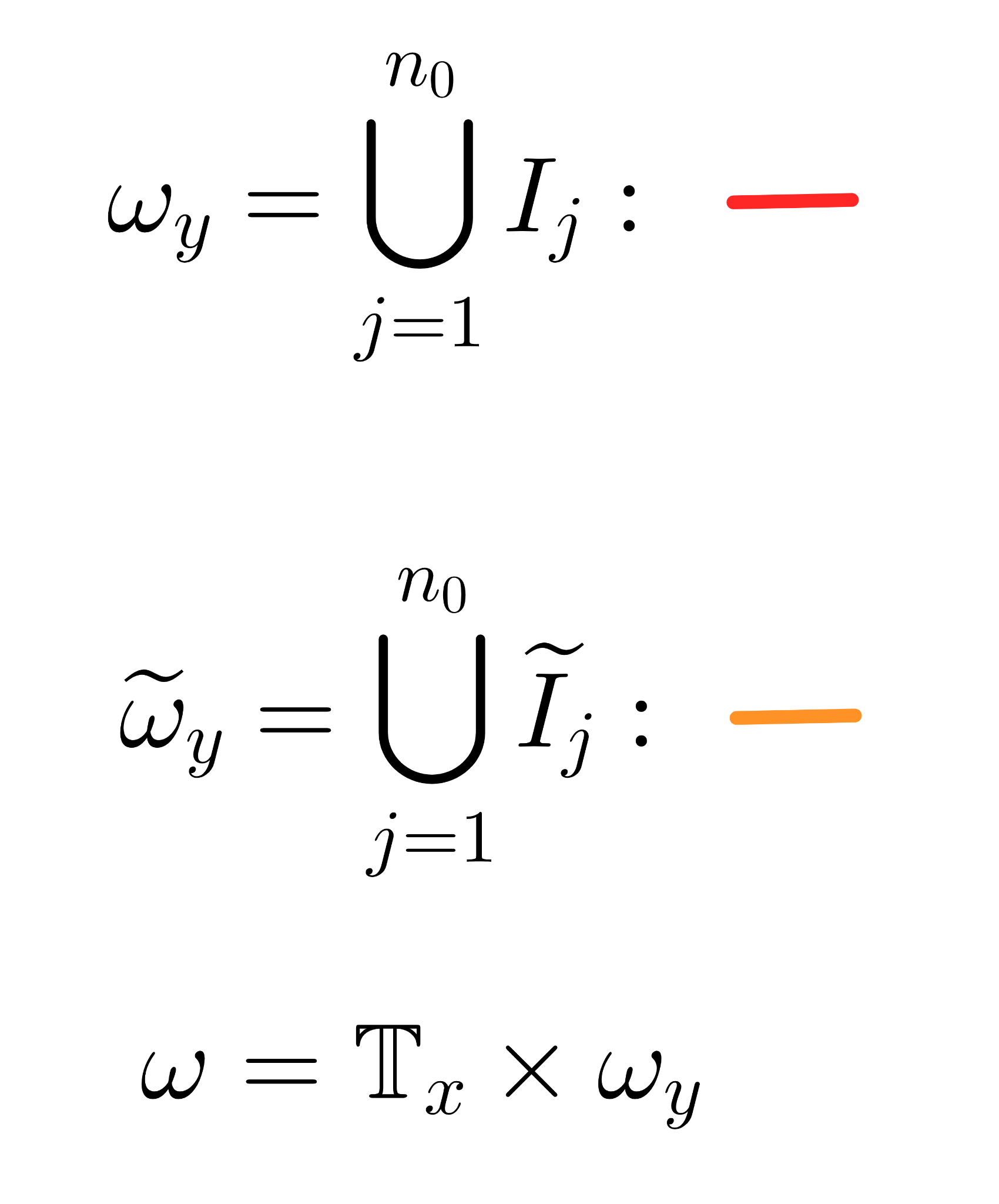}
  \end{minipage}
\end{figure}

Let $\chi^{(1)}\in C^{\infty}(\T_y;[0,1])$ be such that $\chi^{(1)}=1$ on $\bigcup_{j=1}^{n_0}\Tilde{I}_j$ and $\supp(\chi^{(1)})\subset\bigcup_{j=1}^{n_0}I_j$. The following result is an easy consequence of (ii) of Proposition \ref{prop: 1-average}.
\begin{corollary}
\label{cor:supportchi1}
We have $\|\chi^{(1)}\vh{v}{1}\|^2_{L^2((0,T)\times\T^2)}=o(1)$.
\end{corollary}
In prevision of the next step, we define  $\chi^{(2)}\in C^{\infty}(\T_y;[0,1])$ such that $\chi^{(2)}=1$ on $\bigcup_{j=1}^{n_0}\Tilde{I}_j$ and $\supp(\chi^{(2)})\subset\{\chi^{(1)}=1\}$.

\subsection{Second averaging}

In the sequel, we fix the anti-derivative $\partial_x^{-1}$ acting on functions $f$ such that $\int_{-\pi}^{\pi}f(x,y)\d x=0$ by
$$ \partial_x^{-1}f(x,y):=\int_{-\pi}^x f(x',y)\d x' -\frac{1}{2\pi}\int_{-\pi}^{\pi}\d x'\int_{\pi}^{x'} f(x'',y)\d x''.
$$
In particular, when taking the Fourier transform in $x$, we have
$$ \mathcal{F}_x(\partial_x^{-1}f)(k,y)=\frac{\mathcal{F}_x f(k,y)}{\ii k}.
$$

We introduce two cut-off functions
\begin{gather}
\psi\in C^{\infty}_c(\R),\mbox{ supported on }|\xi\pm1|\leq\frac{1}{8}, \mbox{ and }\psi(\xi)=1 \mbox{ on }|\xi\pm1|\leq\frac{1}{16},\label{eq: defi-psi-tilde}\\
\vartheta\in C^{\infty}_c(\R),\mbox{ supported on }|\eta|\leq2, \mbox{ and }\vartheta(\eta)=1 \mbox{ on }|\eta|\leq 1,\label{eq: defi-vartheta}
\end{gather}
and we set $\alpha$ as a parameter that will be fixed later. Note that here we localize near $|\xi|\sim1$. However,  the periodic orbits generated by $\zeta_0=(1,0)$ and $-\zeta_0=(-1,0)$ are the same. Hence, it can be the same as the localization near $\xi\sim1$.

Then, we are in a position to introduce the second averaging process.
\begin{proposition}\label{prop: 2-average}
Let us define the following transformation
\begin{equation}
\label{eq:transformation2}
  \vh{v}{2}:=e^{G_2}\vh{v}{1},\ G_2:=\Op_h(g_2(x,y,\xi)\vartheta(\eta)\eta),
\end{equation} 
where
\begin{equation}
    \label{eq:transformation2Bis}
    g_2(x,y,\xi):=\frac{\psi(\xi)}{2\ii \xi}\partial_x^{-1}(\langle A_2^{(1)}\rangle-A_2^{(1)})(x,y)\qquad (x,y,\xi) \in \T^2\times \R.  
\end{equation}
Then $v^{(2)}_h$ satisfies the following equation
\begin{equation}\label{eq: eq-v-2}
(\ii h^{\frac{3}{2}}\pt-\vh{P}{2}+\vnh{R}{2}_{1+2\alpha}+R^{(2)}_{2+\alpha}+R_h)\vh{v}{2}=0,   
\end{equation}
where
\begin{align}
\vh{P}{2}(y,D_x,D_y)&=h^2D_x^2+h^2D_y^2-2h\langle A_1\rangle(y)hD_x-2h\Op_h(\langle A^{(1)}_2\rangle(y)\eta), \label{eq:defP2}\\ 
\vnh{R}{2}_{1+2\alpha}(x,y,D_x,D_y)&=h\Op_h(r^{(1)}_{1+2\alpha}\eta^2),\qquad r^{(1)}_{1+2\alpha}=2\ii \p_yg_2(x,y,\xi)\vartheta(\eta),\label{eq:defR2alpha}\\
R^{(2)}_{2+\alpha}(x,y,D_x,D_y)&=-h\Op_h\left((1-\psi(\xi))\left(\langle A_2^{(1)}\rangle(y)-A_2^{(1)}(x,y)\right)\eta\right)\label{eq:defRalpha}\\
&\quad +h\Op_h\left(\ii\{\xi^2,g_2(x,y,\xi)(1-\vartheta(\eta))\eta\}\right),\notag\\
R_h&=\bigO{h^2}{\mathcal{L}(L^2)},\label{eq:defreminder2}
\end{align}
Moreover, the following properties hold for $\vh{v}{2}$.
\begin{enumerate}
    \item There exists a constant $C>0$ such that for every $t \geq 0$, $C^{-1}\|v_h\|_{L^2(\T^2)}\leq\|\vh{v}{2}\|_{L^2(\T^2)}\leq C\|v_h\|_{L^2(\T^2)}$. In particular, $C^{-1} \leq \|\vh{v}{2}|_{t=0}\|_{L^2(\T^2)} \leq C$.
    \item  We have $\|\chi^{(2)}\vh{v}{2}\|^2_{L^2((0,T)\times\T^2)}=o(1)$.
    \item The wave front set is preserved $\WFm{m}(\vh{v}{2})=\WFm{m}(v_h)$ for any $m\in\N$.
\end{enumerate}
\end{proposition}
In the following, to simplify the notation, we shall use 
\begin{equation}
    R_h=\bigO{h^2}{\mathcal{L}(L^2)}\mbox{ to denote operators of size at most }\BigO{h^2}.\label{eq: R-h}
\end{equation}
The operators $R_h$ can change from line to line. Here we use notations $R^{(2)}_{1+2\alpha}$ and $R^{(2)}_{2+\alpha}$ without introducing any terms involving the parameter $\alpha$. Indeed, at this point, the sub-indices have no specific meaning. However, we shall see later in Section \ref{sec: transversal HF} and Section \ref{sec: transversal LF}, in particular, Lemma \ref{lem: source-est-hf}, after taking suitable cutoffs $\vartheta(h^{1-\alpha}D_y)$, the term $R^{(2)}_{1+2\alpha}$ will behave like $\bigO{h^{1+2\alpha}}{L^2}$, while $R^{(2)}_{2+\alpha}$ will behave like $\bigO{h^{\infty}}{L^2}$. So we denote these terms by $R^{(2)}_{1+2\alpha}$ and $R^{(2)}_{2+\alpha}$, respectively.

\begin{proof}[Proof]
From Lemma \ref{lem: exponential of bounded op}, we have the following expansion
\begin{multline*}
e^{G_2}(\ii h^{\frac{3}{2}}\pt-\vh{P}{1}-\vh{Q}{1})e^{-G_2}\\
=\ii h^{\frac{3}{2}}\pt-\vh{P}{1}-\vh{Q}{1}-[G_2,\vh{P}{1}+\vh{Q}{1}]+ \frac{1}{2}\int_0^1(1-s)^2e^{sG_h}\ad^2_{G_h}(-\vh{P}{1}-\vh{Q}{1})e^{-sG_h}\d s
\end{multline*}
By the standard semi-classical symbol calculus, we have
\[
\|\frac{1}{2}\int_0^1(1-s)^2e^{sG_h}\ad^2_{G_h}(-\vh{P}{1}-\vh{Q}{1})e^{-sG_h}\d s\|_{\mathcal{L}(L^2)}=\BigO{h^2}.
\]
By recalling the definitions of $\vh{P}{1}$ and $\vh{Q}{1}$ in \eqref{eq:defP1} and in \eqref{eq:defQ1}, 
we have
\begin{align*}
[G_2,\vh{P}{1}]&=[G_2,h^2D_x^2+h^2D_y^2]- h[G_2,2\langle A_1\rangle(y)hD_x]\\
&=\ii h\Op_h(\{\xi^2+\eta^2,g_2\vartheta(\eta)\eta \})-\ii h^2\Op_h(\{
2\langle A_1\rangle(y)\xi,g_2\vartheta(\eta)\eta
\} )+\mathcal{O}_{\mathcal{L}(L^2)}(h^3)  \\
&=\ii h\Op_h(\{\xi^2+\eta^2,g_2\vartheta(\eta)\eta\})+R_h,\\
[G_2,\vh{Q}{1}]&=\ii h^2\Op_h(\{g_2\vartheta(\eta)\eta,2A_2^{(1)}\eta\})+\bigO{h^3}{\mathcal{L}(L^2)}=R_h.
\end{align*}
Using the definition of the Poisson bracket, we simplify the first commutator term $[G_2,\vh{P}{1}+\vh{Q}{1}]$ by
\[
[G_2,\vh{P}{1}+\vh{Q}{1}]=h\Op_h(\ii\{\xi^2,g_2\vartheta(\eta)\eta\})+h\Op_h(r^{(1)}_{1+2\alpha}\eta^2)+R_h,
\]
where $r^{(1)}_{1+2\alpha}(x,y,\xi)=2\ii \p_yg_2(x,y,\xi)\vartheta(\eta)$. Then, we obtain a simplified form for the conjugate operator
\begin{align*}
e^{G_2}(\ii h^{\frac{3}{2}}\pt-\vh{P}{1}-\vh{Q}{1})e^{-G_2}&=\ii h^{\frac{3}{2}}\pt-\vh{P}{1}-h\Op_h\left(\ii\{\xi^2,g_2\eta\}+A_2^{(1)}(x,y)\eta\right)\\
&+h\Op_h(r^{(1)}_{1+2\alpha}\eta^2)
+h\Op_h\left(\ii\{\xi^2,g_2(x,y,\xi)(1-\vartheta(\eta))\eta\}\right)+R_h.
\end{align*}
By the definition of $g_2$ in \eqref{eq:transformation2Bis} we derive that
\begin{align*}
\ii\{\xi^2,g_2\eta\}+A_2^{(1)}\eta&=2\ii \xi\frac{\psi(\xi)}{2\ii \xi}\left(\langle A_2^{(1)}\rangle(y)-A_2^{(1)}(x,y)\right)\eta+A_2^{(1)}(x,y)\eta\\
&=(\psi(\xi)-1)\left(\langle A_2^{(1)}\rangle(y)-A_2^{(1)}(x,y)\right)\eta+\langle A_2^{(1)}\rangle(y)\eta.
\end{align*}

Therefore, we obtain
\[
e^{G_2}(\ii h^{\frac{3}{2}}\pt-\vh{P}{1}-\vh{Q}{1})e^{-G_2}=\ii h^{\frac{3}{2}}\pt-\vh{P}{2}+R^{(1)}_{1+2\alpha}+R^{(2)}_{2+\alpha}+R_h,
\]
where $\vh{P}{2}$, $\vnh{R}{2}_{1+2\alpha}$ and $R^{(2)}_{2+\alpha}$ are defined in \eqref{eq:defP2}, \eqref{eq:defR2alpha} and \eqref{eq:defRalpha}.
By using the definition of $\vh{v}{2}$ in \eqref{eq:transformation2} and the equation satisfied by $\vh{v}{1}$ in \eqref{eq: eq-v-1}, we have 
\[
e^{G_2}(\ii h^{\frac{3}{2}}\pt-\vh{P}{1}-\vh{Q}{1})e^{-G_2}\vh{v}{2}=e^{G_2}(\ii h^{\frac{3}{2}}\pt-\vh{P}{1}-\vh{Q}{1})\vh{v}{1}=0,
\]
which implies that $\vh{v}{2}$ satisfies \eqref{eq: eq-v-2}.\\

The point (i) directly comes from $\|e^{\pm G_2}\|_{\mathcal{L}(L^2)}\leq e^{\|G_2\|_{\mathcal{L}(L^2)}}$ and the point (i) of Proposition \ref{prop: 1-average}.

For the point (ii), we have 
\[
\|\chi^{(2)}\vh{v}{2}\|_{L^2((0,T)\times\T^2)}\leq \|e^{G_2}\left(e^{-G_2}\chi^{(2)} e^{G_2}\right)\vh{v}{1}\|_{L^2((0,T)\times\T^2)}.
\]
By Lemma \ref{lem: exponential of bounded op}, we have the Taylor expansion $$e^{-G_2}\chi^{(2)} e^{G_2}\vh{v}{1}=\chi^{(2)} \vh{v}{1}+\int_0^1e^{-sG_2}\underbrace{[\chi^{(2)},G_2]}_{\bigO{h}{\mathcal{L}(L^2)}} e^{sG_2}\vh{v}{1}\d s,$$ hence
\[
\|e^{G_2}\left(e^{-G_2}\chi^{(2)} e^{G_2}\right)\vh{v}{1}\|_{L^2((0,T)\times\T^2)}\leq C\left(\|\chi^{(1)}\vh{v}{1}\|_{L^2((0,T)\times\T^2)}+h\|\vh{v}{1}\|_{L^2((0,T)\times\T^2)}\right). 
\]
Therefore, (ii) is a direct consequence of Corollary \ref{cor:supportchi1}.

The last statement (iii) is a direct consequence of Lemma \ref{lem: wave front invariance} and the point (iii) of Proposition \ref{prop: 1-average}.
\end{proof}

\section{Frequency localization}
\label{sec:frequencylocalization}

Recall that the cut-off function $\vartheta$ is defined in \eqref{eq: defi-vartheta}. Since the support of $\mu$ contains only a finite number of periodic geodesics and we reduce the problem to a model case $\zeta_0=(1,0)$ by a change of variables, we know that $\zeta_0$ is isolated. Therefore, we can choose a small parameter $\epsilon\in (0,\frac{1}{2})$ to define a frequency cut-off $\vartheta(\frac{hD_y}{\epsilon})$. Again, the sequence $(\vartheta(\frac{hD_y}{\epsilon})v_h)_{h>0}$ is bounded in $L^2_{loc}(\R_t\times\T_z^2)$ by the $L^2$-mass conservation law \eqref{eq:conservationmass} and consequently after possibly extracting a subsequence, there exists a semiclassical defect measure $\mu_{\varepsilon}$ on $\R_t\times T^*\T_z^2$ such that for any function $\psi\in L^1(\R_t)$ and any $a\in C^{\infty}_c(T^*\T_z^2)$, we have
\begin{equation}\label{eq: measure-hSecond}
\poscals{\mu_{\varepsilon}}{\psi(t)a(z,\zeta)}=\lim_{n\rightarrow+\infty}\int_{\R_t\times\T_z^2}\psi(t)\left(\Op_{h}(a)\vartheta(\frac{hD_y}{\epsilon})v_h\right)(t,z)\overline{\left(\vartheta(\frac{hD_y}{\epsilon})v_h\right)(t,z)}\d z\d t.
\end{equation}

We know that $\mu_{\epsilon}=|\vartheta(\frac{\eta}{\epsilon})|^2\mu$. From now on, we fix $\epsilon_0$ sufficiently small such that 
\[
\supp(\mu_{\epsilon_0})=\supp(|\vartheta(\frac{\eta}{\epsilon_0})|^2\mu)\subset \{\zeta_0=(1,0)\}.
\] 
We microlocalize the solution $v_h$ near the direction $\zeta_0=(1,0)$ for some $\alpha\in(\frac{1}{4},\frac{1}{2})$, we introduce
\begin{equation}
    \label{eq:defwh}
w_h=\vartheta(h^{1-\alpha}D_y)\vartheta(\frac{hD_y}{\epsilon_0})\vh{v}{2}.
\end{equation}
The range of $\alpha$ is according to the later proof (See Lemma \ref{lem: source-est-hf} and Remark \ref{rem: index-alpha}). \\

Since we focus on the limit process $h\to0$, we can ask that $h^{\frac{1}{4}}<\frac{\epsilon_0}{4}$. Then, we obtain that $$w_h=\vartheta(h^{1-\alpha}D_y)\vartheta(\frac{hD_y}{\epsilon_0})\vh{v}{2}=\vartheta(h^{1-\alpha}D_y)\vh{v}{2},$$ 
for free due to the support of $\vartheta$. \\

In the following, to simplify the notation, we shall use 
\begin{equation}
    r_h=o(h^{3/2})_{\mathcal{L}(L^2)}\mbox{ to denote operators of size at most }o(h^{3/2}).\label{eq: r-h}
\end{equation}
The operators $r_h$ can change from line to line.

Our first result concerns the equation for the frequency-localized solution $w_h$. 
\begin{proposition}\label{lem: f-h-error}
The function $w_h$ defined in \eqref{eq:defwh} satisfies
\begin{equation}\label{eq: w-h-eq}
\left\{
\begin{array}{l}
     (\ii h^{\frac{3}{2}}\pt-\vh{P}{2}+R^{(2)}_{1+2\alpha}+R^{(2)}_{2+\alpha}+R_h)w_h \\=[\vartheta(h^{1-\alpha}D_y),\vh{P}{2}-R^{(2)}_{1+2\alpha}-R^{(2)}_{2+\alpha}-R_h]\vh{v}{2}=:f_h\\
     w_h|_{t=0}=w^0_h:=\vartheta(h^{1-\alpha}D_y)\vh{v}{2}(0,\cdot),
\end{array}
\right.
\end{equation}
Moreover, we have
\begin{equation}
    \label{eq:boundfh}
    f_h=\bigO{h^{2-\alpha}}{L^2}.
\end{equation}
\end{proposition}
\begin{proof}

The equation satisfied by $w_h$ directly comes from \eqref{eq: eq-v-2}.

We then start from the definition of $\vh{P}{2}$ in \eqref{eq:defP2}. We have that $\vartheta(h^{1-\alpha}D_y)$ commutes with $h^2D_x^2+h^2D_y^2$. So, from the semiclassical calculus, we have
\begin{align*}
[\vartheta(h^{1-\alpha}D_y),\vh{P}{2}]&=-h[\vartheta(h^{1-\alpha}D_y),\Op_h(\langle A_1\rangle(y)\xi+\langle A_2^{(1)}\rangle(y)\eta)]\\
&=h^{2-\alpha}\Op_h(\frac{1}{\ii}\vartheta'(h^{-\alpha} \eta)\left(\langle A_1\rangle'(y)\xi+\langle A_2^{(1)}\rangle'(y)\eta\right))+\bigO{h^{3-2\alpha}}{\mathcal{L}(L^2)}\\
&=\bigO{h^{2-\alpha}}{\mathcal{L}(L^2)}.
\end{align*}
Then, $[\vartheta(h^{1-\alpha}D_y),\vh{P}{2}-R_h]=\bigO{h^{2-\alpha}}{\mathcal{L}(L^2)}$. We turn to another term $[\vartheta(h^{1-\alpha}D_y), \vnh{R}{2}_{1+2\alpha}]$. Recall 
\[
\vnh{R}{2}_{1+2\alpha}=h\Op_h(2\ii \p_yg_2(x,y,\xi)\vartheta(\eta)\eta^2).
\]
Similarly, we obtain $[\vartheta(h^{1-\alpha}D_y), \vnh{R}{2}_{1+2\alpha}]=R_h$. Indeed, let $\hbar=h^{1-\alpha}$. Then we compute the commutator operator 
\begin{align*}
[\vartheta(h^{1-\alpha}D_y),\vnh{R}{2}_{1+2\alpha}]
&=h^{1+2\alpha}[\Op_{\hbar}(\vartheta(\eta)),\Op_{\hbar}(2\ii \p_yg_2(x,y,h^{\alpha}\xi)\vartheta(h^{\alpha}\eta)\eta^2)]\\
&=\hbar h^{1+2\alpha}\Op_{\hbar}\left(\{\vartheta'(\eta)\p_y(2 \p_yg_2(x,y,h^{\alpha}\xi)\vartheta(h^{\alpha}\eta)\eta^2)\}\right)+h^{1+2\alpha}\bigO{\hbar^3}{\mathcal{L}(L^2)}\\
&=h^{1+2\alpha}\bigO{\hbar}{\mathcal{L}(L^2)}\\
&=\bigO{h^{2+\alpha}}{\mathcal{L}(L^2)}\\
&=R_h.
\end{align*}
Let $g(x,y,\xi, \eta)$ be a symbol. We claim that 
\begin{gather*}
[\vartheta(h^{1-\alpha}D_y),\Op_h((\psi(\xi)-1)g(x,y,\xi, \eta))]\vh{v}{2}=\BigO{h^{\infty}},\\
[\vartheta(h^{1-\alpha}D_y),\Op_h((\vartheta(\eta)-1)g(x,y,\xi, \eta))]\vh{v}{2}=\BigO{h^{\infty}}.
\end{gather*}
Indeed, this is a direct sequence of symbolic calculus and the fact that $\vartheta(\frac{\cdot}{h^{\alpha}})(1-\vartheta(\cdot))\equiv0$, $\xi^2+h^{2\alpha}<1$ for $h$ sufficiently small and any $\xi\in\supp(1-\psi)$.

Recall that $R^{(2)}_{2+\alpha}=h\Op_h\left((1-\psi(\xi))\left(\langle A_2^{(1)}\rangle(y)-A_2^{(1)}(x,y)\right)\eta\right)+h\Op_h\left(\ii\{\xi^2,g_2(x,y,\xi)(1-\vartheta(\eta))\eta\}\right)$. Then, we deduce that
\[
[\vartheta(h^{1-\alpha}D_y), R^{(2)}_{2+\alpha}]\vh{v}{2}=\bigO{h^{\infty}}{L^2}.
\]
In summary, provided that $\alpha\in(0,\frac{1}{2})$, we conclude that
\[
[\vartheta(h^{1-\alpha}D_y),\vh{P}{2}-R^{(3)}_{1+2\alpha}-R^{(2)}_{2+\alpha}-R_h]\vh{v}{2}=\bigO{h^{2-\alpha}}{L^2}=o_{L^2}(h^{\frac{3}{2}})=r_h.
\]
This concludes the proof.
\end{proof}
To prove that $\mu\mathbf{1}_{\zeta=\zeta_0}=0$, it suffices to show $\vartheta(\frac{hD_y}{\epsilon})v_h=o_{L^2_{loc}(\R;L^2(\T^2))}(1)$. First we show that the normal form transform $e^{G_h}(\cdot)$ does not change the frequency localization. Recall that $\zeta=(\xi,\eta)$ is the dual variable of $z=(x,y)\in\T^2$.
\begin{lemma}\label{lem: xi near 1}
There exists $\rho_1\in(0,\frac{1}{4})$ such that for any semiclassical symbols $l(z,\eta)$ supported in a compact set away from $\{(z,\zeta);|\xi|\in(1-\rho_1,1+\rho_1)\}$ in $T^*\T^2$ and $L_h:=\Op_h(l)$, then for every $\alpha\in(0,\frac{1}{2})$ we have
\[
L_hw_h=o(1)_{L^2}.
\]
\end{lemma}
\begin{proof}
 By definition, $w_h=\vartheta(h^{1-\alpha}D_y)\vh{v}{2}$. Since $v_h$ satisfies $v_h=\Pi_{\rho_h,h}u_h$, by Proposition \ref{prop: measure property}, more precisely by the property of the support of the semiclassical measure $\mu$, for any semiclassical operator $K_h$ with its symbol supported away from $\{(z,\zeta):|\zeta|\in(1-\rho_0,1+\rho_0)\}$, we have 
\[
K_hv_h=o_{L^2}(1).
\]
Similar to the proof for the invariance of $\WFm{m}(\vh{v}{2})$, we have $K_h\vh{v}{2}=o_{L^2}(1)$. Choosing $\rho_1=2\rho_0$, we have
\begin{gather*}
\supp(l)\cap \{|\xi|\in(1-\rho_1,1+\rho_1)\}=\emptyset.
\end{gather*}
Therefore, $L_hw_h=L_h\vartheta(h^{1-\alpha}D_y)\vh{v}{2}$. There are two cases:
\begin{enumerate}
    \item If $\supp(l)\cap \{\eta\in(-\rho_0,\rho_0)\}=\emptyset$, for $h^{\alpha}<\frac{\rho_0}{4}$, by standard symbolic calculus, $L_hw_h=L_h\vartheta(h^{1-\alpha}D_y)\vh{v}{2}=\bigO{h^{\infty}}{L^2}$;
    \item If $\supp(l)\cap \{\eta\in(-\rho_0,\rho_0)\}\neq\emptyset$, then $\supp(l)\cap \{(z,\zeta):|\zeta|\in(1-\rho_0,1+\rho_0)\}=\emptyset$, thanks to $|\xi|>1+2\rho_0$ or $|\xi|<1-2\rho_0$ in such case. Hence, by semiclassical calculus, 
    \begin{align*}
     L_hw_h=L_h\vartheta(h^{1-\alpha}D_y)\vh{v}{2}
     =\vartheta(h^{1-\alpha}D_y)L_h\vh{v}{2}+\bigO{h^{1-\alpha}}{L^2}.
    \end{align*}
    Using $\supp(l)\cap \{(z,\zeta):|\zeta|\in(1-\rho_0,1+\rho_0)\}=\emptyset$, we know that $L_h\vh{v}{2}=o(1)_{L^2}$, we obtain $L_hw_h=o(1)_{L^2}$.
\end{enumerate}
This ends the proof.
\end{proof}
The preceding lemma implies that the $\xi-$localization of $v_h$ could also provide $\xi-$localization of $w_h$. The next lemma shows that the $\eta-$cutoff of $w_h$ provides the $\eta-$localization of $v_h$. 
\begin{lemma}\label{lem: smallness-chain}
Let $\alpha\in(0,\frac{1}{2})$ and $\{\Tilde{v}_h(t,\cdot)\}_{h>0}\subset L^2(\T^2)$ be a bounded sequence, and $G_h=\Op_h(g)$ be a bounded semiclassical operator. Define $\Tilde{w}_h:=\vartheta(h^{1-\alpha}D_y)e^{G_h}\Tilde{v}_h$. Then,  if $\Tilde{w}_h=o_{L^2_{loc}(\R;L^2(\T^2))}(1)$, we have 
\[
\|\vartheta(h^{1-\alpha}D_y)\Tilde{v}_h\|_{L^2_{loc}(\R;L^2(\T^2))}=o(1).
\]
\end{lemma}
\begin{proof}
Let $\theta\in C^{\infty}_c(\R_t)$, thanks to Lemma \ref{lem: exponential of bounded op} 
\begin{align*}
\theta\Tilde{w}_h=&e^{G_h}e^{-G_h}\theta\vartheta(h^{1-\alpha}D_y)e^{G_h}\Tilde{v}_h\\
=&e^{G_h}\left(\theta\vartheta(h^{1-\alpha}D_y)\Tilde{v}_h-\int_0^1e^{-sG_h}\theta[G_h,\vartheta(h^{1-\alpha}D_y)]e^{sG_h}\Tilde{v}_h\d s\right)\\
=&e^{G_h}\theta\vartheta(hD_y)\Tilde{v}_h+\bigO{h^{1-\alpha}}{L^2}.
\end{align*}
The last equality comes from $\|[G_h,\vartheta(h^{1-\alpha}D_y)]\|_{\mathcal{L}(L^2(\T^2))}\lesssim h^{1-\alpha}$. Hence, $\theta\vartheta(h^{1-\alpha}D_y)\Tilde{v}_h=e^{-G_h}\theta\Tilde{w}_h+\bigO{h^{1-\alpha}}{L^2}$.
If $\Tilde{w}_h=o_{L^2_{loc}(\R;L^2(\T^2))}(1)$, for $\alpha\in(0,\frac{1}{2})$, we obtain
\[
\|\theta\vartheta(h^{1-\alpha}D_y)\Tilde{v}_h\|_{L^2(\R\times\T^2)}=o(1)+\BigO{h^{1-\alpha}}=o(1)
\]
leading to the conclusion of the lemma.
\end{proof}
Using this lemma, we derive the following corollary:
\begin{corollary}\label{cor: zeta0}
Given $\theta\in C^{\infty}_c(\R_t)$, if $\|\theta w_h\|_{L^2(\R_t\times\T^2)}=o(1)$, then 
\begin{equation*}
\theta(t)\vartheta(h^{1-\alpha}D_y)v_h=o(1)_{L^2(\R_t;L^2(\T^2))}.    
\end{equation*}
\end{corollary}

\medskip
\textbf{Decomposition of the low-frequency part and the high-frequency part.} We finally split the transversal frequency into two parts 
\begin{equation}
\label{eq: decomposition of w-h}
    w_h:=w^1_h+w^2_h,\qquad w^1_h:=\vartheta(2h^{1-\alpha}D_y)w_h,\quad w^2_h:=\left(1-\vartheta(2h^{1-\alpha}D_y)\right)w_h,
\end{equation}
for $\alpha\in(\frac{1}{4},\frac{1}{2})$.  In this decomposition,  $w^1_h$ corresponds to the transversal low-frequency part with $|hD_y|\leq h^{\alpha}$, while $w^2_h$ corresponds to the transversal high-frequency part with $h^{\alpha}\leq|hD_y|\leq 2h^{\alpha}$, with $\alpha\in (\frac{1}{4},\frac{1}{2})$.

\section{Transversal high-frequency part}\label{sec: transversal HF}
In this section, we focus on the transversal high-frequency part $w^2_h$ defined in \eqref{eq: decomposition of w-h}. We notice that $w^2_h$ satisfies the following equation
\begin{equation}\label{eq: high-frequency-eq}
\left\{
\begin{array}{l}
     (\ii h^{\frac{3}{2}}\pt-\vh{P}{2})w^2_h=f^2_h,  \\
     w^2_h|_{t=0}=\left(1-\vartheta(2h^{1-\alpha}D_y)\right)w^0_h,
\end{array}
\right.
\end{equation}
where 
\begin{equation}
f^2_h:=\left(1-\vartheta(2h^{1-\alpha}D_y)\right)f_h-[\vartheta(2h^{1-\alpha}D_y),\vh{P}{2}-R^{(2)}_{1+2\alpha}-R^{(2)}_{2+\alpha}]w_h-(R^{(2)}_{1+2\alpha}+R^{(2)}_{2+\alpha})w_h^2.
\end{equation}

As usual, we quantify the size of the source term $f^2_h$.
\begin{lemma}\label{lem: source-est-hf}
Let $\alpha\in (0,\frac{1}{2})$. Then, the following estimates hold
\begin{align}
f_h^{2,a}:= \left(1-\vartheta(2h^{1-\alpha}D_y)\right)f_h-[\vartheta(2h^{1-\alpha}D_y),\vh{P}{2}-R^{(2)}_{1+2\alpha}-R^{(2)}_{2+\alpha}]w_h&=\bigO{h^{2-\alpha}}{L^2}, \label{eq:estimatefh2a}\\
f_h^{2,b}:= (R^{(2)}_{1+2\alpha}+R^{(2)}_{2+\alpha})w_h^2&=\bigO{h^{1+2\alpha}}{L^2}.\label{eq:estimatefh2b}
\end{align}
\end{lemma}

In particular, we can deduce the following result.
\begin{corollary}
For $\alpha\in (\frac{1}{4},\frac{1}{2})$, $(R^{(2)}_{1+2\alpha}+R^{(2)}_{2+\alpha})w_h^2=\bigO{h^{1+2\alpha}}{L^2}=o_{L^2}(h^{\frac{3}{2}})$ so
\begin{equation}
f^2_h = o_{L^2}(h^{\frac{3}{2}}).
\end{equation}
\end{corollary}

\begin{proof}
The proof of the first statement is similar to the proof of Lemma \ref{lem: f-h-error}. We omit some details here.

For $\vnh{R}{2}_{1+2\alpha}w_h^2$, setting $\hbar=h^{1-\alpha}$, we have
\begin{align*}
\vnh{R}{2}_{1+2\alpha}w_h^2&=h\Op_h(2\ii \p_yg_2(x,y,\xi)\vartheta(\eta)\eta^2)\left(1-\vartheta(2h^{1-\alpha}D_y)\right)w_h\\
&=h^{1+2\alpha}\Op_{\hbar}(2\ii \p_yg_3(x,y,h^{\alpha}\xi)\vartheta(h^{\alpha}\eta)\eta^2)\Op_{\hbar}\left(1-\vartheta(2\eta)\right)\vartheta(h^{1-\alpha}D_y)\vh{v}{2}\\
&=\bigO{h^{1+2\alpha}}{L^2}.
\end{align*}
Moreover, for $\vnh{R}{2}_{2+\alpha}w_h^2$, we have 
\begin{gather*}
h\Op_h\left((1-\psi(\xi))\left(\langle A_2^{(1)}\rangle(y)-A_2^{(1)}(x,y)\right)\eta\right)\left(1-\vartheta(2h^{1-\alpha}D_y)\right)\vartheta(h^{1-\alpha}D_y)\vh{v}{2}=\bigO{h^{\infty}}{L^2},\\
h\Op_h\left(\ii\{\xi^2,g_2(x,y,\xi)(1-\vartheta(\eta))\eta\}\right)\left(1-\vartheta(2h^{1-\alpha}D_y)\right)\vartheta(h^{1-\alpha}D_y)\vh{v}{2}=\bigO{h^{\infty}}{L^2}.
\end{gather*}
Therefore, we have 
\[
(R^{(2)}_{1+2\alpha}+R^{(2)}_{2+\alpha})w_h^2=\bigO{h^{1+2\alpha}}{L^2}+\bigO{h^{2}}{L^2}=\bigO{h^{1+2\alpha}}{L^2}.
\]
This concludes the proof.
\end{proof}
\begin{remark}\label{rem: index-alpha}
As we can observe in the preceding lemma, $\alpha>\frac{1}{4}$ is a technical condition to make $R^{(2)}_{1+2\alpha}$ as an error term of order $o_{L^2}(h^{\frac{3}{2}})$. Hence, by a two-step normal form reduction, we obtain a good simplified equation for the transversal high frequency part $w^2_h$
\begin{equation}\label{eq: w^2-simplified eq}
(\ii h^{\frac{3}{2}}\pt-\vh{P}{2})w^2_h=r_h.
\end{equation}
Indeed, if $\alpha\leq\frac{1}{4}$, we could perform more delicate normal form averaging processes to obtain \eqref{eq: w^2-simplified eq}. In order to present our proof in a more comprehensive way, we choose $\alpha>\frac{1}{4}$.
\end{remark}

Our next proposition concerns the energy estimates of $w^2_h$.
\begin{proposition}\label{prop: energy estimates}
Let $\alpha\in(\frac{1}{4},\frac{1}{2})$.
Then, we have
\begin{equation}
    \| w^2_h\|_{L^2((0,T)\times\T^2)}=o(1).
\end{equation}
\end{proposition}
\begin{proof}
We use a commutator method based on the
following basic identity
\[
[yD_y,D_x^2+D_y^2]=2iD_y^2.
\]
Before presenting the concrete computation, we first introduce some cutoff functions. Let $\theta_T\in C^{\infty}_c(\R)$ be a time cutoff such that $\supp\theta_T\subset[0,T]$ and $\varphi\in C^{\infty}_c(\T_y)$ be a space cutoff defined only on $y-$direction such that $\varphi=1$ on $\T_y\setminus I_y$ and $\supp\varphi\subset\T_y\setminus\Tilde{I}_y$.

We consider a localized vector field in the form $\theta^2_T(t)\varphi^2(y)yD_y$. Therefore, we have the following commutator computation
\begin{align*}
[\theta^2_T(t)\varphi^2(y)yD_y, &\; \ii h^{\frac{3}{2}}\pt+h^2\Delta]\\=&-2\ii\theta^2_T(t)\varphi^2(y)h^2D_y^2+\theta^2_T(t)[h^2D_y^2,\varphi^2]yD_y-2\ii h^{\frac{3}{2}}\theta_T'(t)\theta_T(t)\varphi^2(y)yD_y\\
=&\underbrace{-2\ii\theta^2_T(t)\varphi^2(y)h^2D_y^2-4\ii\theta^2_T(t)\varphi'(y)\varphi(y)yh^2D_y^2}_{\mathrm{principal\;part}}\\
-&h\big(2\theta^2_T(t)\varphi^2(y)+\theta^2_T(t)(2(\varphi'(y))^2+2\varphi(y)\varphi''(y))y
\big)hD_y-2\ii h^{\frac{1}{2}}\theta_T(t)\theta_T'(t)\varphi^2(y)yhD_y.
\end{align*}  
Moreover, for the lower order terms, we have
\begin{align*}
& [\varphi^2(y)yD_y, 2\langle A_1\rangle(y)hD_x]=-2\ii\varphi^2(y)y\langle A_1\rangle'(y)hD_x,\\
&[\varphi^2(y)yD_y, \langle \vnh{A}{1}_2\rangle hD_y+hD_y\langle \vnh{A}{1}_2\rangle]=-\ii\varphi^2(y)y\langle \vnh{A}{1}_2\rangle'(y)hD_y-ihD_y\varphi^2(y)y\langle \vnh{A}{1}_2\rangle'(y)\\
&\hspace{5.2cm}+\ii\langle \vnh{A}{1}_2\rangle(\varphi^2(y)y)'hD_y+\ii(\varphi^2(y)y)'hD_y\langle \vnh{A}{1}_2\rangle.
\end{align*}
In summary, from \eqref{eq:defP2}, we have the full commutator $[\theta^2_T(t)\varphi^2(y)yD_y,\ii h^{\frac{3}{2}}\pt-\vh{P}{2}]$ as follows
\begin{align}\label{eq1:commutator}
&[\theta^2_T(t)\varphi^2(y)yD_y,\ii h^{\frac{3}{2}}\pt-\vh{P}{2}]\notag  \\
=&-2\ii\theta^2_T(t)(\varphi^2(y)+2\varphi(y)\varphi'(y)y)h^2D_y^2+\mathcal{O}(h)\Op_h(r)-2\ii h^{\frac{1}{2}}\theta_T(t)\theta_T'(t)\varphi^2(y)yhD_y,
\end{align}
where $r(z,\zeta)\in S_{\mathrm{hom}}^1$. Recall that $S_{\mathrm{hom}}^1$ denotes the homogeneous symbol class, 
\begin{equation*}
S_{\mathrm{hom}}^1:=\{a\in S^1, a(z,\zeta) \text{ is homogeneous of order }1 \text{ with respect to }\zeta\}.
\end{equation*}
In particular, $\Op_h(r)v_h=\bigO{1}{L^2}$. In fact, $v_h$ is well-localized due to $v_h=\widetilde{\chi}(\frac{h^2H_{\A,V}-1}{\rho})v_h$. Then, $\Op_h(r)v_h=\Op_h(r)\widetilde{\chi}(\frac{h^2H_{\A,V}-1}{\rho})v_h$ is bounded in $L^2(\T^2)$.

Now let us turn to the energy estimates of the equation \eqref{eq: high-frequency-eq}. Multiplying \eqref{eq: high-frequency-eq} by $\theta^2_T(t)\varphi^2(y)y\overline{D_yw^2_h}$ and integration over $\R_t\times\T_z^2$ and taking the imaginary part, on the left hand side, we obtain
\begin{align*}
&\Im\int_{\R\times\T^2}(\ii h^{\frac{3}{2}}\pt-\vh{P}{2})w^2_h\theta_T(t)\varphi(y)y\overline{D_yw^2_h}\d z\d t\\
=&\frac{1}{2\ii}\poscals{[\theta^2_T(t)\varphi^2(y)yD_y,\ii h^{\frac{3}{2}}\pt-\vh{P}{2}]w^2_h}{w^2_h}_{L^2(\R\times\T^2)}-\frac{1}{2}\poscals{\theta^2_T(t)(\varphi^2(y)+2\varphi(y)\varphi'(y)y)r_h}{w^2_h}_{L^2(\R\times\T^2)}\\
=&\poscals{-\theta^2_T(t)\left(\varphi^2(y)+2\varphi(y)\varphi'(y)y\right)h^2D_y^2w^2_h}{w^2_h}_{L^2(\R\times\T^2)}-\poscals{h^{\frac{1}{2}}\theta_T(t)\theta_T'(t)\varphi^2(y)yhD_yw^2_h}{w^2_h}_{L^2(\R\times\T^2)}+\mathcal{O}(h).
\end{align*}
Integrating by parts, we obtain
\begin{align*}
-\Im\int_{\R\times\T^2}(\ii h^{\frac{3}{2}}\pt-\vh{P}{2})w^2_h\theta_T(t)\varphi(y)y\overline{D_yw^2_h}&\d z\d t\\
=\int_{\R\times\T^2}|\theta_T(t)\varphi(y)h\p_yw^2_h|^2&\d z\d t+2\int_{\R\times\T^2}\varphi(y)\varphi'(y)y|\theta_T(t)h\p_yw^2_h|^2\d z\d t\\
-&\poscals{h^{\frac{1}{2}}\theta_T(t)\theta_T'(t)\varphi^2(y)yhD_yw^2_h}{w^2_h}_{L^2(\R\times\T^2)}+\mathcal{O}(h).
\end{align*}
Now let us turn to the right hand side, we obtain from \eqref{eq: decomposition of w-h}, \eqref{eq:estimatefh2a} and \eqref{eq:estimatefh2b} that
\begin{gather*}
\Im\int_{\R\times\T^2}f^{2,a}_h\theta^2_T(t)\varphi^2(y)y\overline{D_yw^2_h}\d z\d t
=h\Im\int_{\R\times\T^2}\frac{f^{2,a}}{h^{2-\alpha}}\theta^2_T(t)\varphi^2(y)y\overline{h^{1-\alpha}D_yw^2_h}\d z\d t
=\mathcal{O}(h)\\
\Im\int_{\R\times\T^2}f^{2,b}\theta^2_T(t)\varphi^2(y)y\overline{D_yw^2_h}\d z\d t
=h^{3\alpha}\Im\int_{\R\times\T^2}\frac{f^{2,b}}{h^{1+2\alpha}}\theta^2_T(t)\varphi^2(y)y\overline{h^{1-\alpha}D_yw^2_h}\d z\d t
=\mathcal{O}(h^{3\alpha}).
\end{gather*}
Hence, we deduce that
\begin{equation}\label{eq: rhs-w-2}
\Im\int_{\R\times\T^2}f^2_h\theta^2_T(t)\varphi^2(y)y\overline{D_yw^2_h}\d z\d t
=\mathcal{O}(h^{3\alpha}+h).
\end{equation}
As a consequence, we derive that
\begin{align*}
&\int_{\R\times\T^2}|\theta_T(t)\varphi(y)h\p_yw^2_h|^2\d z\d t+2\int_{\R\times\T^2}\varphi(y)\varphi'(y)y|\theta_T(t)h\p_yw^2_h|^2\d z\d t\\
&\leq |\poscals{h^{\frac{1}{2}}\theta_T(t)\theta_T'(t)\varphi^2(y)yhD_yw^2_h}{w^2_h}_{L^2(\R\times\T^2)}|+|\Im\int_{\R\times\T^2}f^2_h\theta^2_T(t)\varphi^2(y)y\overline{D_yw^2_h}\d z\d t|+Ch\\
&\leq Ch\|\theta_T'\varphi w^2_h\|^2_{L^2(\R\times\T^2)}+\frac{1}{2}\|\theta_T\varphi h\p_yw^2_h\|^2_{L^2(\R\times\T^2)}+C(h+h^{3\alpha})\\
&\leq \frac{1}{2}\|\theta_T\varphi h\p_yw^2_h\|^2_{L^2(\R\times\T^2)}+C(h+h^{3\alpha}).
\end{align*}
This implies that
\begin{equation}\label{eq: H^1-est-1}  
\begin{aligned}
\int_{\R\times\T^2}|\theta_T(t)\varphi(y)h\p_yw^2_h|^2\d z\d t&\leq 4\left|\int_{\R\times\T^2}\varphi(y)\varphi'(y)y|\theta_T(t)h\p_yw^2_h|^2\d z\d t\right|+C(h+h^{3\alpha})\\
&\leq Ch^{2\alpha}\left(\|h^{1-\alpha}\p_yw^2_h\|^2_{L^2((0,T)\times\supp(\varphi'))}+h^{1-2\alpha}+h^{\alpha}\right).
\end{aligned}
\end{equation}
Thanks to the frequency cut-off of $w^2_h$, we have 
\begin{equation}\label{eq: bernstein-ineq}
\|h^{1-\alpha}\p_y(\varphi w^2_h)\|_{L^2(\T^2)}^2\gtrsim\|\varphi w^2_h\|_{L^2(\T^2)}^2 - \|h^{1-\alpha} w^2_h\|^2_{L^2(\supp(\varphi'))}.
\end{equation}
Hence, from \eqref{eq: bernstein-ineq} we know that
\begin{align*}
h^{2\alpha}\|\theta_T\varphi w^2_h\|_{L^2(\R\times\T^2)}^2&\leq \|\theta_Th\p_y(\varphi w^2_h)\|_{L^2(\R\times\T^2)}^2 + \|\theta_T h w^2_h\|^2_{L^2(\R\times\supp(\varphi'))}.\\
&\leq C\left(
\int_{\R\times\T^2}|\theta_T(t)\varphi(y)h\p_yw^2_h|^2\d z\d t+h^2\int_{\R\times\T^2}|\theta_T(t)\varphi'(y)w^2_h|^2\d z\d t\right)\\
&\leq C\left(
\int_{\R\times\T^2}|\theta_T(t)\varphi(y)h\p_yw^2_h|^2\d z\d t+h^2\right).
\end{align*}
Therefore, by the preceding estimate and \eqref{eq: H^1-est-1}, we conclude that
\begin{align*}
h^{2\alpha}\|\theta_T\varphi w^2_h\|_{L^2(\R\times\T^2)}^2\leq Ch^{2\alpha}\left(\|h^{1-\alpha}\p_yw^2_h\|^2_{L^2((0,T)\times\supp(\varphi'))}+h^{1-2\alpha}+h^{\alpha}\right)+C h^2.
\end{align*}
Using the support of $\supp(\varphi')\subset\omega$, $\|h^{1-\alpha}\p_yw^2_h\|^2_{L^2((0,T)\times\supp(\varphi'))}=o(1)$, which implies that
\[
h^{2\alpha}\|\theta_T\varphi w^2_h\|_{L^2(\R\times\T^2)}^2=o(h^{2\alpha}),\mbox{ i.e., }\|\theta_T\varphi w^2_h\|_{L^2(\R\times\T^2)}^2=o(1).
\]
We finally conclude that $\| w^2_h\|_{L^2((0,T)\times\T^2)}=o(1)$ due to the support of $\varphi$.
\end{proof}
\begin{remark}\label{rmk: optimal order-hf}
In general, we could relax many of the hypotheses in Proposition \ref{prop: energy estimates}.
In fact, let $f^2_h=o_{L^2}(h^{\beta})$ with $\beta\geq\frac{3}{2}$ and $\alpha<\frac{1}{2}$. After multiplying \eqref{eq: high-frequency-eq} with $\theta^2_T(t)\varphi^2(y)y\overline{D_yw^2_h}$, \eqref{eq: rhs-w-2} is of order $o(h^{\alpha+\beta-1})$ by
\begin{equation*}
h^{\alpha+\beta-1}\Im\int_{\R\times\T^2}\frac{f_h^2}{h^{\beta}}\theta^2_T(t)\varphi(y)y\overline{h^{1-\alpha}D_yw^2_h}\d z\d t\\
= o(h^{\alpha+\beta-1}).
\end{equation*}
Then, due to the support of $\varphi$, \eqref{eq: H^1-est-1} becomes
\begin{equation*}
\|\theta_T\varphi h\p_yw^2_h\|^2_{L^2}\leq C\left(h^{2\alpha}\left|\int_{\R\times\T^2}\varphi(y)\varphi'(y)y|\theta_T(t)h^{1-\alpha}\p_yw^2_h|^2\d z\d t\right|+h+h^{\alpha+\beta-1}\right)=o(h^{2\alpha})+o(h^{\alpha+\beta-1}).
\end{equation*}
Therefore, $\|\theta_T\varphi w^2_h\|_{L^2(\R\times\T^2)}^2=o(1)+o(h^{\beta-\alpha-1})=o(1)$.
\end{remark}
\begin{remark}\label{rmk: h-order and alpha-1/2}
In the proof of Proposition \ref{prop: energy estimates}, especially in the estimates \eqref{eq: H^1-est-1}, in order to conclude we need $\|w^2_h\|_{L^2([0,T]\times\T^2)}$ is $\BigO{h^{1-2\alpha}+h^{\alpha}}=o(1)$. This implies that $\alpha\in (0,\frac{1}{2})$.
\end{remark}

\section{Transversal low-frequency part}\label{sec: transversal LF}
In this section, we deal with the transversal low-frequency part $w^1_h$, defined in \eqref{eq: decomposition of w-h}. In the spirit of \cite{BZ12}, we use the normal form in Section \ref{sec: normal form reduction} to reduce the problem to a one-dimensional model equation. 

\begin{remark}
In this transversal low-frequency situation, on the one hand, we shall face one difficulty that, any transversal low-frequency truncation of the solution $\vartheta(h^{1-\alpha} D_y)w_h$ yields at most the estimate $\|h\p_y\vartheta(h^{1-\alpha} D_y)w_h\|_{L^2(\T^2)}=\BigO{h^{\frac{1}{2}}}$. Indeed, since $\alpha<\frac{1}{2}$,
\[
\|h\p_y\vartheta(h^{1-\alpha} D_y)w_h\|_{L^2(\T^2)}=h^{\alpha}\|h^{1-\alpha}\p_y\vartheta(h^{1-\alpha} D_y)w_h\|_{L^2(\T^2)}=\BigO{h^{\alpha}}.
\]
Unfortunately, this $\BigO{h^{\frac{1}{2}}}$ bound is not sufficient to conclude (see Lemma \ref{lem: hyperbolic estimate} below for more details). Indeed, this $\BigO{h^{\frac{1}{2}}}$ bound merely ensures the transversal low-frequency part $\vartheta(h^{1-\alpha} D_y)w_h$ remains bounded in $L^2$ but not necessarily small, that is, $\vartheta(h^{1-\alpha} D_y)w_h=\bigO{1}{L^2},\mbox{ but may not }o_{L^2}(1).$  To gain the smallness, we rely on the large time integration. On the other hand, due to the lack of the lower bound of the frequency cut-off of $w^1_h$, we cannot derive an inequality of the form \eqref{eq: bernstein-ineq}. Hence, the commutator method in Proposition \ref{prop: energy estimates} must be adapted accordingly.
\end{remark}

We start by looking at the equation for $w^1_h$ as follows:
\begin{equation}\label{eq: low-frequency-eq}
\left\{
\begin{array}{l}
     (\ii h^{\frac{3}{2}}\pt-\vh{P}{2})w^1_h=f^1_h,  \\
     w^1_h|_{t=0}=\vartheta(h^{1-\alpha} D_y)w^0_h,
\end{array}
\right.
\end{equation}
where 
\begin{equation}
    f^1_h:=\vartheta(2h^{1-\alpha}D_y)f_h+[\vartheta(2h^{1-\alpha} D_y),\vh{P}{2}-R^{(2)}_{1+2\alpha}-R^{(2)}_{2+\alpha}]w_h-(R^{(2)}_{1+2\alpha}+R^{(2)}_{2+\alpha})w^1_h
\end{equation}

Similarly, we have the following lemma.
\begin{lemma}
Let $\alpha\in(\frac{1}{4},\frac{1}{2})$.  Then we have
\begin{align*}
\vartheta(2h^{1-\alpha} D_y)f_h+[\vartheta(2h^{1-\alpha}D_y),\vh{P}{2}-R^{(2)}_{1+2\alpha}-R^{(2)}_{2+\alpha}]w_h&=\bigO{h^{2-\alpha}}{L^2},\\
(R^{(2)}_{1+2\alpha}+R^{(2)}_{2+\alpha})w^1_h&=\bigO{h^{1+2\alpha}}{L^2}.
\end{align*}
In particular, 
\begin{equation}
  f^1_h=o_{L^2}(h^{\frac{3}{2}}).
\end{equation}
\end{lemma}
The proof is the same as Lemma \ref{lem: source-est-hf} so we do not repeat it here. According to the preceding lemma, using the convention \eqref{eq: r-h}, we simplify the equation for $w^1_h$:
\begin{equation}\label{eq: low-frequency-eq-simplified}
\left\{
\begin{array}{l}
     (\ii h^{\frac{3}{2}}\pt-\vh{P}{2})w^1_h=r_h,  \\
     w^1_h|_{t=0}=\vartheta(h^{1-\alpha}D_y)w^0_h. 
\end{array}
\right.
\end{equation}

\subsection{One-dimensional estimates}
Recall $\omega=\T_x\times\omega_y$, with
\[
\omega_y=\bigcup_{j=1}^{n_0}I_j,\quad I_j=(a_j,b_j), \mod 2\pi,
\]
where $-\pi<a_1<b_1<a_2<b_2<\cdots<a_{n_0}<\pi<b_{n_0}<a_1+2\pi$. The following assumption holds throughout this section:

\vspace{1em}
$\mathbf{(MGCC_{y})}:$ ~Set $\delta_0<\frac{1}{2}\min(|I_1|,\cdots,|I_{n_0}|)$ such that any critical point of $\langle A_1\rangle(y)$ is contained in one of $\widetilde{I}_j=(a_j+\delta_0,b_j-\delta_0)$ mod $2\pi$ and $\widetilde{\omega}_y=\bigcup_{j=1}^{n_0}\widetilde{I}_j\subset\omega_y$.
\vspace{1em}

The primary objective of this section is to show that $\|w^1_h\|_{L^2((0,T)\times\T^2)}=o(1)$.

\vspace{1em}
Let us introduce another cutoff function $\psi_2\in C^{\infty}_c(\R),\mbox{ supported on }|\xi\pm1|\leq2\rho_1, \mbox{ and }\psi_2(\xi)=1 \mbox{ on }|\xi\pm1|\leq\rho_1$. 
Indeed, by Lemma \ref{lem: xi near 1}, we know that 
\begin{align*}
w_h^1=(1-\psi_2(hD_x))w_h^1+\psi_2(hD_x)w_h^1=\psi_2(hD_x)w_h^1+o(1)_{L^2}.   
\end{align*}
Therefore, we know that $w^1_h$ is spectrally localized in $|\eta|\leq h^{\alpha}$ and $|\xi\pm1|\leq 2\rho_1$.

\vspace{1em}
Since the left hand side of \eqref{eq: low-frequency-eq-simplified} commutes with $D_x$, by taking the Fourier transform in $x$, we can reduce the analysis to a sequence of one-dimensional problems.

Let $v_{h,K}$ be a solution to
\begin{equation}\label{eq: 1-d sch eq}
\left\{
\begin{array}{l}
     (\ii h^{\frac{3}{2}}\pt-P_{h,K})v_{h,K}=r_{h,K},  \\
     v_{h,K}|_{t=0}=v^0_{h,K},
\end{array}
\right.
\end{equation}
where $P_{h,K}:=h^2D_y^2+K^2-2hK\langle A_1\rangle(y)-2h\Op_h(\langle A_2^{(1)}\rangle(y)\eta)$. Moreover, the solution is microlocalized near $|\eta|\leq h^{\alpha}$ and $|K|\in(1-2\rho_1,1+2\rho_1)$. We aim to prove the following estimate for solutions of \eqref{eq: 1-d sch eq}.

\begin{proposition}\label{prop: 1d-key-est}
Suppose that $\mathbf{(MGCC_y)}$ holds. 
Then, there exists $T_0>0$, so that for every $T\geq T_0$, there exists $h_0\in(0,1)$ and $C_{0,T}>0$ such that for all $h\in(0,h_0)$, and all $|K|\in(\frac{1}{2},\frac{3}{2})$, the solutions $v_{h,K}$ of the equation \eqref{eq: 1-d sch eq} satisfy the uniform estimate:
\begin{equation}\label{eq: 1d-est}
\|v_{h,K}\|^2_{L^2([0,T]\times\T)}\leq C_{0,T}\left(\int_0^T\int_{\R}|\widetilde{\chi}(y)v_{h,K}(t,y)|^2\d y\d t+\frac{1}{h^3}\|r_{h,K}\|^2_{L^2([0,T]\times\T_y)}\right),
\end{equation}
where $\widetilde{\chi}$ is some bump function, compactly supported in $\omega_y$.
\end{proposition}

We start by the commutator method to derive a general  weighted energy estimate, specifically targeting the region $\T_y \setminus \omega_y$:  
\begin{lemma}\label{lem: hyperbolic estimate}
Let $\varphi\in C^{\infty}(\T_y;[0,1])$ and $\theta\in C^{\infty}_c(\R_t;[0,1])$. 
Assume that $\varphi=1$ on $\T_y\setminus \omega_y$, $\supp(\varphi)\subset \T_y\setminus\Tilde{\omega}_y$.
Assume that
 $\theta=1$ on $(\frac{1}{4},\frac{1}{2})$, $\supp(\theta)\subset[0,1]$ and $\theta_T(t)=\theta(\frac{t}{T})$.
 Let $\widetilde{\chi}\in C_c^{\infty}(\omega_y)$ be such that $\widetilde{\chi}=1$ on $\mathrm{supp}(\varphi')\cup\mathrm{supp}(1-\varphi)$.
 
  Let $\widetilde{v}_{h,K}$ be the solution of :
$$  \big(\ii h^{\frac{3}{2}}\partial_t-P_{h,K})\widetilde{v}_{h,K}=\widetilde{r}_{h,K},
$$
with $\widetilde{v}_{h,K}=\vartheta_0(h^{1-\alpha}D_y)\widetilde{v}_{h,K}$ for some $\vartheta_0\in C_c^{\infty}(\R)$.
Then for any $y_0\in\T_y$, there exists $C_M=C_M(\varphi,\langle A_1\rangle,\langle A_2^{(1)}\rangle)>0$ such that for any $T>0$ and sufficiently small $0<h\ll 1$, 
\begin{multline}\label{eq: hyperbolic estimate}
\|\theta_T\varphi h\p_y\widetilde{v}_{h,K}\|^2_{L^2(\R\times\T)}+h\int_{\R\times\T}K(y-y_0)\langle A_1\rangle'(y)|\theta_T\varphi \widetilde{v}_{h,K}|^2\d y\d t\\
\leq C_M
\Big(\frac{1}{h^2}\|\theta_T\widetilde{r}_{h,K}\|_{L^2(\R\times\T)}^2+h^2\|\theta_T\widetilde{v}_{h,K}\|_{L^2(\R\times\T)}^2  \Big)	\notag \\
+
	C_M\Big(\frac{h}{T^2}\big\|\theta'\big(\frac{t}{T}\big)\varphi\widetilde{v}_{h,K} \big\|_{L^2(\R\times\T)}^2 + \|\theta_T\widetilde{\chi}(y) h\partial_y\widetilde{v}_{h,K}\|_{L^2(\R\times\T)}^2 \Big).
\end{multline}
\end{lemma}
\begin{proof}
  We use a similar commutator method as in Proposition \ref{prop: energy estimates} based on $[yD_y,D_y^2]=2\ii D_y^2$. Let us consider a localized vector field $\theta^2_T(t)\varphi^2(y)(y-y_0)D_y$. Then, we have
\begin{align*}
[\theta^2_T(t)\varphi^2(y)&(y-y_0)D_y,\ii h^{\frac{3}{2}}\pt-h^2D_y^2-K^2]\\
&=-2\ii \frac{h^{\frac{3}{2}}}{T}\theta_T'(t)\theta_T(t)\varphi^2(y)(y-y_0)D_y-2\ii \theta^2_T(t)\varphi^2(y)h^2D_y^2+\theta^2_T(t)[h^2D_y^2,\varphi^2](y-y_0)D_y\\
&=\underbrace{-2\ii \theta^2_T(t)\varphi^2(y)h^2D_y^2-4\ii \theta^2_T(t)\varphi\varphi'(y)(y-y_0)h^2D_y^2 }_{\mbox{principal part}}\\
&-h\theta^2_T(2\varphi^2(y)+(\varphi^2)''(y)(y-y_0))hD_y-2\ii \frac{h^{\frac{3}{2}}}{T}\theta_T'(t)\theta_T(t)\varphi^2(y)(y-y_0)D_y
\end{align*}
Moreover, for the lower order terms, we have
\begin{align*}
& [\varphi^2(y)(y-y_0)D_y, 2hK\langle A_1\rangle(y)]=-2\ii hK\varphi^2(y)(y-y_0)\langle A_1\rangle'(y),\\
&[\varphi^2(y)(y-y_0)D_y, 2h\langle A_2^{(1)}\rangle(y)hD_y]=2\ii h\left(\langle A_2^{(1)}\rangle\varphi^2(y)+2\langle A_2^{(1)}\rangle\varphi\varphi'(y)(y-y_0)-\varphi^2(y)(y-y_0)\langle A_2^{(1)}\rangle'\right)hD_y,\\
&[\varphi^2(y)(y-y_0)D_y,-\ii h^2\langle A_2^{(1)}\rangle'(y)]=- h^2\varphi^2(y)(y-y_0) \langle A_2^{(1)}\rangle''(y).
\end{align*}
In summary, the full commutator reads as
\begin{align*}
[\theta^2_T(t)\varphi^2(y)(y-y_0)D_y,\ii h^{\frac{3}{2}}\pt-P_{h,K}]=&-2\ii \theta^2_T(t)\varphi^2(y)h^2D_y^2-4\ii \theta^2_T(t)\varphi\varphi'(y)(y-y_0)h^2D_y^2\notag  \\
-&2\ii hK\theta^2_T(t)\varphi^2(y)(y-y_0)\langle A_1\rangle'(y)+h\theta^2_T(t)F(y)hD_y\\
-&2\ii h^{\frac{3}{2}}\theta_T(t)\theta_T'(t)\varphi^2(y)(y-y_0)D_y+\bigO{h^2}{\mathcal{L}(L^2)},
\end{align*}
where $F(y)$ is a compactly supported function with $\supp(F)\subset\supp(\varphi)$ and $|F|\lesssim |\varphi|$.

\vspace{3mm}
Let us consider $\frac{1}{2\ii}\poscals{[\theta^2_T(t)\varphi^2(y)(y-y_0)D_y,\ii h^{\frac{3}{2}}\pt-P_{h,K}] \widetilde{v}_{h,K}}{\widetilde{v}_{h,K}}_{L^2(\R_t\times\T_y)}$. We deal with it term by term. For $\poscals{\theta^2_T(t)\varphi^2(y)h^2D_y^2 \widetilde{v}_{h,K}}{\widetilde{v}_{h,K}}_{L^2(\R\times\T_y)}$, integrating by parts, we get
\begin{align*}
\poscals{\theta^2_T(t)\varphi^2(y)h^2D_y^2 \widetilde{v}_{h,K}}{\widetilde{v}_{h,K}}_{L^2(\R\times\T)}&=\poscals{\theta_T(t)h\p_y \widetilde{v}_{h,K}}{\theta_T(t)h\p_y(\varphi^2(y)\widetilde{v}_{h,K})}_{L^2(\R\times\T)}\\
&=\int_{\R\times\T}|\theta_T(t)\varphi(y)h\p_y\widetilde{v}_{h,K}|^2\d y\d t+2h\int_{\R\times\T}\theta^2_T\varphi\varphi'h\p_y\widetilde{v}_{h,K}\overline{\widetilde{v}_{h,K}}\d y\d t.
\end{align*}
Similarly, we also have
\begin{align*}
&\poscals{2\theta^2_T(t)\varphi\varphi'(y)(y-y_0)h^2D_y^2 \widetilde{v}_{h,K}}{\widetilde{v}_{h,K}}_{L^2(\R\times\T)}\\
=&2\int_{\R\times\T}\theta^2_T\varphi(y)\varphi'(y)(y-y_0)|h\p_y\widetilde{v}_{h,K}|^2\d y\d t
+h\int_{\R\times\T}\theta^2_T\left(2\varphi\varphi'
+(\varphi^2)''(y)(y-y_0)\right)h\p_y\widetilde{v}_{h,K}\overline{\widetilde{v}_{h,K}}\d y\d t.
\end{align*}
The remaining part is given via
\begin{align*}
&\poscals{- \frac{h^{\frac{3}{2}}}{T}\theta_T\theta_T'(t)\varphi^2(y)(y-y_0)D_y \widetilde{v}_{h,K}}{\widetilde{v}_{h,K}}_{L^2(\R\times\T)}
=-\frac{h^{\frac{1}{2}}}{T}\int_{\R\times\T}\theta_T\theta_T'(t)\varphi^2(y)(y-y_0)hD_y \widetilde{v}_{h,K}\overline{\widetilde{v}_{h,K}}\d y\d t,\\
&\poscals{hK\theta^2_T(t)\varphi^2(y)(y-y_0)\langle A_1\rangle'(y)\widetilde{v}_{h,K}}{\widetilde{v}_{h,K}}_{L^2(\R\times\T)}
=h\int_{\R\times\T}K(y-y_0)\langle A_1\rangle'(y)|\theta_T\varphi\widetilde{v}_{h,K}|^2\d y\d t.
\end{align*}
Therefore, we know that
\begin{equation}\label{eq: imaginary main part}
\begin{aligned}
&\frac{1}{2\ii}\poscals{[\theta_T^2(t)\varphi^2(y)(y-y_0)D_y,\ii h^{\frac{3}{2}}\pt-P_{h,K}] \widetilde{v}_{h,K}}{\widetilde{v}_{h,K}}_{L^2(\R_t\times\T_y)}\\
=&-\|\theta_T\varphi h\p_y\widetilde{v}_{h,K}\|^2_{L^2(\R\times\T)}-h\int_{\R\times\T}K(y-y_0)\langle A_1\rangle'(y)|\theta_T\varphi\widetilde{v}_{h,K}|^2\d y\d t\\
+&\frac{h^{\frac{1}{2}}}{T}\int_{\R\times\T}\theta_T(t)\theta'\big(\frac{t}{T}\big)\varphi^2(y)(y-y_0)hD_y \widetilde{v}_{h,K}\overline{\widetilde{v}_{h,K}}\d y\d t+h\int_{\R\times\T}\theta^2_T(t)\widetilde{F}(y)h\p_y\widetilde{v}_{h,K}\overline{\widetilde{v}_{h,K}}\d y\d t\\
-&2\int_{\R\times\T}\theta^2_T\varphi(y)\varphi'(y)(y-y_0)|h\p_y\widetilde{v}_{h,K}|^2\d y\d t +\BigO{h^2}\|\theta_T \widetilde{v}_{h,K}\|_{L^2(\R\times\T)}^2.
\end{aligned}
\end{equation}
Here $\widetilde{F}(y)$ is a compactly supported function with $\supp(\widetilde{F})\subset\supp(\varphi)$ and $|\widetilde{F}|\lesssim |\varphi|$.
\medskip

Multiplying \eqref{eq: 1-d sch eq} by $\theta^2_T(t)\varphi^2(y)(y-y_0)\overline{D_y\widetilde{v}_{h,K}}$,  integrating over $\R_t\times\T_y$ and taking the imaginary part, we have
\begin{align}\label{eq: integral-identity}
\Big|\Im\int_{\R\times\T_y}(\ii h^{\frac{3}{2}}\pt-P_{h,K}) \widetilde{v}_{h,K}\cdot \theta^2_T(t)\varphi^2(y)(y-y_0)\overline{D_y\widetilde{v}_{h,K}}\d y\d t\Big| \notag \\ \leq \frac{C}{h}\|\theta_T\widetilde{r}_{h,K}\|_{L^2(\R\times\T)}\|\theta_T\varphi h\partial_y \widetilde{v}_{h,K}\|_{L^2(\R\times\T)}.
\end{align}
For the left hand side of \eqref{eq: integral-identity}, we have
\begin{align*}
&\Im\int_{\R\times\T_y}(\ii h^{\frac{3}{2}}\pt-P_{h,K}) \widetilde{v}_{h,K}\cdot \theta^2_T(t)\varphi^2(y)(y-y_0)\overline{D_y\widetilde{v}_{h,K}}\d y\d t\\
=&\frac{1}{2\ii}\poscals{[\theta^2_T(t)\varphi^2(y)(y-y_0)D_y,\ii h^{\frac{3}{2}}\pt-P_{h,K}] \widetilde{v}_{h,K}}{\widetilde{v}_{h,K}}_{L^2(\R\times\T_y)}\\-&\frac{1}{2}\poscals{\theta^2_T(t)\varphi^2(y)(\ii h^{\frac{3}{2}}\pt-P_{h,K}) \widetilde{v}_{h,K}}{\widetilde{v}_{h,K}}_{L^2(\R\times\T_y)}\\
-&\poscals{\theta^2_T(t)\varphi(y)\varphi'(y)(y-y_0)(\ii h^{\frac{3}{2}}\pt-P_{h,K}) \widetilde{v}_{h,K}}{\widetilde{v}_{h,K}}_{L^2(\R\times\T_y)}.
\end{align*}
Therefore, thanks to \eqref{eq: 1-d sch eq}, we derive
\begin{align*}
&\Big|\frac{1}{2\ii}\poscals{[\theta^2_T(t)\varphi^2(y)(y-y_0)D_y,\ii h^{\frac{3}{2}}\pt-P_{h,K}] \widetilde{v}_{h,K}}{\widetilde{v}_{h,K}}_{L^2(\R\times\T_y)}\Big| \\ \leq 
&C\|\theta_T \widetilde{r}_{h,K}\|_{L^2(\R\times\T)}\|\theta_T\varphi \widetilde{v}_{h,K}\|_{L^2(\R\times\T)}+\frac{C}{h}\|\theta_T \widetilde{r}_{h,K} \|_{L^2(\R\times\T)}\|\theta_T\varphi h\partial_y\widetilde{v}_{h,K}\|_{L^2(\R\times\T)}.
\end{align*}
Combing the inequality above and \eqref{eq: imaginary main part}, using  Cauchy-Schwarz, we obtain the key output :
\begin{align*}
&\|\theta_T\varphi h\p_y\widetilde{v}_{h,K}\|^2_{L^2}+h\int_{\R\times\T}K(y-y_0)\langle A_1\rangle'(y)|\theta_T\varphi\widetilde{v}_{h,K}|^2\d y\d t\\
\leq & C\|\theta_T \widetilde{r}_{h,K}\|_{L^2(\R\times\T)}\|\theta_T\varphi\widetilde{v}_{h,K}\|_{L^2(\R\times\T)}+\frac{C}{h}\|\theta_T \widetilde{r}_{h,K} \|_{L^2(\R\times\T)}\|\theta_T\varphi h\partial_y\widetilde{v}_{h,K} \|_{L^2(\R\times\T)}\\
+&C\frac{h^{\frac{1}{2}}}{T}\big\|\theta'\big(\frac{t}{T}\big)\varphi \widetilde{v}_{h,K} \big\|_{L^2(\R\times\T)}\|\theta_T\varphi h\partial_y\widetilde{v}_{h,K} \|_{L^2(\R\times\T)}+ Ch\|\theta_T\varphi h\partial_y\widetilde{v}_{h,K} \|_{L^2(\R\times\T)}\|\theta_T\widetilde{v}_{h,K}\|_{L^2(\R\times\T)}\\
+&C \|\theta_T |\varphi'|^{\frac{1}{2}}h\partial_y\widetilde{v}_{h,K} \|_{L^2(\R\times\T)}^2 +Ch^2\|\theta_T\widetilde{v}_{h,K}\|_{L^2(\R\times\T)}^2.
\end{align*}
Applying Young's inequality of the type
$$ AB\leq \epsilon A^2+C_{\epsilon}B^2,\qquad \epsilon >0,
$$
and absorbing similar terms to the left, 
we obtain that
\begin{multline*}
	\|\theta_T\varphi h\p_y\widetilde{v}_{h,K}\|^2_{L^2(\R\times\T)}+h\int_{\R\times\T}K(y-y_0)\langle A_1\rangle'(y)|\theta_T\varphi \widetilde{v}_{h,K}|^2\d y\d t\\
	\leq C_M
			\Big(\frac{1}{h^2}\|\theta_T\widetilde{r}_{h,K}\|_{L^2(\R\times\T)}^2+h^2\|\theta_T\widetilde{v}_{h,K}\|_{L^2(\R\times\T)}^2  \Big)		\notag \\
	+
			C_M\Big(\frac{h}{T^2}\big\|\theta'\big(\frac{t}{T}\big)\varphi\widetilde{v}_{h,K} \big\|_{L^2(\R\times\T)}^2 + \|\theta_T|\varphi'(y)|^{\frac{1}{2}} h\partial_y\widetilde{v}_{h,K}\|_{L^2(\R\times\T)}^2 \Big).
\end{multline*}
Finally, from the property of $\widetilde{\chi}$,
we conclude the proof of Lemma \ref{lem: hyperbolic estimate}.
\end{proof}

 For $T>1$, large enough, to be fixed later, set
 \begin{align}\label{def:LambdaT}
 \Lambda_{\leq T}:=\vartheta\big(\frac{h^{\frac{1}{2}}D_y }{T}\big).
 \end{align}
 We decompose $v_{h,K}=\widetilde{w}_1+\widetilde{w}_2$, where 
\begin{equation*}
\widetilde{w}_1=\Lambda_{\leq T}(v_{h,K}),\;\;\widetilde{w}_2=(1-\Lambda_{\leq T})(v_{h,K}).
\end{equation*}
Note that for sufficiently small $h$, $T h^{\frac{1}{2}}\leq h^{\alpha}$, since $\alpha<\frac{1}{2}$. Hence, $\widetilde{w}_1$ is genuinely the transversal low-frequency portion ($ |hD_y|\lesssim h^{\frac{1}{2}}$), while $\widetilde{w}_2$ is relatively the transversal high-frequency portion ($h^{\frac{1}{2}}\lesssim |hD_y|\lesssim h^{\alpha}$) of $v_{h,K}$. Moreover, 
\begin{equation*}
(\ii h^{\frac{3}{2}}\pt-P_{h,K})\widetilde{w}_{1}=\widetilde{r}_1,\quad (\ii h^{\frac{3}{2}}\pt-P_{h,K})\widetilde{w}_{2}=\widetilde{r}_2,
\end{equation*}
where
\begin{align*}
&\widetilde{r}_1=\Lambda_{\leq T}(r_{h,K}) - [P_{h,K}, \; \Lambda_{\leq T} ]v_{h,K},\\
&\widetilde{r}_2=(1-\Lambda_{\leq T})(r_{h,K})+[P_{h,K},\;  \Lambda_{\leq T}]v_{h,K}.
\end{align*}
From the symbolic calculus, we have
\begin{align}\label{r_1:commutator}
\|  [P_{h,K},\; \Lambda_{\leq T}]v_{h,K}   \|_{L_y^2(\R)}\leq \frac{Ch^{\frac{3}{2}}}{T}\|v_{h,K}\|_{L_y^2(\R)}
\end{align}
Applying Lemma \ref{lem: hyperbolic estimate} to $\widetilde{v}_{h,K}=\widetilde{w}_1, \widetilde{r}_{h,K}=\widetilde{r}_1$, 
we have the following corollary:

\begin{corollary}\label{cor:tildew_1}
Let $\varphi,\widetilde{\chi},\theta, T, y_0$ be the same as in Lemma \ref{lem: hyperbolic estimate}. Then, there exists $$C_M=C_M(\varphi,\langle A_1\rangle,\langle A_2^{(1)}\rangle)>0,$$ such that
\begin{align*}
\int_{\R\times\T}K\cdot (y-y_0)\langle A_1\rangle'(y)|\theta_T\varphi \widetilde{w}_1|^2 \d y\d t
\leq  &C_M\Big(\frac{1}{h^3}\|\theta_T\Lambda_{\leq T} (r_{h,K})\|_{L^2(\R\times\T)}^2 +\frac{1}{T^2}\| v_{h,K}\|_{L^2([0,T]\times\T)}^2 \Big)\\
+&C_M\big(T^2\|\theta_T\Lambda_{\leq T}(\widetilde{\chi}(y)v_{h,K} ) \|_{L^2(\R\times\T)}^2+Ch\|v_{h,K}\|_{L^2([0,T]\times\T)}^2 \big).
\end{align*}
\end{corollary}
\begin{proof}

Applying Lemma \ref{lem: hyperbolic estimate} and \eqref{r_1:commutator}, we have
\begin{align*}
	\|\theta_T \varphi h\partial_y\widetilde{w}_1\|_{L^2(\R\times\T)}^2+&hK\int_{\R\times\T}(y-y_0)\langle A_1\rangle'(y)|\theta_T\varphi \widetilde{w}_1|^2 \d y\d t\\
	\leq & C_M\Big(\frac{1}{h^2}\|\theta_T  \Lambda_{\leq T}(r_{h,K})\|_{L^2(\R\times\T)}^2 +h^2\|\theta_T\widetilde{w}_1\|_{L^2(\R\times\T)}^2+\frac{h}{T^2}\|\theta_T v_{h,K}\|_{L^2(\R\times\T)}^2 \Big)\\
	+&C_M\Big(\frac{h}{T^2}\big\|\theta'\big(\frac{t}{T}\big)\varphi \widetilde{w}_1 \big\|_{L^2(\R\times\T)}^2+\|\theta_T\widetilde{\chi}(y)h\partial_y\widetilde{w}_1 \|_{L^2(\R\times\T)}^2 \Big).
\end{align*}
Since $\Lambda_{\leq T}=\vartheta\big(\frac{h^{\frac{1}{2}} D_y}{T}\big)$ and $\widetilde{w}_1=\Lambda_{\leq T}(v_{h,K})$, we have
\begin{align*}
\widetilde{\chi}(y)h\partial_y\widetilde{w}_1=&h\partial_y(\widetilde{\chi}(y)\widetilde{w}_1 )+[\widetilde{\chi}(y),h\partial_y]\widetilde{w}_1\\
=&h\partial_y\Lambda_{\leq T}(\widetilde{\chi}(y)v_{h,K} )+h\partial_y[\widetilde{\chi}(y),\Lambda_{\leq T}]v_{h,K}+[\widetilde{\chi}(y),h\partial_y]\widetilde{w}_1\\
=& \Lambda_{\leq T}h\partial_y(\widetilde{\chi}(y)v_{h,K} )+[\widetilde{\chi}(y),\Lambda_{\leq T}]h\partial_y v_{h,K}+\mathcal{O}_{\mathcal{L}(L_y^2)}(h)v_{h,K}\\
=&h^{\frac{1}{2}}\cdot \Lambda_{\leq T}h^{\frac{1}{2}}\partial_y(\widetilde{\chi}(y)v_{h,K} )+\mathcal{O}_{\mathcal{L}(L_y^2)}(h)v_{h,K}.
\end{align*}	
Since $\|\Lambda_{\leq T}h^{\frac{1}{2}}\partial_y\|_{\mathcal{L}(L^2)}\leq CT$, we have 
\begin{align*}
\|\theta_T \widetilde{\chi}(y)h\partial_y\widetilde{w}_1 \|_{L^2(\R\times\T)}^2\leq CT^2h\|\theta_T\Lambda_{\leq T}(\widetilde{\chi}(y)v_{h,K})\|_{L^2(\R\times\T)}^2+Ch^2\|\theta_Tv_{h,K}\|_{L^2(\R\times\T)}^2. 
\end{align*}
Similarly,
\begin{align*}
\frac{h}{T^2}\|\theta'(\frac{t}{T})\varphi\widetilde{w}_1\|_{L^2(\R\times\T)}^2\leq & \frac{h}{T^2}\|\theta'(\frac{t}{T})\Lambda_{\leq T}(v_{h,K}) \|_{L^2(\R\times\T)}^2+Ch^2\|\theta'(\frac{t}{T})v_{h,K}\|_{L^2(\R\times\T)}^2.
\end{align*}
This concludes the proof.
\end{proof}
Next, we estimate the high frequency part $\widetilde{w}_2$.
\begin{corollary}\label{cor:tildew_2}
	Let $\varphi,\widetilde{\chi},\theta,y_0$ be the same as in Lemma \ref{lem: hyperbolic estimate}. Then, there exists $$C_M=C_M(\varphi,\langle A_1\rangle,\langle A_2^{(1)}\rangle)>0,$$ such that
	\begin{align}\label{eq: est-w-2}
		\|\theta_T \widetilde{w}_2\|_{L^2(\R\times\T)}^2\leq C_M\Big(\frac{1}{T^2}\|v_{h,K}\|_{L^2([0,T]\times\T)}^2+\frac{1}{h^3}\|\theta_Tr_{h,K}\|_{L^2(\R\times\T)}^2+\|\theta_T\widetilde{\chi}(y)v_{h,K}\|_{L^2(\R\times\T)}^2 \Big).
	\end{align}
\end{corollary}

\begin{proof}
We perform a Littlewood-Paley dyadic decomposition to connect the scales of $|D_y|$ from $ Th^{-\frac{1}{2}}$ to $h^{-(1-\alpha)}$. Pick $\widetilde{\vartheta}_2,\widetilde{\theta}\in C^{\infty}_c(\R)$, supported away from $0$. Denote 
$$ \Lambda_j:=\widetilde{\vartheta}_2(2^{-j}h^{\frac{1}{2}}D_y).
$$
The cutoff functions $\widetilde{\vartheta}_2$ is well-chosen such that the following dyadic decomposition holds :
\[
\widetilde{w}_{2} : =\sum_{j=N_{T}}^{N_{\alpha,h}}\widetilde{w}_{2,j},\text{ where } \widetilde{w}_{2,j}:= \Lambda_j\widetilde{w}_{2}, \text{ and }N_{\alpha,h}:=\lfloor(\frac{1}{2}-\alpha)\log_2(\frac{1}{h})\rfloor, N_{T}:=\lfloor\log_2(T)\rfloor.
\]
Similar as the equation for $\widetilde{w}_2$, we have $$(\ii h^{\frac{3}{2}}\pt-P_{h,K})\widetilde{w}_{2,j}=\widetilde{r}_{2,j},$$
where
$$
\widetilde{r}_{2,j}=\Lambda_j(1-\Lambda_{\leq T})(r_{h,K})+\mathcal{O}_{\mathcal{L}(L^2)}(2^{-j}h^{\frac{3}{2}})v_{h,K},
$$
where the last term on the right hand side comes from the similar commutator estimate as \eqref{r_1:commutator}. 
Applying Lemma \ref{lem: hyperbolic estimate} to $\widetilde{w}_{2,j}$, we have

\begin{multline*}
	\|\theta_T\varphi h\p_y\widetilde{w}_{2,j}\|^2_{L^2(\R\times\T)}+h\int_{\R\times\T}K(y-y_0)\langle A_1\rangle'(y)|\theta_T\varphi\widetilde{w}_{2,j}|^2\d y\d t\\
	\leq C_M\Big(\frac{1}{h^2}\|\theta_T\Lambda_j(r_{h,K}) \|_{L^2(\R\times\T)}^2+2^{-2j}h\|\theta_Tv_{h,K}\|_{L^2(\R\times\T)}^2  \Big)\\
	+C_M\Big(\frac{h}{T^2}\big\|\theta'\big(\frac{t}{T}\big)\varphi\widetilde{w}_{2,j} \|_{L^2(\R\times\T)}^2+\|\theta_T\widetilde{\chi}(y)h\partial_y\widetilde{w}_{2,j} \|_{L^2(\R\times\T)}^2 \Big).
\end{multline*}
The same argument as in the proof of Corollary \ref{cor:tildew_1} yields (using $\|\Lambda_j h^{\frac{1}{2}}\partial_y \|_{\mathcal{L}(L^2) }\leq C2^j$),
\begin{align*}
	&\|\theta_T\varphi h\p_y\widetilde{w}_{2,j}\|^2_{L^2(\R\times\T)}+h\int_{\R\times\T}K(y-y_0)\langle A_1\rangle'(y)|\theta_T\varphi\widetilde{w}_{2,j}|^2\d y\d t\\
	\leq &C_M\Big(\frac{1}{h^2}\|\theta_T\Lambda_j(r_{h,K}) \|_{L^2(\R\times\T)}^2+2^{-2j}h\|\theta_Tv_{h,K}\|_{L^2(\R\times\T)}^2  \Big)\\
	+&C_M\Big(\frac{h}{T^2}\big\|\theta'\big(\frac{t}{T}\big)\Lambda_j(v_{h,K}) \|_{L^2(\R\times\T)}^2+2^{2j}h\|\theta_T\Lambda_j\big(\widetilde{\chi}(y)v_{h,K}\big) \|_{L^2(\R\times\T)}^2 +Ch^2\big\|\theta'\big(\frac{t}{T}\big)v_{h,K} \big\|_{L^2(\R\times\T)}^2 \Big).
\end{align*}
For the left hand side, write
\begin{align*}
	\varphi h\p_y\widetilde{w}_{2,j}
	&=2^j h^{\frac{1}{2}}\left(-2^{-j}h^{\frac{1}{2}}\widetilde{\Lambda}_j\varphi'\widetilde{w}_{2,j}+2^{-j}h^{\frac{1}{2}}\widetilde{\Lambda}_j\p_y(\varphi \widetilde{w}_{2,j})-[\widetilde{\Lambda}_j,2^{-j}h^{\frac{1}{2}}\varphi\p_y]\widetilde{w}_{2,j}\right),
\end{align*}
which implies that 
\begin{gather*}
	\|\theta_T\varphi h\p_y\widetilde{w}_{2,j}\|^2_{L^2}\geq 2^{2j-1}h\|\theta_T\varphi\widetilde{w}_{2,j}\|^2_{L^2}-Ch^2\|\widetilde{w}_{2,j}\|^2_{L^2},\\
	\left|h\int_{\R\times\T}K(y-y_0)\langle A_1\rangle'(y)|\theta_T\varphi\widetilde{w}_{2,j}|^2\d y\d t\right|\leq Ch\|\theta_T\varphi\widetilde{w}_{2,j}\|^2_{L^2(\R\times\T)}.
\end{gather*}
Combining all these estimates above, we obtain
\begin{align*}
	2^{2j-1}h\|\theta_T\varphi\widetilde{w}_{2,j}\|^2_{L^2}\leq 
&	Ch\|\theta_T\varphi\widetilde{w}_{2,j}\|_{L^2(\R\times\T)}^2+Ch^2\|\theta_T\widetilde{w}_{2,j}\|_{L^2(\R\times\T)}^2
\\
+&C_M\Big(\frac{1}{h^2}\|\theta_T\Lambda_j(r_{h,K})\|_{L^2(\R\times\T)}^2+2^{-2j}h\|\theta_Tv_{h,K} \|_{L^2(\R\times\T)}^2 \Big)
\\
+&C_M\Big(
\frac{h}{T^2}\|\theta'(\frac{t}{T})\Lambda_j(v_{h,K}) \|_{L^2(\R\times\T)}^2
+2^{2j}h\|\theta_T\Lambda_j\big(\widetilde{\chi}(y)v_{h,K}\big)\|^2_{L^2(\R\times\T)}\Big)\\+&C_M h^2\|\theta'(\frac{t}{T})v_{h,K}\|_{L^2(\R\times\T)}^2.
\end{align*}
Now we sum in $j$, using the almost orthogonality of projectors $\Lambda_j$ to obtain
\begin{align*}
	\sum_{j=N_{T}}^{N_{\alpha,h}}\|\theta_T\varphi\widetilde{w}_{2,j}\|^2_{L^2}\leq &\frac{C_M}{T^2}\|\theta'(\frac{t}{T})v_{h,K}\|_{L^2(\R\times\T)}^2+
	\frac{C_M}{h^3}\|\theta_T r_{h,K}\|_{L^2(\R\times\T)}^2 +C_M\|\theta_T \widetilde{\chi}(y)v_{h,K}\|_{L^2(\R\times\T) }^2\\
	+&2^{-4N_{T}}C_M\|\theta_T v_{h,K}\|_{L^2(\R\times\T)}^2+C_M2^{-2N_T}h\|\theta'(\frac{t}{T})v_{h,K}\|_{L^2(\R\times\T)}^2.
\end{align*}
Note that $2^{-2N_{T}}\leq T^{-2}<1$. We finally deduce that 
\begin{equation}\label{eq:LTsumw2} \sum_{j=N_{T}}^{N_{\alpha,h}}\|\theta_T\varphi\widetilde{w}_{2,j}\|^2_{L^2}\leq C_M\Big(\frac{1}{T^2}\|v_{h,K}\|_{L^2(([-T,T])\times\T)}^2+\frac{C_M}{h^3}\|\theta_T r_{h,K}\|_{L^2(\R\times\T)}^2 +C_M\|\theta_T\widetilde{\chi}(y)v_{h,K} \|_{L^2(\R\times\T)}^2  \Big).
\end{equation}
To deal with the left hand side, 
 we introduce an auxiliary cutoff $\widetilde{\vartheta}\in C_c^{\infty}(\R)$ such that $\widetilde{\vartheta}= 1$ on supp$(\widetilde{\vartheta}_2)$, and 
$$  \widetilde{\Lambda}_j=\widetilde{\vartheta}(2^{-j}h^{\frac{1}{2}}D_y).
$$ 
We write
\begin{align*}
	\varphi\widetilde{w}_{2,j}&=\widetilde{\Lambda}_j\Lambda_j(\varphi\widetilde{w}_{2})+\widetilde{\Lambda}_j[\varphi,\Lambda_j]\widetilde{w}_{2}+(1-\widetilde{\Lambda}_j)[\varphi,\Lambda_j]\widetilde{w}_{2}+(1-\widetilde{\Lambda}_j)\Lambda_j(\varphi \widetilde{w}_{2})\\
	&=\Lambda_j(\varphi \widetilde{w}_{2})+\widetilde{\Lambda}_j[\varphi,\Lambda_j]\widetilde{w}_{2}+(1-\widetilde{\Lambda}_j)[\varphi,\Lambda_j]\widetilde{w}_{2}.
\end{align*}
This gives us
\begin{align*}
	\|\Lambda_j(\varphi \widetilde{w}_{2})\|^2_{L^2}&\leq \|\varphi \widetilde{w}_{2,j}\|^2_{L^2}+\|\widetilde{\Lambda}_j[\varphi,\Lambda_j] \widetilde{w}_{2}\|^2_{L^2}+\|(1-\widetilde{\Lambda}_j)[\varphi,\Lambda_j] \widetilde{w}_{2}\|^2_{L^2}\\
	&\leq \|\varphi \widetilde{w}_{2,j}\|^2_{L^2}+2^{-2j}Ch\| \widetilde{w}_{2}\|^2_{L^2}.
\end{align*}
Recall that $\widetilde{\chi}=1$ on supp$(1-\varphi)$, 
Plugging \eqref{eq:LTsumw2}, we have
\begin{align*}
\|\theta_T \widetilde{w}_2\|_{L^2(\R\times\T)}^2\leq C_M\Big(\frac{1}{T^2}\|v_{h,K}\|_{L^2([0,T]\times\T)}^2+\frac{C_M}{h^3}\|\theta_Tr_{h,K}\|_{L^2(\R\times\T)}^2+C_M\|\theta_T\widetilde{\chi}(y)v_{h,K}\|_{L^2(\R\times\T)}^2 \Big).
\end{align*}
This completes the proof.
\end{proof}

We are now in a position to prove Proposition \ref{prop: 1d-key-est}.
\begin{proof}[Proof of Proposition \ref{prop: 1d-key-est}]
 Recall the definition of $\mathbf{(MGCC_{y})}$
  and the corresponding parameter $\delta_0>0$. We take $\varphi_{j}\in C_c^{\infty}((b_j-\delta_0,a_j+\delta_0))$ such that
$ \varphi_j=1
$ on $(b_j,a_j)$. Note that on each connected component of  $(b_j-\delta_0,a_j+\delta_0)$ of $\T\setminus\widetilde{\omega}_y$, there is a choice of $y_{0,j}\in\{b_j,a_j \}$ (depending on the sign of $K$), such that
$ K(y-y_{0,j})\langle A_1\rangle'(y)>0
$
on this component. Furthermore, there is a universal constant $\gamma_0>0$, depending only on $\mathbf{(MGCC_{y})}$, such that
$$  K(y-y_{0,j})\langle A_1\rangle'(y)\geq \gamma_0\delta_0|K|\geq\frac{\gamma_0\delta_0}{2}
$$
on supp$(\varphi_j)$.

 Recall the constant $C_M>0$ in Corollary \ref{cor:tildew_1}, Corollary \ref{cor:tildew_2}. 
Let us fix $T_0\geq 1$ with the constraint
\begin{equation}\label{eq: defi-T0}
\frac{4C_M}{T_0^2}\Big(\frac{2}{\gamma_0\delta_0}+1\Big)<\frac{1}{2},
\end{equation}
and fix $T\geq T_0$ in the sequel. 
\begin{remark}
Here we give an explicit form of $T_0(\zeta_0)$ associated with the direction $\zeta_0=(1,0)$, which depends only on $\A$ and $\omega$. Thanks to Corollary \ref{cor: finite-direction}, we only have to analyze finite periodic orbits, say $\{\zeta_k\}_{0\leq k\leq k_0}$ with $k_0\in\N^*,k_0\leq 2p_0$. For each $\zeta_k$, we can find $T_0(\zeta_k)$ that depends only on $\A$ and $\omega$ similarly. Then, for the numerical constant $T_0$ appearing in Proposition \ref{prop: semiclassical ob}, we take $T_0:=\max_{0\leq k\leq k_0}T_0(\zeta_k)$. Consequently, we deduce that $T_0$ depends only on $\A$ and $\omega$, which is determined totally by the geometry and the magnetic effects. 
\end{remark}

Recall the decomposition $v_{h,K}=\widetilde{w}_1+\widetilde{w}_2$.

 First, applying Corollary \ref{cor:tildew_1} for $\widetilde{w}_1$, we have
 \begin{align*}
 \frac{\gamma_0\delta_0}{2} \cdot \|\theta_T\varphi \widetilde{w}_1\|_{L^2(\R\times\T)}^2\leq & \frac{C_M}{h^3}\|\theta_T r_{h,K}\|_{L^2(\R\times\T)}^2+C_MT^2\|\theta_T\widetilde{\chi}(y)v_{h,K}\|_{L^2(\R\times\T)}^2+ \frac{C_M}{T^2}\|v_{h,K}\|_{L^2([0,T]\times\T)}^2.
 \end{align*}
Using the fact that $\widetilde{\chi}(y)=1$ on supp$(1-\varphi)$ and $[\widetilde{\chi},\Lambda_{\leq T}]=\mathcal{O}_{\mathcal{L}(L^2)}(T^{-1}h^{\frac{1}{2}})$, we have
$$ \|\theta_T\widetilde{w}_1\|_{L^2(\R\times\T)}^2\leq \|\theta_T\varphi\widetilde{w}_1\|_{L^2(\R\times\T)}^2+
\|\theta_T(1-\varphi)\widetilde{w}_1\|_{L^2(\R\times\T)}^2+
 ChT^{-2}\|v_{h,K}\|_{L^2([0,T]\times\T)}^2,
$$
which yields (using $|1-\varphi|\leq \widetilde{\chi}$)
\begin{align}\label{output:w1}
 \|\theta_T\widetilde{w}_1\|_{L^2(\R\times\T)}^2\leq \frac{2}{\gamma_0\delta_0}\Big(\frac{C_M}{h^3}\|\theta_Tr_{h,K}\|_{L^2(\R\times\T)}^2+C_MT^2\|\theta_T\widetilde{\chi}(y)v_{h,K}\|_{L^2(\R\times\T)}^2+\frac{C_M}{ T^2}\|v_{h,K}\|_{L^2([0,T]\times\T)}^2\Big).
\end{align}
Next, apply Corollary \ref{cor:tildew_2} for $\widetilde{w}_2$, we have
\begin{align}\label{output:w2}
\|\theta_T\widetilde{w}_2\|_{L^2(\R\times\T)}^2\leq \frac{C_M}{h^3}\|\theta_Tr_{h,K}\|_{L^2(\R\times\T)}^2+C_M\|\theta_T\widetilde{\chi}(y)v_{h,K}\|_{L^2(\R\times\T)}^2+\frac{C_M}{T^2}\|v_{h,K}\|_{L^2([0,T]\times\T)}^2.
\end{align}
Adding \eqref{output:w1}, \eqref{output:w2}, using the fact that $K>1-2\rho_1$ we obtain that
\begin{align*}
\|v_{h,K}\|_{L^2([T/4,T/2]\times\T)}^2\leq & \Big(\frac{2C_M}{\gamma_0\delta_0}+1\Big)\frac{1}{h^3}\|\theta_Tr_{h,K}\|_{L^2(\R\times\T)}^2\\ + 
&C_M\Big(\frac{2T^2}{\gamma_0\delta_0}+1\Big)\|\theta_T\widetilde{\chi}(y)v_{h,K}\|_{L^2(\R\times\T)}^2\\
+&C_M\Big(\frac{2}{\gamma_0\delta_0}+1\Big)\frac{1}{T^2}\|v_{h,K}\|_{L^2([0,T]\times\T)}^2,
\end{align*}
where we have used the support property of $\theta_T$ such that
$$ \|v_{h,K}\|_{L^2([T/4,T/2]\times\T )}^2\leq \|\theta_T v_{h,K}\|_{L^2(\R\times\T)}^2.
$$

By Duhamel's formula and Cauchy-Schwarz, for any $t_1,t_2\in[0,T]$,
$$ \|v_{h,K}(t_1)\|_{L^2(\T)}\leq \|v_{h,K}(t_2)\|_{L^2(\T)}+\frac{|t_1-t_2|^{\frac{1}{2}}}{h^{\frac{3}{2}}}\|r_{h,K}\|_{L^2([t_1,t_2]\times\T)}.
$$
Taking the square and integrating over $t_1\in[0,T], t_2\in[T/4,T/2]$ for both sides, we have
\begin{align*}
	\frac{T}{4}\|v_{h,K}\|_{L^2([0,T]\times\T)}^2\leq T\|v_{h,K}\|_{L^2([T/4,T/2]\times\T)}^2+\frac{T^3}{h^3}\|r_{h,K}\|_{L^2([0,T]\times\T)}^2.
\end{align*}
This finally yields
\begin{align*}
\|v_{h,K}\|_{L^2([0,T]\times\T)}^2\leq &C_{T,\gamma_0,\delta_0}\Big(\frac{1}{h^3}\|\theta_Tr_{h,K}\|_{L^2(\R\times\T)}^2 + \|\theta_T\widetilde{\chi}(y)v_{h,K}\|_{L^2(\R\times\T)}^2\Big)\\+&\frac{4C_M}{T^2}\Big(\frac{2}{\gamma_0\delta_0}+1 \Big)\|v_{h,K}\|_{L^2([0,T]\times\T)}^2,
\end{align*}
where in the last line, we used the fact that $\rho_1<\frac{1}{4}$ in Lemma \ref{lem: xi near 1} so that $1-2\rho_1>\frac{1}{2}$. By definition of $T_0$ in \eqref{eq: defi-T0}, for $T\geq T_0$, we finally obtain that
\begin{align*}
\|v_{h,K}\|_{L^2([0,T]\times\T)}^2\leq C_{T,\gamma_0,\delta_0}\Big(\frac{1}{h^3}\|\theta_T r_{h,K}\|_{L^2(\R\times\T)}^2 + \|\theta_T \widetilde{\chi}(y)v_{h,K} \|_{L^2(\R\times\T)}^2 \Big)+\frac{1}{2}\|v_{h,K}\|_{L^2([0,T]\times\T)}^2.
\end{align*}
Recall that supp$(\theta_T)\subset (0,T)$ and supp$(\widetilde{\chi})\subset\widetilde{\omega}_y$, we complete the proof of Proposition \ref{prop: 1d-key-est}.
\end{proof}

\medskip

To finish this section, we finally show that $\|w_h^1\|^2_{L^2([0,T]\times\T^2)}=o(1)$. 
\begin{proposition}\label{prop: w-1-small}
Let $T_0$ be defined as in \eqref{eq: defi-T0} then for every $T \geq T_0$, we have 
\[
\|w_h^1\|^2_{L^2([0,T]\times\T^2)}=o(1).
\]
\end{proposition}
\begin{proof}
By taking the Fourier transform in $x$, 
\[
w_h^1(t,x,y)=\sum_{k\in\Z}w^1_{h,k}(t,y)e^{\ii kx}, r_h(t,x,y)=\sum_{k\in\Z}r_{h,k}(t,y)e^{\ii kx}.
\]
Thanks to the relation $w_h^1=\psi_2(hD_x)w_h^1+o(1)_{L^2}$ with a well-chosen cutoff $\psi_2$, we know that
\[
w_h^1(t,x,y)=\sum_{|hk|\in(1-2\rho_1,1+2\rho_1)}w^1_{h,k}(t,y)e^{\ii kx}+o(1)_{L^2}.
\]
Then each $w_{h,k}^1(t,y)$ satisfies the assumptions of $v_{h,K}(t,y)$ with $K=hk$. By Proposition \ref{prop: 1d-key-est}, we know that for $|hk|\sim1$,
\[
\|w^1_{h,k}\|^2_{L^2([0,T]\times\T_y)}\leq C_0\left(\int_0^T\int_{\R}|\widetilde{\chi}(y)w^1_{h,k}(t,y)|^2\d y\d t+\frac{1}{h^3}\|r_{h,k}\|^2_{L^2([0,T]\times\T_y)}\right).
\]
Applying the Plancherel theorem, 
\begin{align*}
\|w_h^1\|^2_{L^2([0,T]\times\T^2)}&=\sum_{|hk|\in(1-2\rho_1,1+2\rho_1)}\|w^1_{h,k}\|^2_{L^2([0,T]\times\T)}+\sum_{|hk|\notin(1-2\rho_1,1+2\rho_1)}\|w^1_{h,k}\|^2_{L^2([0,T]\times\T)}\\
&\leq C_0\sum_{|hk|\in(1-2\rho_1,1+2\rho_1)}\left(\int_0^T\int_{\R}|\widetilde{\chi}(y)w^1_{h,k}(t,y)|^2\d y\d t+\frac{1}{h^3}\|r_{h,k}\|^2_{L^2([0,T]\times\T_y)}\right)+o(1)\\
&\leq \int_0^T\int_{\R\times\T}|\widetilde{\chi}(y)w^1_{h}(t,z)|^2\d z\d t+\frac{1}{h^3}\|r_{h}\|^2_{L^2([0,T]\times\T^2)}+o(1)\\
&=o(1).
\end{align*}
The last equality is due to Lemma \ref{lem: source-est-hf}, the fact that supp$(\widetilde{\chi})\subset \omega_y$, and the vanishing hypothesis $\|w^1_{h}\|^2_{L^2([0,T]\times\omega_y)}=o(1)$.
\end{proof}

\subsection{End of the proof of Proposition \ref{prop: semiclassical ob}}
In this subsection, we finish the proof of Proposition \ref{prop: semiclassical ob}. Recall that it suffices to prove that $\mu\mathbf{1}_{\zeta=\zeta_0}=0$ with $\zeta_0=(1,0)$. 
\\

Thanks to Proposition \ref{prop: energy estimates} and Proposition \ref{prop: w-1-small}, we know that 
\begin{equation*}
\|w_h^2\|^2_{L^2([0,T]\times\T^2)}=o(1),\;\;\|w_h^1\|^2_{L^2([0,T]\times\T^2)}=o(1).
\end{equation*}
This leads to $\|w_h\|^2_{L^2([0,T]\times\T^2)}=o(1)$ due to the decomposition $w_h=w_h^1+w_h^2$. Then, by Corollary \ref{cor: zeta0}, we have $\vartheta(h^{1-\alpha}D_y)v_h=o(1)_{L^2((0,T)\times\T^2)}$.

For the $\epsilon_0\in(0,\frac{1}{2})$, we define $v_{h,\epsilon_0}=\vartheta(\frac{hD_y}{\epsilon_0})(1-\vartheta(h^{1-\alpha}D_y))v_h$. It suffices to prove that $v_{h,\varepsilon}=o(1)_{L^2_{loc}(\R_t;L^2(\T^2))}$. Indeed, as a direct consequence, 
\begin{equation}\label{eq: decom-mu}
\|\vartheta(\frac{hD_y}{\epsilon_0})v_h\|^2_{L^2((0,T)\times\T^2)}\leq \|v_{h,\epsilon_0}\|^2_{L^2((0,T)\times\T^2)}+\|\vartheta(\frac{hD_y}{\epsilon_0})\vartheta(h^{1-\alpha}D_y)v_h\|^2_{L^2((0,T)\times\T^2)}=o(1).
\end{equation}
This gives us the desired result $\mu\mathbf{1}_{\zeta=\zeta_0}=0$. Let us focus on presenting $ \|v_{h,\epsilon_0}\|^2_{L^2((0,T)\times\T^2)}=o(1)$. 
To prove this, we perform a Littlewood-Paley dyadic decomposition to connect the scales from $hD_y$ to $h^{1-\alpha}D_y$, similar to the proof of Corollary \ref{cor:tildew_2}.  Indeed, for some $\widetilde{\vartheta}\in C^{\infty}_c(\R)$, we further decompose 
\[
    v_{h,\epsilon_0} : =\sum_{j=0}^{N_{\alpha,h}}v_{j,h}^{\epsilon_0},\text{ where } v_{j,h}^{\epsilon_0}:= \widetilde{\vartheta}(2^{-j}h^{1-\alpha}D_y)v_{h,\epsilon_0}, \text{ and }N_{\alpha,h}=\lfloor\alpha\log_2(\frac{1}{h})\rfloor.
 \]
 Then we run the (essentially the same) argument in the proof of Proposition \ref{prop: energy estimates} and obtain the following proposition.
 \begin{proposition}\label{prop: energy estimates-app}
Let $\alpha<\frac{1}{2}$. 
Then, we have
\begin{equation*}
    \| v_{h,\epsilon_0}\|_{L^2((0,T)\times\T^2)}=o(1).
\end{equation*}
\end{proposition}

\begin{proof}
The proof is almost the same as Proposition 6.4. Here we only mention the differences for the sketch of the proof. The function $v^{\epsilon_0}_{j,h}$ satisfies the following equation:
\begin{equation}\label{eq: high-fre-v-eq-app}
(\ii h^{\frac{3}{2}}\partial_t-h^2H_{\A,V})v^{\epsilon_0}_{j,h}=r_{j,h}=\bigO{h^{2-\alpha}}{L^2},
\end{equation}
Let $\Phi_t$ be the geodesic flow on $T^*\T^2$ and for any $\zeta\in \mathbb{S}^1$, we denote by $\Phi_t(\cdot,\zeta): \T^2\to\T^2$ the projection of the flow map $\Phi_t$. Up to the change of coordinates, we assume that 
\begin{equation*}
\omega_0:=\T_x\times I_y\subset\bigcup_{t\in[0,2\pi]}\Phi_t(\cdot,\zeta_0)(\omega).
\end{equation*}
\begin{figure}
    \centering
\tikzset{every picture/.style={line width=0.75pt}} 

\begin{tikzpicture}[x=0.75pt,y=0.75pt,yscale=-0.8,xscale=0.8]

\draw   (163,12.5) -- (439,12.5) -- (439,270.5) -- (163,270.5) -- cycle ;
\draw  [fill={rgb, 255:red, 183; green, 216; blue, 247 }  ,fill opacity=0.87 ][line width=0.75]  (163,104.5) -- (439,104.5) -- (439,197.5) -- (163,197.5) -- cycle ;
\draw  [fill={rgb, 255:red, 126; green, 211; blue, 33 }  ,fill opacity=1 ][dash pattern={on 0.84pt off 2.51pt}][line width=0.75]  (319.04,78.12) .. controls (339.12,74.56) and (354.64,106.55) .. (353.7,149.59) .. controls (352.76,192.62) and (335.72,230.39) .. (315.64,233.95) .. controls (295.56,237.51) and (280.04,205.52) .. (280.98,162.49) .. controls (281.92,119.45) and (298.96,81.68) .. (319.04,78.12) -- cycle ;
\draw [line width=1.5]    (164,42.5) -- (437,41.51) ;
\draw [shift={(440,41.5)}, rotate = 179.79] [color={rgb, 255:red, 0; green, 0; blue, 0 }  ][line width=1.5]    (14.21,-4.28) .. controls (9.04,-1.82) and (4.3,-0.39) .. (0,0) .. controls (4.3,0.39) and (9.04,1.82) .. (14.21,4.28)   ;

\draw (302,189.4) node [anchor=north west][inner sep=0.75pt]  [font=\large]  {$\omega $};
\draw (167,114.4) node [anchor=north west][inner sep=0.75pt]    {$\omega _{0} =\mathbb{T}_{x} \times I_{y}$};
\draw (414,247.4) node [anchor=north west][inner sep=0.75pt]    {$\mathbb{T}^{2}$};
\draw (301,22.4) node [anchor=north west][inner sep=0.75pt]    {$\zeta _{0} =( 1,0)$};
\end{tikzpicture}
\end{figure}
Therefore, there exists $\epsilon_0$ sufficiently small, and $T_0>0$ such that for any $\zeta\in\mathbb{S}^1$, $|\zeta-\zeta_0|<\epsilon_0$ and any $z_0\in\omega_0$, we have $\Phi_t(z_0,\zeta)\in\omega$ for some $t\leq T_0$. Using the classical propagation argument and geometric control condition, we derive that $\|v_{h,\epsilon_0}\|_{L^2([0,T]\times\omega_0)}=o(1)$ (one can find similar proof in \cite[Lemma 3.4]{Sun-dampedwave}). Let $\theta_T\in C^{\infty}_c(\R)$ be a time cutoff such that $\supp\theta_T\subset[0,T]$ and $\varphi\in C^{\infty}_c(\T_y)$ be a space cutoff defined only on $y-$direction such that $\varphi=1$ on $\T_y\setminus I_y$ and $\supp\varphi\subset\T_y\setminus\Tilde{I}_y$. 
In summary, we have the full commutator $[\theta^2_T(t)\varphi^2(y)yD_y,\ii h^{\frac{3}{2}}\pt-h^2H_{\A,V}]$ as follows
\begin{align}\label{eq1:commutator-app}
&[\theta^2_T(t)\varphi^2(y)yD_y,\ii h^{\frac{3}{2}}\pt-h^2H_{\A,V}]\notag  \\
=&-2\ii\theta^2_T(t)(\varphi^2(y)+2\varphi(y)\varphi'(y)y)h^2D_y^2+\mathcal{O}(h)\Op_h(r)-2\ii h^{\frac{1}{2}}\theta_T(t)\theta_T'(t)\varphi^2(y)yhD_y,
\end{align}
where $r(z,\zeta)\in S_{\mathrm{hom}}^1$. 
In particular, $\Op_h(r)v^{\epsilon_0}_{j,h}=\bigO{1}{L^2}$. Now let us turn to the energy estimates of the equation. Multiplying \eqref{eq: high-fre-v-eq-app} by $\theta^2_T(t)\varphi^2(y)y\overline{D_yv^{\epsilon_0}_{j,h}}$ and integration over $\R_t\times\T_z^2$ and taking the imaginary part, we obtain
\begin{align}
-\Im\int_{\R\times\T^2}\left((\ii h^{\frac{3}{2}}\pt-h^2H_{\A,V})v^{\epsilon_0}_{j,h}\right)&\theta^2_T(t)\varphi^2(y)y\overline{D_yv^{\epsilon_0}_{j,h}}\d z\d t\notag\\
=\int_{\R\times\T^2}&|\theta_T(t)\varphi(y)h\p_yv^{\epsilon_0}_{j,h}|^2\d z\d t+2\int_{\R\times\T^2}\varphi(y)\varphi'(y)y|\theta_T(t)h\p_yv^{\epsilon_0}_{j,h}|^2\d z\d t\notag\\
-&\poscals{h^{\frac{1}{2}}\theta_T(t)\theta_T'(t)\varphi^2(y)yhD_yv^{\epsilon_0}_{j,h}}{v^{\epsilon_0}_{j,h}}_{L^2(\R\times\T^2)}+\BigO{h}.\label{eq: energy-est-identity-app}
\end{align}
In the left hand side of \eqref{eq: energy-est-identity-app}, using the equation \eqref{eq: high-fre-v-eq-app},  we have $r_{j,h}=\BigO{h^{2-\alpha}}$ and thanks to the frequency cutoff, we obtain
\begin{equation}\label{eq: rhs-v-hf-app}
\Im\int_{\R\times\T^2}r_{j,h}\theta^2_T(t)\varphi^2(y)y\overline{D_yv^{\epsilon_0}_{j,h}}\d z\d t=
\BigO{2^jh}.
\end{equation}
Since $0\leq j\leq \alpha\log_2(1/h)$, we know that $h^{\alpha}\leq 2^jh^{\alpha}\leq 1$ and $\frac{2^jh}{2^{2j}h^{2\alpha}}=2^{-j}h^{1-2\alpha}\ll1$. 
Using the frequency localization again, we notice that the sizes of the terms of the right hand side are 
\begin{equation*}
\|h\p_yv^{\epsilon_0}_{j,h}\|^2_{L^2(\T^2)}=\BigO{2^{2j}h^{2\alpha}},h^{\frac{1}{2}}\poscals{hD_yv^{\epsilon_0}_{j,h}}{v^{\epsilon_0}_{j,h}}_{L^2}=\BigO{2^jh^{\frac{1}{2}+\alpha}}.
\end{equation*}
Note that $\alpha<\frac{1}{2}$. Hence, the dominating term is $\|h\p_yv^{\epsilon_0}_{j,h}\|^2_{L^2(\T^2)}$. Now we perform more precise estimates. As a consequence of Cauchy--Schwarz's inequality and direct computation, we derive that
\begin{equation}\label{eq: H^1-est-1-app}  
\begin{aligned}
\int_{\R\times\T^2}|\theta_T(t)\varphi(y)h\p_yv^{\epsilon_0}_{j,h}|^2\d z\d t&\leq C\int_{\R\times\T^2}|\theta_T(t)\varphi'(y)h\p_yv^{\epsilon_0}_{j,h}|^2\d z\d t+Ch+C2^jh\\
&\leq C2^{2j}h^{2\alpha}\left(\|2^{-j}h^{1-\alpha}\p_yv^{\epsilon_0}_{j,h}\|^2_{L^2((0,T)\times\supp(\varphi'))}+h^{1-2\alpha}(2^{-2j}+2^{-j})\right).
\end{aligned}
\end{equation}
Thanks to the frequency cut-off of $v^{\epsilon_0}_{j,h}$, we have
\begin{equation}\label{eq: bernstein-ineq-app}
\|2^{-j}h^{1-\alpha}\p_y(\varphi v^{\epsilon_0}_{j,h})\|_{L^2(\T^2)}^2\gtrsim\|\varphi v^{\epsilon_0}_{j,h}\|_{L^2(\T^2)}^2-\|2^{-j}h^{1-\alpha}\varphi' v^{\epsilon_0}_{j,h}\|_{L^2(\T^2)}^2+\BigO{2^{-2j}h^{2-2\alpha}}
\end{equation}
Hence, from \eqref{eq: bernstein-ineq-app} and \eqref{eq: H^1-est-1-app}, we conclude that
\begin{align*}
2^{2j}h^{2\alpha}\|\theta_T\varphi v^{\epsilon_0}_{j,h}\|_{L^2(\R\times\T^2)}^2\leq C2^{2j}h^{2\alpha}\left(\|v^{\epsilon_0}_{j,h}\|^2_{L^2((0,T)\times\supp(\varphi'))}+h^{1-2\alpha}\right).
\end{align*}
Summing up in $j$, 
\begin{equation*}
\sum_j\|\theta_T\varphi v^{\epsilon_0}_{j,h}\|_{L^2(\R\times\T^2)}^2\leq C\sum_j\|v^{\epsilon_0}_{j,h}\|^2_{L^2((0,T)\times\supp(\varphi'))}+C\alpha h^{1-2\alpha}\log_2(h^{-1}).
\end{equation*}
Using the support of $\supp(\varphi')\subset\omega_0$, $\|v_{h,\epsilon_0}\|^2_{L^2((0,T)\times\supp(\varphi'))}=o(1)$, which implies that
\[
\|\theta_T\varphi v_{h,\epsilon_0}\|_{L^2(\R\times\T^2)}^2=o(1)+\BigO{h^{1-2\alpha}\log_2(h^{-1})}=o(1).
\]
We finally conclude that $\| v_{h,\epsilon_0}\|_{L^2((0,T)\times\T^2)}=o(1)$ due to the support of $\varphi$.
\end{proof}

\section{From the semiclassical observability to the observability}
\label{sec:refobsmultid}

We start from \eqref{eq: semiclassical ob}. Then the goal is to prove that 
\begin{equation}
\label{eq:obsandcompact}
 \left\| u_0 \right\|_{L^2(\T^2)}^2 \leq C{'} \left(\int_0^T \int_{\omega} \left| e^{-\ii tH_{\A,V}} u_0 \right|^2 \d z\d t+ \|u_0\|^2_{H^{-2}(\T^2)}\right).
\end{equation}
The term $\|u_0\|^2_{H^{-2}(\T^2)}$ will be removed in a further step. The strategy is by now classical and it is strongly inspired by \cite[Section 3.1]{BZ19}.

\begin{proof}

Let $T>0$. By assumptions, there exist $\rho_0>0$, $h_0>0$ such that \eqref{eq: semiclassical ob} holds. Fix $R>1$ such that 
$(R^{-1},R) \subset \{r \in \R\ ;\ \chi((r-1)/\rho_0) = 1\}$. Then from \cite[Proposition 2.10]{BCD11}, one can find a dyadic partition of the unity as follows, there exist $\varphi_0 \in C_c^{\infty}((-1,1);[0,1])$ and $\varphi \in C_c^{\infty}((R^{-1},R);[0,1])$ such that, denoting $\varphi_k^2(r)=\varphi^2(R^{-k}r)$ for $k \geq 1$,
\begin{equation}
\label{eq:decompositiondyadique}
    \forall r \in \R^+,\ \varphi_0^2(r) + \sum_{k=1}^{+\infty} \varphi_k^2(r) = 1.
\end{equation}

Let $\psi \in C_c^{\infty}((0,T);[0,1])$ that satisfies $\psi(t) =1$ on $(T/3,2T/3)$. Let us choose $K\geq 1$ large enough such that $R^{-K} \leq h_0^2$. Then, we have that for every $k \geq K+1$, setting $h=R^{-k/2}$, $\varphi_{k}(H_{\A,V})$ coincides with $\Pi_{h,\rho_0,V} \varphi_{k}(H_{\A,V})$. One can then use the observability estimate \eqref{eq: semiclassical ob} on the time interval $(T/3,2T/3)$ to obtain
\begin{equation}
\label{eq:EstimationFrequencyCutOfft3}
    \norme{\varphi_{k}(H_{\A,V})u(T/3)}_{L^2(\T^2)}^2 \leq C \int_{\R} \psi(t)^2 \norme{1_{\omega} \varphi_{k}(H_{\A,V}) u(t)}_{L^2(\T^2)}^2 dt\qquad \forall k \geq K+1.
\end{equation}
Now using that $D_t u = -H_{\A,V} u$ and $D_t 1_{\omega} = 1_{\omega} D_t$, we deduce from \eqref{eq:EstimationFrequencyCutOfft3} that
\begin{align}
\label{eq:EstimationFrequencyCutOfft3Bis}
    \norme{\varphi_{k}(H_{\A,V})u(T/3)}_{L^2(\T^2)}^2 \leq C \norme{1_{\omega} \psi(t) \varphi_k(D_t) u}_{L^2(\R_t \times  \T^2)}^2 \qquad \forall k \geq K+1.
\end{align}

Let $\widetilde{\psi} \in C_c^{\infty}((0,T);[0,1])$ such that $\widetilde{\psi}=1$ on $\text{supp}(\psi)$. Setting $h=R^{-k/2}$, the semiclassical parameter, from the semiclassical calculus on $\R$, the asymptotic expansion holds
\begin{align}
    \psi(t) \varphi_k^2(D_t) &= \psi(t) \varphi^2(hD_t)\notag\\
    &=\psi(t) \varphi^2(hD_t) \widetilde{\psi}(t)+\psi(t) \varphi^2(hD_t) (1- \psi(t)) \notag \\
   &= \psi(t) \varphi^2(hD_t) \widetilde{\psi}(t) + E(t,hD_t)(1+|t|^2)^{-1}(1+|hD_t|^2)^{-1},\label{eq:psitildepsisemiclassical}
\end{align}
where $$E(t,hD_t) = \Op_h(c),\ c \in \mathcal{S}(\R^2)\ \text{and}\ \sup_{(t,\tau) \in \R\times \R} |(1+t^2)^{\alpha} (1+\tau^2)^{\beta} \partial_{t}^\gamma \partial_{\tau}^{\delta} c(t,\tau)| \leq C_{\alpha,\beta,\gamma,\delta} h^3,\ \alpha,\beta,\gamma,\delta \in \N.$$
Then, by using \eqref{eq:EstimationFrequencyCutOfft3Bis}, \eqref{eq:psitildepsisemiclassical},  Calder\'on-Vaillancourt's theorem and again $D_t u = -H_{\A,V} u$ in $(0,T)$ we have 
\begin{align*}
  \notag & \norme{\varphi_{k}(H_{\A,V})u(T/3)}_{L^2(\T^2)}^2 \\
 \notag   & \leq C \norme{1_{\omega} \varphi_k(D_t) \widetilde{\psi}(t) u}_{L^2(\R_t \times\T^2)}^2 + C  h^6 \norme{(1+t^2)^{-1} (1+|hD_t|^2)^{-1} u(t)}_{L^2(\R_t \times \T^2)}^2 \\
 & \leq C \norme{ \varphi_k(D_t) \widetilde{\psi}(t) 1_{\omega} u}_{L^2(\R_t \times\T^2)}^2 + C  h^6 \norme{(1+t^2)^{-1} (1+|hH_{\A,V}|^2)^{-1} u(t)}_{L^2(\R_t \times \T^2)}^2 \qquad \forall k \geq K+1.
\end{align*}
Therefore, by summing for $k \geq K+1$ the preceding estimate, remembering that $h=R^{-k/2}$, we get from   the conservation of the $L^2$-norm for the electromagnetic Schr\"odinger equation, and \eqref{eq:decompositiondyadique}
\begin{align}
\sum_{k=K+1}^{+\infty} \norme{\varphi_{k}(H_{\A,V})u(T/3)}_{L^2(\T^2)}^2  &\leq \int_{\R} \widetilde{\psi}(t)^2 \norme{1_{\omega}  u(t)}_{L^2(\T^2)}^2 dt + C \norme{(1+|H_{\A,V}|^2)^{-1} u(t)}_{L^{\infty}(\R_t;L^2(\T^2))}^2\notag \\
 & \leq C \int_{0}^T \norme{1_{\omega}  u(t)}_{L^2(\T^2)}^2 dt + C \norme{(1+|H_{\A,V}|^2)^{-1} u_0}_{L^2(\T^2)}^2 \notag \\
 & \leq C \int_{0}^T \norme{1_{\omega}  u(t)}_{L^2(\T^2)}^2 dt + C \norme{u_0}_{H^{-2}(\T^2)}^2
\label{eq:estimatehighfrequenciescutoff}
\end{align}

To sum up, we then have  \eqref{eq:estimatehighfrequenciescutoff}
\begin{multline*}
     \norme{u_0}_{L^2(\T^2)}^2 = \sum_{k=0}^{+\infty} \norme{\varphi_k(H_{\A,V})u_0}_{L^2(\T^2)}^2
      \leq C \norme{u_0}_{H^{-2}(\T^2)}^2 + \sum_{k=K+1}^{+\infty} \norme{e^{-i(T/3)H_{\A,V}}\varphi_{k}(H_{\A,V})u_0}_{L^2(\T^2)}^2\\
     \leq C \norme{u_0}_{H^{-2}(\T^2)}^2 + \sum_{k=K+1}^{+\infty} \norme{\varphi_{k}(H_{\A,V})u(T/3)}_{L^2(\T^2)}^2
     \leq C \int_{0}^T \norme{1_{\omega}  u(t)}_{L^2(\T^2)}^2 dt + C \norme{u_0}_{H^{-2}(\T^2)}^2,
\end{multline*}
which concludes the proof of \eqref{eq:obsandcompact}.
\end{proof}

\subsection{Remove the compact term in the weaker observability estimate}

The goal is to prove the true observability estimate without the error term.
\begin{proof}
Let us assume that \eqref{eq:obsandcompact} holds for any time $T>0$.\\

\textit{First step: an unique continuation property.} Let us consider the following space
\begin{equation*}
    \mathcal N_{T, \A, V}=\left\{u \in L^2(\T^2); \ e^{-itH_{\A,V}}u(x)=0\ \text{on}\ [0,T]\times \omega\right\},
\end{equation*}
First of all, let us show that the Hilbert space $(\mathcal N_{T, \A, V}, \|\cdot\|_{L^2(\T^2)})$ is a finite dimensional space. Thanks to \eqref{eq:obsandcompact}, we have that the norms $\|\cdot\|_{L^2(\T^2)}$ and $\|\cdot\|_{H^{-2}(\T^2)}$ are equivalent. As a consequence, one can readily show that the unit closed ball $\overline{B_{L^2(\T^2)}(0,1)} \cap \mathcal N_{T, \A, V}$ is compact, thanks to the Rellich Theorem. In particular, the Riesz Theorem provides that the dimension of $\mathcal N_{T, \A, V}$ is finite. 

Let us now show that $\mathcal N_{T, \A, V}=\left\{0\right\}$. To that end, we proceed by contradiction and assume $\mathcal N_{T, \A, V}\neq \left\{0\right\}$. We begin by noticing that $\mathcal N_{T, \A, V}$ is invariant by the action of $H_{\A,V}$. Indeed, if $u \in \mathcal N_{T, \A, V}$, then for all $0<\varepsilon< T$, $u_\varepsilon= \frac{e^{-i\varepsilon H_{\A,V}}u-u}{\varepsilon}$ belongs to $\mathcal N_{T-\varepsilon, \theta, V}$. Thus, by applying \eqref{eq:obsandcompact} at time $T-\varepsilon$, we obtain that for all $0< \varepsilon < T$,
\begin{equation*}
    \|u_{\varepsilon}\|_{L^2(\T^2)} \leq C{'} \|u_{\varepsilon} \|_{H^{-2}(\T^2)},
\end{equation*}
for some positive constant $C'>0$.
Since $u_{\varepsilon} \underset{\varepsilon\to 0^+}{\longrightarrow} \mathcal H_{\theta, V}u$ in $\mathcal D'(\T^2)$ and $H^{-2}(\T^2)$, it follows that $u \in D\left(H_{\A,V}\right)$ (see for instance \cite[Section 1.1]{Paz83}) and 
$$\|H_{\A,V}u\|_{L^2(\T^2)} \leq C'' \|u\|_{L^2(\T^2)},$$
for some positive constant $C''>0$.
Moreover, since $u_{\varepsilon}$ belongs to $\mathcal N_{T-\varepsilon, \A, V}$ for all $0< \varepsilon < T$, we deduce that $H_{\A,V}u \in \mathcal N_{T-\delta, \A, V}$ for all $0< \delta <T$. Then, $H_{\A,V}u \in \mathcal N_{T, \A, V}$ and $\mathcal N_{T, \A, V}$ is invariant by the action of $H_{\A,V}$. As a consequence, since $\mathcal {H_{\A, V}}_{\left| \mathcal N_{T, \A, V} \right.}$ is a self-adjoint operator on the finite dimensional Hilbert space $(\mathcal N_{T, \A, V}, \|\cdot\|_{L^2(\T^2)})$, there exists a nontrivial function $\phi \in L^2(\T^2)$ and $\lambda \in \rr$ such that
$$H_{\A, V}\phi = \lambda \phi \quad \text{and} \quad \phi \in \mathcal N_{T, \A, V}.$$
In particular, $\phi$ is an eigenfunction of $H_{\A, V}$ which vanishes on an open subset. By the unique continuation result, \cite[Theorem 5.2]{LLRR}, we deduce that $\phi\equiv 0$. This provides a contradiction and consequently, $\mathcal N_{T, \A, V}=\left\{0\right\}.$\\

\textit{Second step: we remove the $H^{-2}$-norm in the weak observability estimate.}
By now, we establish that there exists a positive constant $C{'}=C{'}(T,M)>0$ such that 
\begin{equation*}
    \forall u_0 \in L^2(\T^2), \quad \left\| u_0 \right\|_{L^2(\T^2)}^2 \leq C{'} \int_0^T \int_{\omega}  \left| e^{-itH_{\A,V}} u_0(z)\right|^2 dz dt.
\end{equation*}
Once again, we proceed by contradiction and it provides sequences $(u_n)_{n \in \nn} \subset L^2(\T^2)$ with $\|u_n\|_{L^2(\T^2)}=1$ for all $n \in \nn$, 
\begin{equation}\label{observability_contradiction}
    \int_0^T \int_{\omega} \left| e^{-it\mathcal{H}_{\theta_n,V_n}} u_n(z)\right|^2 dz dt \leq \frac1n.
\end{equation}
Since $(u_n)_{n \in \nn}$ is bounded in $L^2(\T^2)$, there exists $f\in L^2(\T^2)$ such that, up to a subsequence, $(u_n)_{n \in \nn}$ weakly converges to $f$ in $L^2(\T^2)$ and strongly converges to $f$ in $H^{-2}(\T^2).$ Thanks to the weak observability estimate \eqref{eq:obsandcompact}, it follows that 
\begin{equation}
    \label{eq:fnonnul}
    1\leq C \|f\|^2_{H^{-2}(\T^2)}.
\end{equation}
We also have that $e^{-itH_{\A,V}} u_n(z)$ weakly converges in $L^2(0,T;L^2(\T^2))$ to $e^{-itH_{\A,V}} f(z)$

From \eqref{observability_contradiction}, we then deduce that for all $t \in [0,T]$, 
\begin{equation*}
    \int_0^T \int_{\omega}  \left| e^{-itH_{\A,V}} f(z)\right|^2 dz dt \leq \liminf_{n \to +\infty}  \int_0^T \int_{\omega}   \left| e^{-itH_{\A,V}} u_n(z)\right|^2 dz dt = 0,
\end{equation*}
that implies that $f \in \mathcal N_{T, \A, V}=\left\{0\right\}$. This contradicts \eqref{eq:fnonnul} and ends the proof of the observability estimate.
\end{proof}

\section{Optimality}\label{sec: optimality}
In this section, we prove the optimality of {\bf(MGCC)}, in the context of Theorem \ref{thm: optimal}. 
\subsection{Reduction to the model equation }
We begin with some simplifications. By definition, if {\bf{(MGCC)}} is not satisfied, there exists a periodic direction $\vec{\gamma}$ such that $A_{\gamma}$ has at least one non-degenerate critical point outside the strips $\omega_{\gamma^{\perp}}$.  
By changing coordinate system, we may assume that $\vec{\gamma}=\zeta_0=(1,0)$.
Next, since the observability is invariant by the gauge $\mathbf{A}\mapsto \mathbf{A}+\nabla g$, we assume  that $\mathbf{A}=(A_1(y),A_2(x,y))$. By assumption, $A_{\gamma}=A_1$ has a critical point outside the horizontal strips $\overline{\omega}_{\gamma}^{\perp}$. By translation invariance, we assume that $y=0$ is a non-degenerate critical point outside $\overline{\omega}_{\gamma}^{\perp}$.
Given any $T>0$, our goal is to construct a sequence of solutions $(u_k(t))_{k\geq 1}$ of the magnetic Schr\"odinger equation, such that 
\begin{align}\label{optimal-1}
	\lim_{k\rightarrow\infty}\frac{\int_0^T\int_{\omega}|u_k(t,z)|^2 \d z \d t }{\|u_k(0)\|_{L^2}^2}=0.
\end{align}
The strategy is to construct well-prepared solutions for the model equation
\begin{align}\label{eq:model}
	i\partial_tu - (D_x^2+D_y^2)u +2A_1(y)D_xu+A_2(y)D_yu+D_y(A_2(y)u)-\langle V^{(1)}\rangle (y)u=0,
\end{align}
that are oscillated at the scale $|D_y|\leq h^{-\frac{1}{2}}, |D_x|\sim h^{-1}$ and concentrated near the non-degenerate critical point $y_0=0$ outside $\overline{\omega}_{\gamma^{\perp}}$. Then we take the normal form transform from the opposite direction to obtain the desired solutions of the original equation. Here we have to construct $o(h^2)$-quasimodes instead of $o(h^{\frac{3}{2}})$, comparing to the proof of Theorem \ref{thm: main}. To this end, we have to perform a finer normal form transformation as described below.

Set $g_2(x,y,\xi), G_2=\Op_h(g_2\vartheta(\eta)\eta )$ as in Proposition \ref{prop: 2-average}. Replacing $g_1(x,y), A_2^{(1)}(x,y), \langle A_1\rangle(y)$ there by $0, A_2(x,y),A_1(y)$, the computation in the proof of Proposition \ref{prop: 2-average} gives
\begin{align}\label{conjugate}
	e^{G_2}(P_h^{(1)}+Q_h^{(1)})e^{-G_2}=P_h^{(2)}+R_h^{(2)}+R_h^{(3)},
\end{align} 
where
\begin{align}\label{op:reminder} 
	&R_h^{(2)}:= h^2\Op_h\big(V^{(1)}(x,y)+2\ii g_2(x,y,\xi)\vartheta(\eta)A_1'(y)\xi \big)-h\Op_h(r_{1+2\alpha}(x,y,\xi)\eta^2),\\
	&R_h^{(3)}:=h^2\Op_h(\Theta(x,y,\xi,\eta)\eta)+R_{2+\alpha}^{(2)}+\mathcal{O}_{\mathcal{L}(L^2)}(h^3).
\end{align}
We recall the model operator $P_h^{(2)}$ defined in \eqref{eq:defP2}, the symbol $r_{1+2\alpha}(x,y,\xi)=2\ii \partial_yg_2\vartheta(\eta)$, and 
$\Theta(x,y,\xi,\eta)$ is a smooth symbol with compact support in all variables. The computation \eqref{conjugate} expands one more term in the Taylor expansion compared to the proof of Proposition \ref{prop: 2-average}. and the right hand of \eqref{conjugate} is organized in the way that, the first term $P_h^{(2)}$is the model operator, while the second term $R_h^{(2)}$ will eventually generate an $O(h^2)$ error applying to the well-prepared solutions, thus they have to be averaged one more time. The third term $R_h^{(3)}$ will contribute an acceptable error $o(h^2)$ when applying to the well-prepared solutions.

The point here is that $r_{1+2\alpha}(x,y,\xi)$ and $g_2(x,y,\xi)$ both have zero mean value when integrating over $\mathbb{T}$ with respect to $x$. Therefore, by successive normal form reduction, we will obtain the model equation \eqref{eq:model} after averaging these symbols. 

More precisely, set
$$ G_3=\Op_h(g_3(x,y,\xi)\vartheta(\eta)\eta^2+hg_4(x,y,\xi,\eta)),
$$
Then
\begin{align*}
	e^{G_3}(P_h^{(2)}+R_h^{(2)}+R_h^{(3)})e^{-G_3}=& P_h^{(2)}+R_h^{(2)}+\frac{h}{i}\Op_h(\{g_3\vartheta(\eta)\eta^2+hg_4,\xi^2\})+\widetilde{R}_h^{(3)},
\end{align*}
where here and in the sequel, $\widetilde{R}_h^{(3)}$ denotes operators with similar structure as $R_h^{(3)}$. Set
\begin{align*}
	g_3(x,y,\xi):=&\frac{\psi(\xi)}{2\ii \xi}\partial_x^{-1} (r_{1+2\alpha})(x,y,\xi),\\
	g_4(x,y,\xi,\eta):=&-\frac{\psi(\xi)}{2\ii \xi}\partial_x^{-1}(V^{(1)}-\langle V^{(1)}\rangle)(x,y)-\partial_x^{-1}(g_2)(x,y,\xi)\vartheta(\eta).
\end{align*}
We finally get
$$ e^{G_3}(P_h^{(2)}+R_h^{(2)}+R_h^{(3)})e^{-G_3}=P_h^{(2)}+h^2\langle V^{(1)}\rangle(y)+\widetilde{R}_h^{(3)}.
$$

Hence it suffices to construct well-prepared solutions for the model equation \eqref{eq:model}. 

Set $w_h(t,x,y):=v_h(t,y)e^{\ii kx}$, where $k=\frac{1}{h}\in\mathbb{Z}$. We ask for $w_h$ satisfying the equation
$$ ih^2\partial_tw_h-P_h^{(2)}w_h-h^2\langle V^{(1)}\rangle(y)w_h=0,
$$
which is equivalent to the model equation
\begin{align}\label{model:eq}
 \ii h\partial_tv_h-\frac{1}{h}v_h-(hD_y^2-2A_1(y))v_h+h\langle A_2\rangle(y)D_yv_h+hD_y(\langle A_2\rangle(y)v_h)-h\langle V^{(1)}\rangle(y)v_h=0.
\end{align}
Take $0<b<\pi$ small enough such that $\mathbb{T}_x\times (-b,b)\subset \mathbb{T}^2\setminus \overline{\omega}_{\gamma^{\perp}}$.

We have the following non-observability result of the model equation:
\begin{proposition}\label{non-ob:model}
Let $T>0$, $0<b<\pi$. There exist a sequence of small parameters $h_k\rightarrow 0$ and a sequence of approximate solutions $v_{k}(t,y)=v_{h_k}(t,y)$ to \eqref{model:eq}:
\begin{align}\label{model:eqapprox}
 \ii h_k\partial_tv_k-\frac{1}{h_k}v_k-(h_kD_y^2-2A_1(y))v_k+h_k\langle A_2\rangle(y)D_yv_k+h_kD_y(\langle A_2\rangle(y)v_k)-h_k\langle V^{(1)}\rangle(y)v_k=r_k,
 \end{align}
such that
\begin{align*}
\liminf_{k\rightarrow\infty}\|v_k(0)\|_{L^2(\T)}^2>0, \quad \sup_{t\in[0,T]}\|r_k(t)\|_{L^2(\T)}\leq Ch_k^{\frac{5}{4}}.
\end{align*}
Moreover, for any $m\in\mathbb{N}$ and $\epsilon_0\in(0,\pi)$, 
$$ \sup_{t\in[0,T]}\|(h_k^{\frac{1}{2}}\partial_y)^mv_k(t,y)\|_{L^2(\T)} \leq C_mh_k^{\frac{m}{4}}, \quad \sup_{t\in[0,T]}\|v_k(t)\|_{L^2(\epsilon_0<|y|<\pi)}\leq Ch_k^{100}.
$$
\end{proposition}
We postpone the proof of Proposition \ref{non-ob:model} in the next subsection and proceed on the proof of Theorem \ref{thm: optimal}. 
Indeed, for $h=h_k$, $w_h(t,x,y)=v_h(t,y)e^{\ii kx}$, $r_h(t,x,y)=r_h(t,y)e^{\ii kx}$, we have 
\begin{align*}
	\liminf_{h\rightarrow 0}\|w_h(0)\|_{L^2(\T^2)}^2>0.
\end{align*}
Moreover, for any $m\in\mathbb{N}$ and $\epsilon_0\in(0,\pi)$, 
$$ \sup_{t\in[0,T]}\|(h^{\frac{1}{2}}\partial_y)^mw_h(t)\|_{L^2(\T)} \leq C_mh^{\frac{m}{4}}, \quad \sup_{t\in[0,T]}\|w_h(t)\|_{L^2(\epsilon_0<|y|<\pi)}\leq Ch^{100}.
$$
In particular, when acting on $w_h$ by the operator reminder $\widetilde{R}_h^{(3)}$, similar as $R_h^{(3)}$ defined in \eqref{op:reminder}, we have 
$$ \sup_{t\in[0,T]}\|\widetilde{R}_h^{(3)}w_h \|_{L^2(\T^2)}+\sup_{t\in[0,T]}\|r_h\|_{L^2(\T^2)}=o(h^2).
$$
Set $\widetilde{u}_h=e^{-G_2}\circ e^{-G_3}w_h$, then
\begin{align*}
 e^{G_3}\circ e^{G_2}\circ(\ii h^2\partial_t +h^2H_{\mathbf{A},V})\widetilde{u}_h=&e^{G_3}\circ e^{G_2}\circ(\ii h^2\partial_t+ h^2H_{\mathbf{A},V})e^{-G_2}\circ e^{-G_3}w_h \\
 =&\big(\ii h^2\partial_t-P_h^{(2)}-h^2\langle V^{(1)}\rangle(y)\big)w_h-\widetilde{R}_h^{(3)}w_h\\
 =&-\widetilde{R}_h^{(3)}w_h+r_h.
\end{align*}
Since the closed geodesic $\gamma$ parametrized by $y=0$ stays completely outside the closure of the set $\overline{\omega}_{\gamma^{\perp}}$, there exists $\epsilon_0>0$, small enough such that the closed horizontal strip $S_{\epsilon_0}:=\mathbb{T}_x\times [-\epsilon_0,\epsilon_0]$ does not intersect with $\omega$.  In other words, $\overline{\omega}$ is contained in the complementary of the strip $S_{\epsilon_0}$. Let $\widetilde{\chi}$ be a smooth function, compactly supported in $\mathbb{T}^2\setminus S_{\epsilon_0}$ such that $\widetilde{\chi}>0$ on $\overline{\omega}$. From the concentration property of $w_h$, 
$$ \sup_{t\in[0,T]}\|\widetilde{\chi} w_h(t)\|_{L^2(\T^2)}\leq Ch^{100}.
$$
Hence
\begin{align*}
 \|\widetilde{u}_h\|_{L^2([0,T];L^2(\omega))}\leq &C\|\widetilde{\chi}\widetilde{u}_h\|_{L^2([0,T]\times\T^2)}\\ \leq &C\|e^{-G_3}\circ e^{-G_2}(\widetilde{\chi}w_h)\|_{L^2([0,T]\times\T^2)}+C\|[\widetilde{\chi},e^{-G_3}\circ e^{-G_2}]w_h \|_{L^2([0,T]\times\T^2)}.
\end{align*}
Note that $G_2,G_3$ are bounded $h$-pseudo differential operators, the commutator term is bounded by $Ch\|w_h\|_{L^2([0,T]\times\T^2)}$. Hence we have
\begin{align*}
\liminf_{h\rightarrow 0}\|\widetilde{u}_h(0)\|_{L^2(\T^2)}>0,\quad \limsup_{h\rightarrow 0}\|\widetilde{u}_h\|_{L^2([0,T]\times\omega)}\leq O(h).
\end{align*}
As $\widetilde{u}_h$ is not the exact solution of the original Schr\"odinger equation, to conclude, we assume that $u_h$ is the solution to $(i\partial_t+H_{\mathbf{A},V})u_h=0$ with $u_h(0)=\widetilde{u}_h(0)$. It follows that
$$ \|u_h(t)-\widetilde{u}_h(t)\|_{L^2([0,T]\times\T^2)}\leq Ch^{-2}\|\widetilde{R}_h^{(3)}w_h\|_{L^2([0,T]\times\T^2)}+Ch^{-2}\|r_h\|_{L^2([0,T]\times\T^2)} = o(1).
$$
In particular,
$$ \limsup_{h\rightarrow 0}\|u_h\|_{L^2([0,T]\times\omega)}=o(1),\quad h\rightarrow 0.
$$
This contradicts the observability \eqref{eq:SchrodingerObs}.
\subsection{One-dimensional non-observability result}\label{sec: 1d non-ob}
First we construct suitable quasi-modes for a semiclassical Schr\"odinger operator:
\begin{lemma}\label{lem: quasimode-construction}
	Assume that $A_1,A_2\in C^{4}(\T_y)$ are two real-valued functions and $A_1$ attains its maximum at $y=0$ with $A_1'(0)=0$ and $A_1''(0)<0$. Assume that $0<b<\pi$.
	Then, there exist a complex number $\Lambda_0$, a sequence $\hbar_k\rightarrow0$, and a sequence of smooth functions $v_k\in C^{\infty}(\T)$ such that 
	\begin{equation*}
		\begin{split}
			(-\hbar_k^2\p_y^2-2A_1(y))v_k+&\ii \hbar_k^2 A_2(y)\p_y v_k+\ii \hbar_k^2\p_y(A_2(y) v_k)+\hbar_k^2W(y)v_k\\=&\sqrt{-A_1''(0)}\hbar_kv_k+(-2A_1(0)+\Lambda_0 \hbar_k^2)v_k+\bigO{\hbar_k^{\frac{5}{2}}}{L^2(\T)}.
		\end{split}
	\end{equation*}
	with 
	\begin{equation*}
		\|v_k\|_{L^2(\T)}=1,\; \|v_k\|_{L^2(\T\setminus (-b,b))}=\mathcal{O}(\hbar_k^{\infty}).
	\end{equation*}
Moreover, for any $m,\ell\in\mathbb{N}$,
\begin{align}\label{localization} \|(\partial_y)^my^{\ell}v_k\|_{L^2((-b,b))}+\|y^{\ell}(\partial_y)^mv_k\|_{L^2((-b,b))} =\mathcal{O}(\hbar_k^{-\frac{m}{2}+\frac{\ell}{2}}),\quad \|(\partial_y)^mv_k\|_{L^2(\T)}=\mathcal{O}(\hbar_k^{-\frac{m}{2}}).
\end{align}
\end{lemma}
	Set $\hbar_k=h_k^{\frac{1}{2}}$.  Let
	$$ v_k(t,y):=e^{-\frac{\ii t}{\hbar_k^4} +2\ii \frac{A_1(0)t}{\hbar_k^2}-\ii \sqrt{-A_1''(0)} \frac{t}{\hbar_k}-\ii \Lambda_0 t}v_k(y),
	$$
	where $v_k$ be the sequence of quasi-modes in Lemma \ref{lem: quasimode-construction}.
	Hence $v_k(t,y)$ satisfies the approximate equation \eqref{model:eqapprox}. 
	 Consequently, Proposition \ref{non-ob:model} follows directly from Lemma \ref{lem: quasimode-construction}.

Finally, we prove Lemma \ref{lem: quasimode-construction}.
\begin{proof}
The proof follows from the standard WKB approximation method. For simplicity, we omit the subscription $k$ in $\hbar_k$, and without loss of generality, we also assume that $A_1(0)=0$ and $W(0)=0$ (up to changing the value of $\Lambda_0$).
	Let $$\beta:=\sqrt{-A_1''(0)},\quad  \mathcal{L}_{\hbar}:=-\hbar^2\partial_y^2+\beta^2y^2-\beta\hbar.$$
	Denote $H_j(y)$ be the $j$-th Hermite polynomial and $\phi_{j,\hbar}(y)=c_{j,\hbar}H_j(\beta^{\frac{1}{2}}\hbar^{-\frac{1}{2}}y)e^{-\frac{\beta y^2}{2\hbar}}$ is the normalized eigenfunction of $\mathcal{L}_{\hbar}$ associated to the eigenvalue $2j\beta \hbar$. Denote $E_j:=\mathrm{span}_{\C}(\phi_{j,\hbar})$
	 Note that $\mathrm{Ker}(\mathcal{L}_{\hbar})=E_0$, and $\mathcal{L}_{\hbar}$ is invertible on orthogonal spaces $E_j$ for all $j\geq 1$ such that $\mathcal{L}_{\hbar}^{-1}|_{E_j}=(2j\beta\hbar)^{-1}\mathrm{Id}_{E_j}$ . Moreover, the operator of the multiplication $y$ and the derivation $\partial_y$ send $E_j$ to $E_{j-1}\oplus E_{j+1}$ (with the convention $E_{-1}=\{0\}$) and for all $m\in\mathbb{N}$, 
	 \begin{align}\label{eq:localization}
	 \| y^m|_{E_j}\|_{\mathcal{O}(L^2)}\leq C_{m,j}\hbar^{\frac{m}{2}}, \quad  
	 \| \partial_y^{m}|_{E_j}\|_{\mathcal{O}(L^2)}\leq C_{m,j}\hbar^{-\frac{m}{2}}.
	 \end{align}
	 We would like to approximate the 
	 the following equation
	\begin{align}\label{quasimode:vh-1}
		\mathcal{L}_{\hbar}v_{\hbar}=R_1(y)v_{\hbar}
		-\ii \hbar^2A_2(y)\p_yv_{\hbar} -\ii \hbar^2\partial_y(A_2(y)v_{\hbar})-\hbar^2W(y)v_{\hbar}
	\end{align}
with an acceptable error $\mathcal{O}(\hbar^{\frac{5}{2}})$. 
Taylor expanding $A_1$ and $A_2$, we have
\begin{align*}
	&A_1(y)=-\frac{\beta^2}{2}y^2+R_1(y), \text{ with }R_1(y)=\sum_{j=3}^4r^{(1)}_jy^j+\BigO{|y|^{5}}=\BigO{|y|^3},\\
	&A_2(y)=\sum_{j=0}^4 r^{(2)}_jy^j+\BigO{|y|^{5}},\quad W(y)=\mathcal{O}(|y|).
\end{align*}
We will first construct the WKB approximation on $\R$ and later cut it on the circle $\T$. Set
$$ v_{\hbar}=v_{0,\hbar}+v_{1,\hbar}+v_{2,\hbar},
$$
where $v_{0,\hbar}=\phi_{0,\hbar}$ and $v_{1,\hbar},v_{2,\hbar}\in\oplus_{j=0}^6E_j$ to be determined. 
Plugin \eqref{quasimode:vh-1} with the Taylor expansion, the right hand side equals to
\begin{align}\label{rhs-1}
\sum_{j=3}^4r_{j}^{(1)}y^jv_{\hbar}-2\ii \hbar^2 \sum_{j =0}^1r_{j}^{(2)} y^j\partial_yv_{\hbar}-\ii \hbar^2 r_1^{(2)} v_{\hbar}+\mathcal{O}_{L^2}(\hbar^{\frac{5}{2}})
\end{align}
The leading order $\mathcal{O}(\hbar^{\frac{3}{2}})$ or \eqref{rhs-1} equals to
\begin{align*}
r_3^{(1)}y^3\phi_{0,\hbar}-2\ii r_0^{(2)}\hbar^2 \partial_y\phi_{0,\hbar}\in E_1\oplus E_2\oplus  E_3=\mathcal{O}_{L^2}(\hbar^{\frac{3}{2}}).
\end{align*}
There exists complex coefficients $\beta_{j}^{(1)},j=1,2,3$, with $v_{1,\hbar}:=\sum_{j=1}^3\hbar^{\frac{1}{2}}\beta^{(1)}_j\phi_{j,\hbar}$, such that 
$$ \mathcal{L}_{\hbar}v_{1,\hbar}=r_3^{(1)}y^3\phi_{0,h}-2\ii r_0^{(2)}\hbar^2\partial_y\phi_{0,\hbar}.
$$
The next order $\mathcal{O}(\hbar^2)$ in \eqref{rhs-1} now becomes
$$ -\ii \hbar^2 r_1^{(2)}v_{0,\hbar}
+(r_3^{(1)}y^3-2\ii r_0^{(2)}\hbar^2\partial_y)v_{1,\hbar}+r_4^{(1)}y^4v_{0,h}.
$$
Since $(r_3^{(1)}y^3-2\ii r_0^{(2)}\hbar^2\partial_y)v_{1,\hbar}+r_4^{(1)}y^4v_{0,\hbar}\in \oplus_{j=0}^6E_j$, there exist complex coefficients $c_0$ and $\beta_j^{(2)}, j=1,\cdots,6$, such that with $v_{2,\hbar}:=\sum_{j=1}^6 \hbar \beta_j^{(2)}\phi_{j,\hbar}$,
$$ \mathcal{L}_{\hbar}v_{2,\hbar}=(r_3^{(1)}y^3-2\ii r_0^{(2)}\hbar^2\partial_y)v_{1,\hbar}+r_4^{(1)}y^4v_{0,\hbar}-c_0\hbar^2v_{0,\hbar}.
$$
Note that
\begin{align*}
\hbar^2c_0:=&\int_{-\infty}^{\infty}\big((r_3^{(1)}y^3-2\ii r_0^{(2)}\hbar^2\partial_y)v_{1,\hbar}(y)+r_4^{(1)}y^4v_{0,\hbar}(y)\big)\cdot  v_{0,\hbar}(y)\; \d y\\
=&\hbar^2\int_{-\infty}^{\infty}\phi_{0,\hbar}(y)\cdot \Big(r_3^{(1)}(\hbar^{-\frac{1}{2}}y)^3\sum_{j=1}^3\beta_j\phi_{j,\hbar}(y)-2\ii r_0^{(2)}\sum_{j=1}^3 \hbar^{\frac{1}{2}}\partial_y \phi_{j,\hbar}(y) 
 \Big)\; \d y\\
 +&\hbar^2 \int_{-\infty}^{\infty} (\phi_{0,\hbar}(y))^2\cdot (\hbar^{-\frac{1}{2}}y)^4 \; \d y.
\end{align*}
In particular, $c_0$ is a complex constant independent of $\hbar$.
Hence for $v_{\hbar}=\sum_{j=0}^2v_{j,\hbar}$,
\begin{align*}
 \mathcal{L}_{\hbar}v_{\hbar}=&(r_3^{(1)}y^3-2\ii r_0^{(2)}\hbar^2\partial_y)v_{0,\hbar}\\
 +&r_4^{(1)}y^4v_{0,\hbar}+(r_3^{(1)}y^3-2\ii r_0^{(2)}\hbar^2\partial_y)v_{1,\hbar}+\hbar^2(-c_0-\ii r_1^{(2)})v_{0,\hbar}\\
 =& \sum_{j=3}^4 r_j^{(1)}y^jv_{\hbar}-2\ii \hbar^2\sum_{j=0}^1r_j^{(2)}y^j\partial_yv_{\hbar}-(\ii r_1^{(2)}+c_0)\hbar^2v_{\hbar}+\mathcal{O}_{L^2}(\hbar^{\frac{5}{2}}).
\end{align*}
Set $\Lambda_0=-\ii r_1^{(2)}-c_0$, we obtain the desired quasi-mode equation on $\R$. 
To complete the proof, we introduce a cut-off function $\varphi\in C_c^{\infty}((-\pi,\pi))$ such that $\varphi(\xi)=1$ for $|\xi|<b$. As $v_{\hbar}$ constructed above decays of order $\mathcal{O}(\hbar^{\infty})$ on supp$(\varphi')$, the function $\varphi(y)v_{\hbar}(y)$ satisfies the desired quasi-mode equation. Moreover, the localization property \eqref{localization} follows from the properties of \eqref{eq:localization}. 
The proof of Lemma \ref{lem: quasimode-construction}
is now complete.
\end{proof}

\appendix
\section{Semiclassical measures for electromagnetic Schr\"odinger equations}\label{sec: semiclassical measure}
\subsection{Semiclassical Weyl quantization on \texorpdfstring{$\T^2$}{T2}}
\label{sec:semiclassicaltorus}
In this part, we recall the semiclassical Weyl quantization on $\T^2$. A symbol $a=a(z, \zeta)\in C^{\infty}(T^*\T^2,\C)$ is periodic with respect to $z$-variable. Following the convention in \cite[Chapter 5]{Zwo12}, the quantization is explicitly given by

\begin{equation}\label{semi-quantiWeylcl}
\Op_h(a(z,\zeta))g:=\sum_{k\in\Z^2}A^{\mathrm{w}}_kg(z),\;\;A^{\mathrm{w}}_kg(z):=
\int_{\R^2\times\T^2}e^{\frac{\ii(z-z'+2k\pi)\cdot\zeta}{h}}a(\frac{z+z'}{2},\zeta)
g(z')\frac{\d z'\d\zeta}{(2\pi h)^2}.
\end{equation}
Then $A^{\mathrm{w}}_k=\mathbf{1}_{\T^2}\tau_{-2k\pi}\Op_h(a)$, where $\tau_{z_0}g(z):=g(z-z_0)$. 
The following theorem, borrowed from \cite{Zwo12}, states the basic property of the semiclassical Weyl quantization on the torus.
\begin{theorem}\label{thm:sc_results_torus}
Suppose that $a, b \in C_b^{\infty}(T^* \T^2)$.
\begin{enumerate}
\item The operator $\Op_h(a)$ can be extended as an operator mapping $L^2(\T^2)$ to $L^2(\T^2)$. Moreover, we have
    \begin{equation*}
            \|\Op_h(a)\|_{\mathcal L (L^2(\mathbb T^d))} \leq C \|a\|_{L^{\infty}(T^* \T^d)} + \mathcal O(h^{\frac{1}{2}}).
    \end{equation*}
\item If $a \geq 0$, then there exist constants $C \geq 0$ and $h_0 >0$ such that 
\begin{equation*}
    \langle \Op_h(a) u, u \rangle \geq - C h \|u\|_{L^2(\T^2)}^2\qquad \forall 0 < h < h_0,\ \forall u \in L^2(\T^2).
\end{equation*}
\item There exists $c \in \mathcal{S}(\R^{4})$ such that
\begin{equation*}
    \Op_h(a) \Op_h(b) = \Op_h(c),
\end{equation*}
where
\begin{equation*}
c(z,\zeta) = e^{i \frac{h}{2} \sigma(D_z, D_{\zeta}, D_{\Tilde{z}}, D_{\Tilde{\zeta}})} (a(z, \zeta) b(\Tilde{z},\Tilde{\zeta}))_{|(\Tilde{z},\Tilde{\zeta})=(z,\zeta)},
\end{equation*}
where the symplectic product is defined by 
$$\sigma(u,v) = \langle \zeta, \Tilde{z} \rangle - \langle z, \Tilde{\zeta}\rangle,\text{ for }u=(z,\zeta),v=(\Tilde{z},\Tilde{\zeta}).$$
Moreover, we have
\begin{equation*}
    c= ab + \frac{h}{2i} \{a,b\} + \mathcal{O}_{\mathcal S}(h^2),
\end{equation*}
and
\begin{equation*}
    [ \Op_h(a),  \Op_h(b)] =  \frac h\ii \Op_h( \{a, b\}) + \mathcal{O}_{\mathcal S}(h^3).
\end{equation*}
\item Given a polynomial function $P$ of degree $2$, we have
\begin{equation}\label{eq: commutators_laplacian}
    [\Op_h(a), \Op_h{P(\zeta)}] = -\frac hi \Op_h(\nabla_{\zeta}P(\zeta) \cdot \nabla_z a)
\end{equation}
\end{enumerate} 
\end{theorem}
Finally, we recall the definition of the semiclassical wavefront set, following the convention in \cite[Chapter 8]{Zwo12}. We say that $u=\{u_h\}_{h\in(0,h_0)}\subset\mathscr{S}'$ is a $h-$temptered family of distributions if there exists a constant $N_0\in\N$ such that $\|u_h\|_{L^2(\T^d)}=\BigO{h^{-N_0}}$.
\begin{definition}[Wavefront set]
We define the wavefront sets $\WF(u)$ and $\WFm{m}(u)$ as follows:
\begin{enumerate}
    \item The semiclassical wavefront set $\WF(u)$ associated with a $h-$temptered family $u=\{u_h\}_{h\in(0,h_0)}$ is the complement of the set of points $(z_0,\zeta_0)\in T^*\T^d$ for which there exists a symbol $a\in\mathcal{S}^0$ such that $a(z_0,\zeta_0)\neq0$ and $\Op_h(a)u_h=\bigO{h^N}{L^2}$ for all $N\in\N$.
    \item The semiclassical wavefront set of order $m$, denoted by $\WFm{m}(u)$, is the complement of the set of points $(z_0,\zeta_0)\in T^*\T^d$ for which there exists a symbol $a\in\mathcal{S}^0$ such that $a(z_0,\zeta_0)\neq0$ and $\Op_h(a)u_h=\bigO{h^m}{L^2}$.
\end{enumerate}
\end{definition}

\subsection{Semiclassical toolbox of exponentials of linear bounded operators}

Let $\h$ be an Hilbert space. For $A$, $B$ in $\mathcal{L}(\h)$, we first introduce the notations
\begin{equation}
    \label{eq:defadj}
  \ad_{A}^0(B)=B,\qquad    \ad_{A}^k(B)= [A, \ad_{A}^{k-1}(B)]\qquad \forall k \geq 1.
\end{equation}
For every $k \geq 0$, $\ad_{A}^k(B) \in \mathcal{L}(\h)$. We also define
\begin{equation}
    e^{A} = \sum_{k=0}^{\infty}\frac{A^k}{k!}.
\end{equation}
We have that $e^{A} \in \mathcal{L}(\h)$ is an invertible operator whose inverse is $e^{-A}$.

We have the following basic lemma.
\begin{lemma}{\cite[Lemma 5.1]{Sun-dampedwave}}\label{lem: exponential of bounded op}
Let $B \in \mathcal{L}(\h)$, $\{G_h\}_{h\in(0,1)}$ be a family of $h-$dependent, uniformly bounded operators on $\h$. The following Taylor expansion holds, for every $N\in\N^*$,
\[
e^{sG_h}Be^{-sG_h}=\sum_{k=0}^{N-1}\frac{s^k}{k!}\ad^k_{G_h}(B)+\frac{1}{(N-1)!}\int_0^1(1-s)^{N-1}e^{sG_h}\ad^N_{G_h}(B)e^{-sG_h}\d s,\qquad s \in \R.
\]
\end{lemma}
The next result tells us that the wave front set is conserved under the action of $e^{G_h}$. Given the definition of the exponential of bounded operators, we shall prove that the transform $e^{G_h}v_h$ does not change the wavefront $\WFm{m}(v_h)$. 
\begin{lemma}\label{lem: wave front invariance}
Let $\{G_h\}_{h\in(0,1)}$ be a family of $h-$dependent, uniformly bounded operators on $L^2(\T^2)$,  $(v_h)_{h\in(0,1)}$ be a family of tempered distributions on $\T^2$ and $\widetilde{v}_h = e^{G_h} v_h$. Then $\WFm{m}(\widetilde{v}_h)=\WFm{m}(v_h)$ for every $ m \geq 1$.
\end{lemma}
\begin{proof}
The proof is similar to \cite[Proof of Lemma 5.3]{Sun-dampedwave}. 
It only suffices to show that $\WFm{m}(\widetilde{v}_h)\subset\WFm{m}(v_h)$ as the reverse inclusion follows by using that $v_h = e^{-G_h} \widetilde{v}_h$.

Let $\kappa(z,\zeta)$ be a symbol supported on a compact set of $T^*\T^2\setminus\WFm{m}(v_h)$ and $K_h=\Op_h(\kappa)$, by definition, it suffices to show 
\[
K_h\widetilde{v}_h=\bigO{h^m}{L^2},\mbox{ i.e., }K_he^{G_h}v_h=\bigO{h^m}{L^2}.
\]
From the definition of $\WFm{m}(v_h)$, for any semiclassical operator $L_h$ with its principal symbol supported away from $\WFm{m}(v_h)$, we have 
\[
L_hv_h=\bigO{h^m}{L^2}.
\]
Using the Taylor expansion for the conjugate operator, we write
\begin{align*}
K_he^{G_h}v_h&=e^{G_h}(e^{-G_h}K_he^{G_h})v_h\\
&=e^{G_h}\sum_{k=0}^{m-1}\frac{s^k}{k!}\ad^k_{G_h}(K_h)v_h+\frac{e^{G_h}}{(m-1)!}\int_0^1(1-s)^{m-1}e^{-sG_h}\ad^m_{G_h}(K_h)e^{sG_h}v_h\d s.
\end{align*}
Using the standard symbolic calculus, $\ad^m_{G_h}(K_h)=\bigO{h^m}{\mathcal{L}(L^2)}$, which implies that
\[
\frac{e^{G_h}}{(m-1)!}\int_0^1(1-s)^{m-1}e^{-sG_h}\ad^m_{G_h}(K_h)e^{sG_h}v_h\d s=\bigO{h^m}{L^2}.
\]
And for $k\leq m-1$, $\ad^k_{G_h}(K_h)$ is a semiclassical operator with its principal symbol supported away from $\WFm{m}(v_h)$ due to the definition of $K_h$. Hence, $\ad^k_{G_h}(K_h)v_h=\bigO{h^m}{L^2}$. Consequently, we obtain $K_he^{G_h}v_h=\bigO{h^m}{L^2}$. We conclude $\WFm{m}(\widetilde{v}_h)\subset\WFm{m}(v_h)$.
\end{proof}


\end{document}